\documentclass[12pt,a4paper, twoside]{amsart}
\usepackage{amsfonts}
\usepackage{amssymb}
\usepackage[english]{babel}
\usepackage{mathtools}
\usepackage{amscd}
\usepackage{amsthm}
\usepackage{thmtools}
\usepackage{nicefrac}
\usepackage{amscd}
\usepackage{tikz-cd}
\usepackage{tikz}
\usepackage{enumerate}
\usepackage{enumitem}
\usetikzlibrary{matrix,arrows,decorations.pathmorphing}
\usepackage{setspace}
\usepackage{hyperref}

\newcommand{\partialt}{\frac{\partial}{\partial t}}
\newcommand{\partialb}{\bar \partial}

\makeatletter
\def\paragraph{\@startsection{paragraph}{4}%
	\z@\z@{-\fontdimen2\font}%
	{\normalfont\bfseries}}
\makeatother

\makeatletter
\newcommand{\definetitlefootnote}[1]{%
	\newcommand\addtitlefootnote{%
		\makebox[0pt][l]{$^{*}$}%
		\footnote{\protect\@titlefootnotetext}
	}%
	\newcommand\@titlefootnotetext{\spaceskip=\z@skip $^{*}$#1}%
}
\makeatother

\begin{document}
	
			\newtheorem{lemma}{Lemma}[section]
	\newtheorem{prop}[lemma]{Proposition}
	
	\newtheorem{theorem}[lemma]{Theorem}
	\newtheorem{corollary}[lemma]{Corollary}
	\newtheorem{theoremintro}{Theorem}
	\theoremstyle{definition}
	\newtheorem{defi}[lemma]{Definition}
	\newtheorem{example}[lemma]{Example}
	\newtheorem*{claim*}{Claim}
	\newtheorem*{notation}{Notation}
	\newtheorem{assumption}[lemma]{Assumption}

	\newtheorem{remark}[lemma]{Remark}
	
	%\pagenumbering{gobble}
	
	\title[ACyl K\"ahler-Ricci solitons]{Asymptotically cylindrical steady K\"ahler-Ricci solitons}
	\author{Johannes Sch\"afer}

	\begin{abstract}

  Let $D$ be a compact K\"ahler manifold with trivial canonical bundle and $\Gamma$ be a finite cyclical group of order $m$ acting on $\mathbb{C} \times D$ by biholomorphisms, where the action on the first factor is generated by rotation of angle $2\pi /m$. Furthermore, suppose that $\Omega_D$ is a trivialisation of the canonical bundle such that $\Gamma$ preserves the holomorphic form $dz \wedge \Omega_D$ on $\mathbb C \times D$, with $z$ denoting the coordinate on $\mathbb{C}$.

  The main result of this article is the construction of new examples of gradient steady K\"ahler-Ricci solitons on certain crepant resolutions of the orbifolds $\left(  \mathbb{C}\times D \right) / \Gamma$. These new solitons converge exponentially  to a Ricci-flat cylinder $\mathbb{R} \times(\mathbb{S}^1 \times D) / \Gamma$.
	\end{abstract}

	\maketitle

	\section{Introduction}

A \textit{steady Ricci soliton} is a Riemannian manifold $(M,g)$ together with a vector field $X$ such that
\begin{align}\label{intro real soliton equation}
	\operatorname{Ric}(g) = \frac{1}{2} \mathcal{L}_X g,
\end{align}
where $\operatorname{Ric}(g)$ denotes the Ricci tensor of $g$ and $\mathcal{L}_X$ is the Lie derivative in direction of $X$. The soliton $(M,g,X)$ is called \textit{gradient} if $X$ is the gradient field of some function on $M$. 

If $(M,g)$ is K\"ahler and the vector field $X$  real holomorphic, equation (\ref{intro real soliton equation}) is equivalent to 
\begin{align}\label{intro: equation steady KR soliton}
	\operatorname{Ric}(\omega) = \frac{1}{2} \mathcal{L}_X \omega,
\end{align}
where $\omega$ is the K\"ahler form of $g$ and $\operatorname{Ric}(\omega)$ the corresponding Ricci form.
A K\"ahler manifold $(M,g)$ which admits a real holomorphic vector field $X$ satisfying (\ref{intro: equation steady KR soliton}) is called a \textit{steady K\"ahler-Ricci soliton}.

 Steady solitons may be viewed as natural generalisations of Einstein manifolds, which correspond to the case $X\equiv0$. Non-Einstein steady solitons, however, must be non-compact (\cite{ivey1993ricci}). 
 
To each steady Ricci soliton $(M,g,X)$ one can associate a self-similar Ricci-flow by rescaling and pulling back $g$ along the flow of $X$. Thus, steady solitons may be possible candidates for singularity models for Ricci-flow. They  are also  important in the context of so-called Type II singularities, i.e. when a  Ricci-flow exists up to the finite time $T>0$, and the curvature  blows up   faster than $(T-t)^{-1}$. For recent progress in the study of singularities as well as steady Ricci solitons, we refer the reader to \cite{bamler2021four}, \cite{chow2020four}, \cite{bamler2020entropy}, \cite{chow2020curvature}, \cite{deng2020classification}, and the references therein. 

This article focuses on the case of steady K\"ahler-Ricci solitons, and our main result is  the existence of a new class of such solitons. 
In contrast to general Ricci-solitons, it suffices to solve a single equation of top-dimensional differential forms in order to construct a gradient K\"ahler-Ricci soliton. If $M$ is a 
 complex manifold  of (complex) dimension $n$, together with a nowhere-vanishing holomorphic $(n,0)$-form $\Omega$, and  a K\"ahler metric $g$ whose K\"ahler form $\omega$ satisfies
\begin{align}\label{in intro: equation of volume forms}
	\omega ^n = e^{-f} i^{n^2} \Omega \wedge \overline{\Omega}
\end{align}
for some function $f:M \to \mathbb{R}$, then $(M,g,\nabla^g f)$ defines a gradient steady K\"ahler-Ricci soliton. In fact, if $M$ is simply-connected, then one can always associate such a form $\Omega$ to a gradient steady K\"ahler-Ricci soliton, compare \cite{bryant2008gradient}[Theorem 1].

However, given $M$ and  a nowhere-vanishing holomorphic $n$-form $\Omega$ on $M$ it is not known if $M$ admits a steady soliton, i.e. there is no general existence theory for steady K\"ahler-Ricci solitons as is the case for compact Ricci-flat K\"ahler manifolds due to Yau \cite{yau1978ricci}. 

All previously known examples of steady K\"ahler-Ricci solitons may be divided into two classes. The first group consists of solitons constructed by reducing (\ref{intro: equation steady KR soliton}) to an ODE, for instance by Hamilton \cite{hamilton1988ricci}, Cao \cite{Caosoliton}, Dancer and Wang \cite{dancer2011ricci}, Yang \cite{yang2012characterization} and the author \cite{schafer2020existence}. Most notably, we mention Hamilton's cigar on $\mathbb{C}$ (\cite{hamilton1988ricci}) and Cao's soliton on $\mathbb{C}^n$ for $n\geq 2$ (\cite{Caosoliton}). The cigar is asymptotic to the cylinder $dt^2+ d\theta ^2$ on the product $\mathbb{R} \times \mathbb{S}^1 \cong \mathbb C^*$, whereas Cao's soliton has a more complicated asymptotic behavior. (It is a so-called cigar-paraboloid whose precise asymptotics are explained in \cite{conlon2020steady}[Section 3].)

The second group of examples are constructed by PDE methods (\cite{Biquard17}, \cite{conlon2020steady}). Here, the underlying complex manifolds are equivariant, crepant resolutions of certain orbifolds $\mathbb{C}^n / G$ (\cite{Biquard17}) and of more general Calabi-Yau cones (\cite{conlon2020steady}). In both cases, the solitons have an asymptotic behavior similar to Cao's soliton.

In this article, we build on ideas developed in \cite{conlon2020steady} and find new examples of steady K\"ahler-Ricci solitons which are asymptotic to a product $\mathbb{C}\times D$ of Hamilton's cigar and a compact Ricci-flat K\"ahler manifold $D$. (Note that this product is also a steady K\"ahler-Ricci soliton.) These new examples  exist on resolutions $\pi : M \to \left( \mathbb{C} \times D  \right) / \Gamma$ of certain orbifolds $\left( \mathbb{C} \times D  \right) / \Gamma$. Before introducing the precise conditions  on $D,\Gamma$ and $M$,  consider the following example.
\begin{example} \label{in intro: example}
	Let $D= \mathbb{T}$ be the (real) 2-torus and let $\Gamma=\{\pm \operatorname{Id} \} $. Then $\left(\mathbb{C} \times \mathbb{T}  \right)/\Gamma$ has precisely four singular points, each isomorphic to a neighborhood of the origin in $\mathbb{C}^2/ \{ \pm \operatorname{Id}\}$. Thus, we may blow-up each of these singular points to obtain a resolution $\pi :M \to \left(\mathbb{C} \times \mathbb{T}  \right)/\Gamma$.   (Note that previously,  certain Calabi-Yau metrics, so-called ALG gravitational instantons, were constructed on this resolution, see \cite{biquard2011kummer}.)
\end{example}

This resolution $\pi : M \to \left( \mathbb{C} \times \mathbb{T}  \right) / \Gamma$ satisfies three essential properties. First, the resolution is crepant, i.e. the  holomorphic (2,0)-form $\Omega$ on $\left(  \mathbb{C}^* \times \mathbb{T}  \right)/ \Gamma$, which lifts to the canonical form  $dz_1\wedge dz_2$ on $\mathbb{C}^2 $, extends to a nowhere-vanishing form on the entire resolution $M$.   

Second, the $\mathbb{C}^*$-action on $\left(  \mathbb{C}^* \times \mathbb{T}  \right)/ \Gamma$ given by 
	\begin{align*}
	\lambda * (z,w)= (\lambda z, w), \;\;\; \lambda \in \mathbb{C}^*,
\end{align*}
extends $\pi$-equivariantly to a holomorphic action on $M$, since the resolution is toric. In particular, the infinitesimal generator $z_1 \frac{\partial }{\partial z_1}$ on $(\mathbb{C}^* \times \mathbb{T})/ \Gamma$ extends to a holomorphic vector field $Z$ on $M$. 

And third, $M$ admits a natural complex compactification $\overline M$ obtained by adding the  divisor $\overline{\mathbb{T}}:= \mathbb{T}/ \{\pm \operatorname{Id} \} $ `at infinity', i.e. we compactify $\mathbb{C}$ by the Riemann sphere $\mathbb{C} \cup \{\infty\}$ and let $\overline{M}= M \cup \left( \{ \infty\} \times \overline{ \mathbb{T}}    \right) $.
Given a K\"ahler class $\kappa_{\overline{M}} \in H^2(\overline{M},\mathbb{R})$ on $\overline{M}$, it is possible to  construct a new K\"ahler form on $M$ in the class $\kappa_{\overline{M}}|_M \in H^2(M,\mathbb{R})$ that is asymptotic to the cylinder 
\begin{align}\label{in intro: the cylinder}
	|z_1|^{-2} \frac{i}{2} dz_1\wedge d \bar z_1 + \frac{i}{2} dz_2\wedge d \bar z_2.
\end{align}
(This construction follows by adapting ideas from the case of asymptotically cylindrical Calabi-Yau manifolds \cite{haskins2015asymptotically}.) 

Thus, one may ask if there exists a steady K\"ahler-Ricci soliton on $M$ which is asymptotic to the cylinder (\ref{in intro: the cylinder}), whose K\"ahler form is contained in the class $\kappa_{\overline{M}}|_M$ and whose  soliton vector field  equals  the real part of $Z$.
This is indeed a non-trivial question, because $M$ is \textit{not}  a product, but a resolution of the orbifold $\left(\mathbb{C} \times \mathbb{T}  \right) / \Gamma$.

 Our main result (Theorem \ref{geometric existence theorem in introduction}), however, implies that  $M$ \textit{does} admits such solitons. In fact, Theorem \ref{geometric existence theorem in introduction} proves the existence of steady K\"ahler-Ricci solitons for a more general setup:

\begin{theorem} \label{geometric existence theorem in introduction}
	Let $D^{n-1}$ be a compact K\"ahler manifold with nowhere-vanishing holomorphic $(n-1,0)$-form $\Omega_D$. Suppose $\gamma: D \to D$ is a complex automorphism of order $m>1$ such that 
	\begin{align*} 
		\gamma^* \Omega_D = e^{-\frac{2\pi i}{m}} \Omega_D,
	\end{align*} 
	and consider the orbifold $(\mathbb{C} \times D ) / \langle \gamma \rangle$, where $\gamma$ acts on the product via 
	\begin{align*}
		\gamma(z,w)= \left( e^{\frac{2\pi i }{m}}z ,\gamma(w)  \right).
	\end{align*} 
	Let $\pi : M \to (\mathbb{C} \times D ) / \langle \gamma \rangle$ be a crepant resolution   such that the $\mathbb{C}^*$-action on $(\mathbb{C} \times D ) / \langle \gamma \rangle$ given by 
	\begin{align*}
		\lambda * (z,w)= (\lambda z, w), \;\;\; \lambda \in \mathbb{C}^*,
	\end{align*}
	extends $\pi$-equivariantly to a holomorphic action of  $\mathbb{C}^*$ on $M$. 
	
	Let $\overline{M}= M \cup \overline D$ be the complex compactification of $M$ by adding the orbifold divisor $\overline{D}:= D / \langle \gamma \rangle$ at infinity. 
	Then for every orbifold K\"ahler class $\kappa_{\overline{M}}$ on $\overline M$, there exists a steady K\"ahler-Ricci soliton on $M$ whose K\"ahler form is contained in the class $\kappa_{\overline{M}}|_M \in H^2(M,\mathbb R)$. 
\end{theorem}

As in Example \ref{in intro: example}, $M$ admits a nowhere-vanishing holomorphic $(n,0)$-form because the resolution is crepant, and the infinitesimal generator of the $\mathbb{C}^*$-action on $M$ provides a candidate for the soliton vector field. Also, the K\"ahler class is determined by the compactification $\overline{M}$ and the resulting K\"ahler-Ricci soliton is asymptotic to the cylinder $dt^2 + d \theta ^2 +g_D $ on the product $ \left(\mathbb{C}^* \times D \right) / \langle \gamma \rangle \cong \mathbb{R} \times \left(\mathbb{S}^1 \times D  \right) /\langle \gamma \rangle $ for some Ricci-flat K\"ahler metric $g_D$ on $D$.

The new examples of steady K\"ahler-Ricci solitons provided by Theorem \ref{geometric existence theorem in introduction} are geometrically different from all previously found examples in complex dimension $n\geq2$. For instance, their volume grows linearly since they are asymptotically cylindrical, while the examples modelled on Cao's soliton in complex dimension $n$  have volume growth equal to $n$, compare \cite{Caosoliton}, \cite{Biquard17} and \cite{conlon2020steady}.

Interestingly, our examples  also seem to be the only (non-Einstein) steady K\"ahler-Ricci solitons   whose asymptotic model is \textit{Ricci-flat}. This contrasts with the fact that Cao's soliton has \textit{positive} Ricci curvature (\cite{Caosoliton}[Lemma 2.2]). 
Moreover, our new examples are $\kappa$-noncollapsed, whereas Cao's soliton and the ones constructed by Conlon-Deruelle are collapsed (compare \cite{deng2018asymptotic}[Appendix]).

The strategy for proving Theorem \ref{geometric existence theorem in introduction} is analogue to the proof of \cite{conlon2020steady}[Theorem A]. We adapt Conlon and Deruelle's ideas to our setting and reduce (\ref{intro: equation steady KR soliton}) to a complex Monge-Amp\`ere equation, whose solution exists by the following result, which is similar to \cite{conlon2020steady}[Theorem 7.1]

\begin{theorem} \label{analytic existence theorem in intro}
	Let $(M,g,J)$ be an asymptotically cylindrical K\"ahler manifold of complex dimension $n$ with K\"ahler form $\omega$.  Suppose that  $M$ admits a real holomorphic vector field $X$ such that 
	\begin{align*}
		X = 2 \Phi_* \frac{\partial}{\partial t}
	\end{align*}
	outside some compact domain, where  $\Phi$ denotes the diffeomorphism onto the cylindrical end of $(M,g)$ and $t$ is the radial parameter on this end.
	 Moreover, assume that $JX$ is Killing for $g$. 
	 
	 If  $1 < \varepsilon<2$  and $F\in C^{\infty}_{\varepsilon}(M)$  is JX-invariant, then there exists a unique, $JX$-invariant $\varphi \in C^{\infty}_{\varepsilon}(M)$ such that $\omega + i \partial \partialb \varphi >0$ and 
	\begin{align*}
		\left( \omega + i \partial \partialb \varphi \right) ^n  = e^{F - \frac{X}{2} (\varphi)} \omega^n .
	\end{align*} 
\end{theorem}
Note that in this theorem, we do allow more general manifolds than those appearing in Theorem \ref{geometric existence theorem in introduction}. This is because the proof of Theorem \ref{analytic existence theorem in intro} essentially only requires that we have a K\"ahler manifold  $(M,g,J)$, asymptotic to a cylinder (in the sense of Definition \ref{definition ACyl manifold} below) and satisfying two further assumptions: Firstly, we need  the radial vector field on the cylinder to be extended to a real holomorphic vector field on $(M,J)$ and secondly, $JX$ must be an infinitesimal isometry of $g$. We will see in Proposition \ref{proposition when acyl metric is hamiltonian} below that this ensures $X= \nabla^g f$ for some function $f$ with understood asymptotical behavior. 

The spaces $C^{\infty}_\varepsilon (M)$ in Theorem \ref{analytic existence theorem in intro} contain all smooth functions on $M$ whose covariant derivatives (with respect to $g$) decay at least like $e^{-\varepsilon t}$ with $t$ denoting the cylindrical parameter of $(M,g)$ (compare Definitions \ref{definition ACyl manifold} and \ref{definition weighted function spaces}). These function spaces are well-adapted to the cylindrical geometry and have previously been used in the construction of asymptotically cylindrical Calabi-Yau manifolds \cite{haskins2015asymptotically}.

Following the proof of \cite{conlon2020steady}[Theorem 7.1], we also implement a continuity method to conclude Theorem \ref{analytic existence theorem in intro}. To this end, we need to show two things. First, that the linearisation of the Monge-Amp\`ere operator is an isomorphism, which can be deduced from standard results on asymptotically translation invariant differential operators. Second, and most importantly, we have to derive a priori-estimates along the continuity path, where the $C^0$-estimate is the key part of the proof. To obtain this estimate, we adapt the $C^0$-estimate of Conlon and Deruelle (\cite{conlon2020steady}[Section 7.1]) to our cylindrical setup. 
These authors first assume that the right-hand side $F$ is \textit{compactly supported} to obtain the $C^0$-estimate (\cite{conlon2020steady}[Theorem 7.1]) and in a second step, they explain how to solve the Monge-Amp\`ere equation for  \textit{decaying} $F$ (\cite{conlon2020steady}[Theorem 9.2]). 
 We, however, present a modification of their argument, which allows us to achieve the $C^0$-estimate \textit{directly} for $F$ decaying exponentially in Theorem \ref{analytic existence theorem in intro}.

This article is structured as follows. In Section \ref{section linear analysis on acyl manifolds}, we recall the notion of asymptotically cylindrical manifolds and the theory of linear asymptotically translation-invariant operators on such manifolds. This is later applied to the linearisation of the Monge-Amp\`ere operator. 

The basics of steady K\"ahler-Ricci solitons are covered in Section \ref{section preliminaries on KRS}. We recall the underlying Monge-Amp\`ere equation and also discuss when a soliton is gradient. Most notably, we show at the end of this section that, under the assumptions of Theorem \ref{analytic existence theorem in intro}, $X$ must be a gradient field.

In Section \ref{section existence theorem}, we reduce Theorem \ref{geometric existence theorem in introduction} to Theorem \ref{analytic existence theorem in intro}. We discuss the existence of cylindrical K\"ahler metrics on manifolds as in Theorem \ref{geometric existence theorem in introduction} in Section \ref{subsection constructing background metric} and also explain which K\"ahler classes do indeed admit such metrics. Theorem \ref{geometric existence theorem in introduction} is then proven in Section \ref{subsection: proof of geometric existence theorem}, before we provide further examples in Section \ref{subsection examples}.

The fifth and final section is entirely devoted to Theorem \ref{analytic existence theorem in intro}. We explain the continuity method and reduce the proof to an a priori-estimate. 

\subsection*{Acknowledgement} This article is part of the author's PhD thesis. The author is financially supported  by the graduate school \grqq IMPRS on Moduli Spaces" of the Max-Planck-Institute for Mathematics in Bonn  and  would like to thank his advisor, Prof. Ursula Hamenstädt, for her encouragement as well as helpful discussions. Moreover, the author is grateful to Prof. Hans-Joachim Hein for his interest in this work and his comments on earlier versions of this article.

	\section{Linear analysis on ACyl manifolds} \label{section linear analysis on acyl manifolds}

%\subsection{Asympotically cylindrical K\"ahler manifolds}
In this section, we review the basic definitions and theorems about asymptotically translation-invariant operators on ACyl manifolds following the presentation in \cite{haskins2015asymptotically}[Section 2.1] and \cite{nordstromPhd}[Section 2.3]. The goal is to apply the general theory to the special class of operators that arise as the linearisation of the Monge-Amp\`ere operator in Section \ref{section monge ampere equation} below.

We begin by recalling the definition of ACyl manifolds. For simplicity, we restrict our attention to the case of only \textit{one} cylindrical end, i.e. a connected cross-section. 
\begin{defi} \label{definition ACyl manifold}
	A complete Riemannian manifold $(M,g)$ is called \textit{asymptotically cylindrical (ACyl) of rate $\delta>0$} if there is a bounded open set $U\subset M$, a connected and closed Riemannian manifold $(L,g_L)$ as well as a diffeomorphism $\Phi: [0,\infty) \times L \to M \setminus U$ such that 
	\begin{align*}
		|\nabla^k \left(  \Phi^* g -g_{cyl}\right)|= O(e^{-\delta t} )
	\end{align*}
	for all $k\in \mathbb{N}_0$, where $g_{cyl}:= dt^2 + g_L$ is the product metric and both $\nabla$ and $|\cdot|$ are taken with respect to this metric. Here $t$ denotes the projection onto $[0,\infty)$ and we extend the function $t \circ \Phi^{-1}$ smoothly to all of $M$. This extension is called a \textit{cylindrical coordinate function}, $(L,g_L)$ is called the \textit{cross-section} and $\Phi$ the \textit{ACyl map}.
\end{defi}

Throughout this section, $(M,g)$ denotes an ACyl manifold of rate $\delta>0$ as defined above. It will be convenient to suppress $\Phi$ and simply view $t$ as smooth a function on $M$. 

Let $E,F\to M$ be tensor bundles over $M$ and denote the corresponding space of smooth sections of $E$ and $F$ by $\Gamma(E)$ and $\Gamma(F)$, respectively. Then we consider a differential operator $P: \Gamma(E) \to \Gamma(F)$ of order $l$ and we would like to understand $P$ on the cylindrical end $M\setminus U \cong [0,\infty) \times L$.  

As in \cite{marshallphd}[Section 4], we cover the compact link $L$ by charts $V_1,\dots, V_N$ so that both $E$ and $F$ are trivial over each $\mathbb{R}_+ \times V_\alpha$. Given $u\in \Gamma(E)$, we denote the components of $u$ and $Pu$ on $\mathbb{R}_+\times V_\alpha$ by $u_j ^\alpha$ and $(Pu)_i ^\alpha$, respectively, where  $\alpha=1,\dots, N$,  $j=1,\dots , \operatorname{rank} E$ and $i=1,\dots, \operatorname{rank} F $.  Moreover, there are smooth functions $P_{ij}^{\alpha \beta}:\mathbb{R}_+ \times V_\alpha \to \mathbb{C}$ such that
\begin{align}\label{differential operator in coordinates}
	(Pu)^\alpha_i =\sum_{j=1}^{\operatorname{rank} E} \sum_{0\leq |\beta|\leq l } P^{\alpha \beta}_{ij} D^{\beta} u_j ^{\alpha} 
\end{align}
where the second sum runs over all multi-indices $\beta=(\beta_0,\dots,\beta _{\operatorname{dim}L})$ of order $|\beta|$ at most $l$ and $D^{\beta}$ is defined to be 
\begin{align*}
	D^{\beta}:= \frac{\partial ^{|\beta|}}{\partial t^{\beta _0} x_1^{\beta_1} \cdots \partial x_{\operatorname{dim}L}^{\beta_{\operatorname{dim}L}}}
\end{align*}
for coordinates $(x_1,\dots,x_{\operatorname{dim}L})$ of $V_\alpha$. 

Given a second operator $Q:\Gamma(E) \to \Gamma(F)$ also of order $l$, we say that $P$ is \textit{asymptotic} to $Q$ if the  coefficients $P^{\alpha\beta }_{ij},Q ^{\alpha \beta}_{ij}$ defined by (\ref{differential operator in coordinates}) satisfy
\begin{align*}
	\sup_{ \{t\}\times V_\alpha  } \left| \rho_\alpha D^{\gamma} \left( P^{\alpha\beta}_{ij} - Q^{\alpha \beta }_{ij}   \right) \right| \to 0 \;\; \text{ as } \;\; t \to \infty
\end{align*}
for all $i=1,\dots, \operatorname{rank} F$, $j=1,\dots \operatorname{rank}E$, $\alpha=1,\dots N$, $|\beta|\leq l$ and all multi-indices $\gamma$, where $\rho_1,\dots,\rho_N$ is a partition of unity subordinate to the cover $V_1,\dots, V_N$.   
Note that this definition does neither depend on the choice of covering nor on the partition of unity.

With this notion of asymptotic operators, we may introduce the following definitions, compare \cite{marshallphd}[Section 4.2.2].

\begin{defi}\label{definition translation operators}
	Let $P,P_\infty: \Gamma(E) \to \Gamma(F)$ be two differential operators of order $l$ between sections of tensor bundles $E,F\to M$. 
	\begin{itemize}
		 \item[(i)]  $P_\infty$ is called \textit{translation-invariant} if the functions $(P_\infty)^{\alpha\beta}_{ij}$ defined in (\ref{differential operator in coordinates}) are invariant under translation in the $\mathbb{R}_+$-factor, for all $i=1,\dots, \operatorname{rank}F$, $j=1,\dots , \operatorname{rank}E$, $\alpha=1,\dots, N$ and all multi-indices $\beta$ of order at most $l$. 
		 \item[(ii)] $P$ is called \textit{asymptotically translation-invariant} if $P$ is asymptotic to some translation-invariant operator $P_\infty$. 
	\end{itemize}
\end{defi}

Important examples of asymptotically translation-invariant operators include the Laplacian $\Delta_g$ and the operator $d^*$ associated to the ACyl metric $g$. These  are asymptotic to the corresponding operators associated to the cylinder $g_{cyl}$.

Such operators  may  in general not be Fredholm between the usual H\"older spaces  because $M$ is noncompact. However, this changes if we introduce weight functions.

\begin{defi} \label{definition weighted function spaces}
Let ($M,g$) be an ACyl manifold with cylindrical coordinate $t$ and suppose $E\to M$ is a tensor bundle. The metric on $E$ induced by $g$ is also denoted by $g$, with corresponding connection $\nabla$. 
\begin{itemize}
	\item[(i)] For $\alpha \in (0,1)$, the H\"older semi-norm $[\,\cdot \,]_{C^{0,\alpha}}$ is defined for any continuous tensor field $v$ over $M$ by
	\begin{align*}
		[v]_{C^{0,\alpha}} := \sup_{\substack{x\neq y \in M \\ d_g(x,y)<\frac{i(g)}{2}}} \frac{|v_x-v_y|_g}{d_g(x,y)^\alpha},
	\end{align*}
where $v_x-v_y$ is defined by parallel transport along the minimal geodesic from $x$ to $y$ and $i(g)>0$ denotes the injectivity radius of $g$. 

\item[(ii)] For $k\in \mathbb{N}_0$, $\alpha \in (0,1) $ and $\varepsilon \in \mathbb{R}$, we define $C^{k,\alpha}_\varepsilon (E)$ to be the space of $k$-times continuously differentiable sections $u$ of $E$ such that the norm
\begin{align*}
	||u||_{C_\varepsilon^{k,\alpha}} := \sum_{j=0}^k \sup_M \left|   e^{\varepsilon t} \nabla^ju     \right|_g + [e^{\varepsilon t} \nabla ^k u ]_{C^{0,\alpha}}
\end{align*}
is finite.

\item[(iii)] $C^{\infty}_\varepsilon (E)$ is defined to be the intersection of $C^{k,\alpha}_\varepsilon (E)$ over all $k\in \mathbb{N}_0$.

\item[(iv)] If $u$ is a function on $M$, the corresponding spaces are denoted by $C^{k,\alpha}_\varepsilon (M)$. 
\end{itemize}

\end{defi}

In other words, elements in $C^{\infty}_\varepsilon (E)$, as well as their covariant derivatives, are bounded from above by $e^{-\varepsilon t}$. It is not difficult to see that the definition is independent of the extension of the cylindrical coordinate $t$. 
Moreover, there are continuous inclusions 
\begin{align*}
 C^{k+1}_{\varepsilon} (E)\subseteq C^{k,\alpha}_{\varepsilon} (E) \;\; \text{ and } \;\;	C^{k,\alpha} _{\varepsilon_1 }(E) \subseteq C^{k,\alpha} _{\varepsilon _0}(E),
\end{align*}
if $\varepsilon_0 \leq \varepsilon_1$. 

This notion of weighted H\"older spaces is well-adapted to the study of asymptotically translation-invariant operators. If the operator is moreover elliptic, we have the following weighted Schauder estimates.

\begin{theorem}\label{theorem: ACyl schauder estimates}
Let $(M,g)$ be ACyl and let $P: \Gamma(E) \to \Gamma(F)$ be an elliptic, asymptotically translation-invariant operator of order $l$. Suppose  $h \in C^{k,\alpha}_{\varepsilon}(E)$ and that $u$ is a $k+l$-times continuously differentiable solution to $Pu=h$. If $u \in C^0_\varepsilon (E)$, then $u \in C^{k+l,\alpha}_\varepsilon (E)$ and 
\begin{align*}
	||u||_{C^{k+l,\alpha}_\varepsilon} \leq C \left( ||h||_{C^{k,\alpha}_\varepsilon} + ||u||_{C^0_\varepsilon}   \right)
\end{align*}
for some constant $C>0$ independent of $u$. 
\end{theorem}
 \begin{proof}
 	This is \cite{mazyaschauder}[Theorem 3.16]. 
 \end{proof}

Every translation-invariant operator $P: \Gamma (E) \to \Gamma (F)$ of order $l$ induces  a bounded map $P: C^{k+l,\alpha}_{\varepsilon}(E) \to C^{k,\alpha}_{\varepsilon}(F)$. If $P$ is moreover elliptic, it depends on the weight $\varepsilon\in \mathbb{R}$ whether or not the induced map $P:C^{k+l,\alpha}_{\varepsilon}(E) \to C^{k,\alpha}_{\varepsilon}(F)$ is Fredholm. This naturally leads to the definition of so called critical weights.

\begin{defi}
Let $P:\Gamma(E) \to \Gamma(F)$ be a differential operator asymptotic to a translation-invariant operator $P_\infty: \Gamma(E) \to \Gamma(F)$. $\varepsilon\in \mathbb{R}$ is called a \textit{critical weight} if there exists a non-trivial solution $v=e^{i\lambda t} u: \mathbb{R} \times L\to \mathbb{C}$ to
\begin{align}\label{definition of critical weight}
	P_{\infty} (v)=0
\end{align}
  for some $\lambda \in \mathbb{C}$ with $\operatorname{Im} \lambda = \varepsilon$ and for some smooth section $u=u(t,x)$ of $E$ over $\mathbb{R} \times L$ that is a polynomial in $t$.
\end{defi}
 Note that the set of critical weights is discrete in $\mathbb{R}$.  In the case of functions, i.e. if $E$ is the trivial line bundle, $u$ in the above definition is simply a polynomial in $t$ with smooth functions on $L$ as coefficients. This is crucial because it allows us to explicitly compute critical weights in examples.

The fundamental result in the theory of asymptotically translation-invariant operators is  the following

\begin{theorem}\label{fredholm property of Atrans operator}
Let $P: \Gamma(E)  \to \Gamma(F)$ be an elliptic, translation-invariant operator of order $l$. If $\varepsilon$ is not a critical weight, then the map $P: C^{k+l,\alpha}_\varepsilon(E) \to C^{k,\alpha}_\varepsilon (F)$ is Fredholm. 
\end{theorem}

This result was originally formulated for weighted Sobolev spaces (\cite{lockhart1985elliptic}[Theorem 6.2]). However, as explained in \cite{haskins2015asymptotically}[Section 2.1], the same proof applies in the H\"older setting as well.

Knowing that the induced map $P: C^{k+l,\alpha}_\varepsilon(E) \to C^{k,\alpha}_\varepsilon (F)$ is Fredholm for all non-critical weights $\varepsilon$, we would now like to have a better understanding of its kernel and image. 

\begin{prop}\label{prop kernel independent of small weight change}
	Let $P : \Gamma(E)\to \Gamma (F)$ be an elliptic, translation-invariant operator of order $l$. If an interval $[\varepsilon_1,\varepsilon_2]$ contains no critical weights, then the kernels of $P$ in $C^{k,\alpha} _{\varepsilon_1}(M)$ and $C^{k,\alpha}_{{\varepsilon_2}}(M)$ are equal. 
\end{prop}
This is proven in \cite{lockhart1985elliptic}[Lemma 7.1]. 
To give a precise characterization of the image of $P$, we need to introduce the formal adjoint $P^*: \Gamma(F) \to \Gamma(E)$. It is uniquely defined by the condition that
\begin{align}\label{definition of formal adjoint}
	\langle P u, v\rangle_{L^2} = \langle u, P^*v \rangle _{L^2}
\end{align}
holds for all smooth, compactly supported sections $u,v$. Here, the $L^2$-inner product is defined with respect to the ACyl metric $g$. 
Observe that the identity (\ref{definition of formal adjoint}) extends to sections $u,v$ in certain H\"older spaces.  

\begin{lemma}\label{lemma integration by parts}
	Let $P: \Gamma (E) \to \Gamma (F)$ be an asymptotically translation-invariant operator of order $l$ with formal adjoint $P^*: \Gamma(F) \to \Gamma (E)$. Suppose that $u\in C^{l,\alpha}_{\varepsilon_1}(E)$ and $v\in C^{l,\alpha}_{\varepsilon_2}(F)$ with $\varepsilon_1+ \varepsilon_2 >0$. Then 
	\begin{align*}
		\langle P u, v\rangle_{L^2} = \langle u, P^*v \rangle _{L^2}.
	\end{align*} 
\end{lemma}

The proof is straight forward, and written out in (\cite{nordstromPhd}[Lemma 2.3.15]), for example.

\begin{prop} \label{prop image of A- trans invariant operator}
	Let $P: \Gamma (E) \to \Gamma (F)$ be an elliptic, asymptotically translation-invariant operator of order $l$ with formal adjoint $P^*: \Gamma(F ) \to \Gamma (E)$. If $\varepsilon$ is not a critical weight, then the image of $P: C^{k+l,\alpha}_{\varepsilon}(E) \to C^{k,\alpha}_\varepsilon(F)$ is precisely  the $L^2$-orthogonal complement to the kernel of $P^* : C^{k+l,\alpha}_{-\varepsilon}(F) \to C^{k,\alpha}_{-\varepsilon}(E)$ in $C^{k,\alpha}_\varepsilon (F)$. 
\end{prop}

\begin{proof}
	This can be deduced from Theorem \ref{fredholm property of Atrans operator} and Proposition \ref{prop kernel independent of small weight change}, compare \cite{nordstromPhd}[Proposition 2.3.16] for details.
\end{proof}

We seek to apply this general theory to a certain subclass of asymptotically translation-invariant operators, which naturally arise as the linearisation of the Monge-Amp\`ere operator in Section \ref{section monge ampere equation} below.

\begin{defi}\label{definition ACyl drift Laplacian}
	Let $f$ be a smooth function on an ACyl manifold $(M,g)$. Then the following operator 
	\begin{align*}
		\Delta_f u:= \Delta_g u + g(\nabla^g  f, \nabla^g u)
	\end{align*}
	is called the \textit{drift Laplacian with potential function f}. If additionally 
	$f-2t \in C^{\infty}_{\delta_0} (M)$	for some $\delta_0>0$, we refer to  $\Delta_f$ as an \textit{ACyl drift Laplace operator}.
\end{defi}
Any such operator $\Delta_f$ is self-adjoint with respect to the $L^2$-inner product induced by the measure $e^{f} \operatorname{d V}_g$, i.e.
\begin{align*}
	\int_M (\Delta_fu) v \, e^f \operatorname{d V}_g = \int_M  u (\Delta_f v) \, e^f\operatorname{dV}_g 
\end{align*}
for all smooth, compactly supported functions $u,v$. If $\Delta_f$ is moreover an ACyl drift Laplacian, then it is asymptotic to the translation-invariant operator 
\begin{align} \label{asymptotic model of acyl drift laplacian}
\Delta_{2t} u = \Delta_{g_{cyl}} u + g_{cyl}\left(2\frac{\partial}{\partial t}, \nabla^{g_{cyl} } u\right) = \frac{\partial ^2u}{\partial t ^2}  + 2 \frac{\partial u}{\partial t}  + \Delta_{g_L} u 
\end{align}
where $g_{cyl}=dt^2 + g_L$ is the product metric. From the general theory, we deduce the next

\begin{theorem}\label{theorem ACyl drift laplace is iso}
	Let $(M,g)$ be an ACyl manifold and suppose that $\Delta_f$ is an ACyl drift Laplacian with potential function $f$. Then for any  $k\in \mathbb{N}_0$, $\alpha \in (0,1)$ and  $0<\varepsilon<2$ the operator 
	\begin{align*}
		\Delta_f  : C^{k+2,\alpha}_\varepsilon  (M) \to C^{k,\alpha}_\varepsilon (M)
	\end{align*}
	is an isomorphism.
\end{theorem}

\begin{proof}
	Since $\varepsilon>0$, the injectivity of $\Delta_f$ follows immediately from the standard maximum principle, so we only need to show surjectivity. 
	Before using Proposition \ref{prop image of A- trans invariant operator}, we need to prove the following
	\begin{claim*}
		There are no critical weights for $\Delta_f$ in the interval $(0,2)$. 
	\end{claim*}
Since $\Delta_f$ is asymptotic to the operator given in (\ref{asymptotic model of acyl drift laplacian}), the definition of critical weights requires us  to show that there are no solutions $v$  to the equation
\begin{align}\label{in iso theorem: model equation}
\frac{\partial ^2 v}{\partial t ^2} + 2\frac{\partial v}{\partial t} + \Delta_{g_L} v=0
\end{align}
of the form
\begin{align}\label{in iso theorem: ansatz for v}
	v=e^{i\lambda t} \sum_{j=0}^{d} a_j t^j
\end{align}
  with $\operatorname{Im} \lambda =\varepsilon \in (0,2)$ and functions $a_j$  on $L$. To see this, we plug (\ref{in iso theorem: ansatz for v})  into (\ref{in iso theorem: model equation}) and by considering the coefficient of $t^d$, we observe that (\ref{in iso theorem: model equation}) can only be satisfied if 
  \begin{align}\label{in iso theorem: reduced equation in highest coefficient}
  	\Delta_{g_L} a_d + (-\lambda^2 +2 i \lambda) a_d =0.
  \end{align}
  This implies that $-\lambda^2 + 2i \lambda$ must be real and non-negative because $\Delta_{g_L}$ is a negative and self-adjoint operator on the closed manifold $(L,g_L)$. 
  Writing $\lambda= \gamma + i \varepsilon$, this translates into 
  \begin{align}\label{in iso theorem: final equation}
  2 \gamma (1-\varepsilon)=0, \;\;\; \text{and } \;\; -\gamma^2 +  \varepsilon (\varepsilon-2) \geq 0.
  \end{align}
   If $\varepsilon=1$, the second equation in (\ref{in iso theorem: final equation}) gives a contradiction, and so $\gamma=0$. Then, however, the second equation implies $\varepsilon\geq 2$ since $\varepsilon>0$. Thus, there cannot be a solution to (\ref{in iso theorem: model equation}) of the form (\ref{in iso theorem: ansatz for v}) with $\varepsilon \in (0,2)$, proving the claim. 
   
Hence, according to Proposition \ref{prop image of A- trans invariant operator}, it suffices to show that the formal adjoint $\Delta_f ^*$ of $\Delta_f$ is injective when viewed as a map $\Delta_f^*: C^{k+2,\alpha}_{-\varepsilon} (M) \to C^{k,\alpha}_{-\varepsilon}(M)$ with $0<\varepsilon<2$. A simple computation shows that $\Delta_f ^*$ is given by 
\begin{align*}
	\Delta_f^* u= \Delta_g u - g(\nabla^g f, \nabla^g u) - u \, \Delta_g f.
\end{align*}
  Assuming that $u \in C^{k+2,\alpha}_{-\varepsilon}(M)$ satisfies $\Delta^*_f u =0$, we compute
  \begin{align*}
  	\Delta_f (e^{-f }u )&= u\, \Delta_f e^{-f} + e^{-f} \Delta_f  u + 2 g(\nabla^g e^{-f},\nabla^g  u) \\
  	&=- u e^{-f} \Delta_f f+ e^{-f} \Delta_g u - e^{-f} g(\nabla ^g f, \nabla^g u)\\
  	&= e^{-f} \Delta_f ^* u\\
  	&=0.
  \end{align*}
  Since $\varepsilon<2$, the function $e^{-f}u$ tends to zero as $t\to \infty$, and so the maximum principle implies that $e^{-f}u$ vanishes identically. Thus, the kernel of $\Delta^*_f$ is trivial and the theorem follows. 
\end{proof}

We end this section by proving a (global) Poincar\'e-inequality for a certain drift Laplace operator, which is needed later on to obtain $L^2$-estimates for the Monge-Amp\`ere operator as in \cite{conlon2020steady}. 

\begin{prop}\label{poincare inequality}
	Let $(M,g)$ be an ACyl manifold.  If $f$ is a $C^2$-function on $M$ satisfying $f-2t\in C^{2}_{\delta_0}(M)$ for some $\delta_0>0$ then there exists a constant $\lambda>0$ such that
	\begin{align*}
		\lambda  \int_M u^2 \frac{e^{f}}{(f+c)^2} \operatorname{d V}_g \leq \int_M |\nabla^g u|^2_g\, \frac{e^f}{(f+c)^2} \operatorname{dV}_g	
	\end{align*}
	holds for all smooth, compactly supported functions $u$ on $M$, where $c>0$ is chosen so that $f+c>0$. 
\end{prop}

\begin{proof} First of all note that we can  assume that $f+c>0$ for some $c>0$ because $f$ is proper since $f-2t \in C^{2}_{\delta_0}(M)$. 
	By \cite{conlon2020steady}[Lemma 5.1], it is sufficient to find a positive $C^2$-function $v$ on $M$ and a positive constant $\lambda_0$ such that $\Delta_{f-2\log (f+c)} \, v \leq -\lambda_0 v$ outside some compact subset $K\subset M$. 
	
	We claim that this condition holds for the function $v:= e^{-\frac{f}{2}}$. Indeed, we first calculate
	\begin{align*}
		\Delta_{f-2\log (f+c)} \, e^{-\frac{f}{2}}= - e^{-\frac{f}{2}} \left( \frac{1}{2} \Delta_g f + \left(\frac{1}{4}-\frac{1}{f+c}\right)g(\nabla^g f, \nabla^g f)   \right),
	\end{align*}
	and, since $f-2t\in C^2_{\delta_0}(M)$, we then observe that $(f+c)^{-1} \to 0$ in the limit $t\to \infty$, as well as
	\begin{align*}
	\Delta_g f \to \Delta_{g_{cyl}} t =0, \;\; \text{and }\;  |\nabla^g f|^2 _g \to |\nabla^{g_{cyl}} t|^2_{g_{cyl}}=1 \;\; \text{if } \;\; t\to \infty.  
	\end{align*}
	The claim now follows by taking for instance $\lambda_0=1/8$ and  $ K:= \{ x\in M \; | \; t(x)\leq m  \}$ for $m\gg 1$ large enough. 
\end{proof}

	\section{Preliminaries on K\"ahler-Ricci solitons} \label{section preliminaries on KRS}

In this section, we recall some basic definitions and facts about steady K\"ahler-Ricci solitons. In particular, we review when a solitons is gradient. The main result in this direction is Proposition \ref{proposition when acyl metric is hamiltonian}, which states a criterion for the radial vector field on an ACyl K\"ahler  manifold to be a gradient field. 

\begin{defi}\label{definition KRS}
	 A triple $(M,g,X)$ consisting of a    complete K\"ahler manifold $(M,g)$ and a complete real holomorphic vector field $X$ on $M$ is  a   \textit{steady K\"ahler-Ricci soliton} if the corresponding K\"ahler form $\omega$ satisfies
	 \begin{align}\label{KRS defining equation}
	 	\operatorname{Ric}(\omega) = \frac{1}{2} \mathcal{L}_X \omega,
	 \end{align}
	 where $\operatorname{Ric}(\omega)$ denotes the Ricci form of $\omega$ and $\mathcal{L}_X $  is the Lie derivative  in direction of $X$. The vector field $X$ is called the \textit{soliton vector field}.  
\end{defi}

We say that  a steady K\"ahler-Ricci soliton ($M,g, X$) is \textit{gradient} if $X=\nabla^g  f$ for some smooth real-valued function $f$ on $M$. In this case, $f$ is called the \textit{soliton potential} and equation (\ref{KRS defining equation}) becomes 
\begin{align}\label{formula for ricci form of gradient KRS}
 \operatorname{Ric}(\omega) = i \partial \partialb f.
\end{align}
For us, the most important example is Hamilton's cigar soliton \cite{hamilton1988ricci}.

\begin{example}[Cigar soliton] \label{example cigar soliton}
	Let $M= \mathbb{C}$ and consider the following  metric 
	\begin{align*}
		g_{cig} = \frac{1}{1+x^2 +y^2}\, (dx^2 + dy^2)  
	\end{align*}	
	which is K\"ahler with K\"ahler form 
	\begin{align*}
		\omega_{cig} = \frac{1}{1+|z|^2}\, \frac{i}{2}dz \wedge d\bar{z},
	\end{align*}
	where $z=x+iy$ is the standard coordinate for $\mathbb{C}$. A straight forward computation then shows that  $(\mathbb{C}, g_{cig})$  defines a K\"ahler-Ricci soliton with real holomorphic vector field 
	\begin{align*}
		X= 2 x \frac{\partial}{\partial x} + 2y \frac{\partial}{\partial y}= 4 \operatorname{Re} \left( z \frac{\partial}{\partial z} \right).
	\end{align*}

\end{example}

In fact, $(\mathbb{C},g_{cig})$ is also an ACyl manifold in the sense of Definition \ref{definition ACyl manifold}, with ACyl map $\Phi: \mathbb{R}\times S^1 \to \mathbb{C}^*$ given by
\begin{align*}
	\Phi(t, e^{2\pi i \theta }) := e^{t+ 2\pi i \theta}.
\end{align*}
Under this change of coordinates, we have 
\begin{align*}
\Phi_* \frac{\partial }{\partial t} =  x \frac{\partial}{\partial x} + y \frac{\partial}{\partial y} \;\; \text{ and } \;\;
\Phi^* g_{cig}  = \frac{1}{1+ e^{2t}} (dt^2 + d\theta^2), 
\end{align*}
from which it is easy to see that $g_{cig}$ is indeed asymptotic to $g_{cyl}= dt^2 + d\theta ^2$. 

In higher dimension, further examples of ACyl K\"ahler-Ricci solitons can be obtained by taking the product of the cigar soliton with a compact, Ricci-flat K\"ahler manifold.
 Such examples, and their finite quotients, are the asymptotic model for the solitons constructed in Theorem \ref{geometric existence theorem}. 
 
 Under certain conditions, the soliton equation (\ref{KRS defining equation}) can be reduced to solving a  Monge-Amp\`ere equation, as shown in \cite{conlon2020steady}[Proposition 4.5], for example. We adapt their arguments to obtain the next
 \begin{lemma}\label{lemma reducing to MA equation}
 	Let $(M,g,J)$ be a K\"ahler manifold of dimension $n$ with K\"ahler form $\omega$. Let $X$ be a real holomorphic vector field such that $X=\nabla^g f$, for some smooth function $f:M\to \mathbb{R}$, and suppose $M$ admits a nowhere-vanishing, holomorphic $(n,0)$-form $\Omega$. If there is a smooth function $\varphi:M \to \mathbb{R}$ satisfying
 	\begin{align} \label{MA in lemma on reducing in general}
 		\left(\omega+i \partial \partialb \varphi \right)^n= e^{-f-\frac{X}{2}(\varphi)} i^{n^2 } \, \Omega \wedge \overline{\Omega},
 	\end{align}
 	then $\omega+i\partial \partialb \varphi $ defines a steady K\"ahler-Ricci soliton with vector field $X$. Moreover, if $\varphi$ is $JX$-invariant,  the resulting soliton is gradient. 
 \end{lemma}

\begin{proof} We closely follow the computation provided in the proof of \cite{conlon2020steady}[Proposition 4.5].
	Suppose $\omega_{\varphi}:= \omega + i \partial \partialb \varphi$ satisfies (\ref{MA in lemma on reducing in general}) and compute:
	\begin{align*}
	\operatorname{Ric}(\omega_{\varphi})&= - i \partial \partialb \log \frac{\omega_{\varphi}^n}{i^{n^2} \Omega\wedge \overline{\Omega}} \\
	&=i \partial \partialb  f+ \frac{X}{2}(\varphi) \\
	&=\frac{1}{2}\mathcal{L}_X \omega + \frac{1}{2}\mathcal{L}_X \left(i \partial \partialb \varphi \right) =\frac{1}{2} \mathcal{L}_X \omega_\varphi,
	\end{align*}
	where we used in the last line that $X$ is real holomorphic, and
	\begin{align*}
		\frac{1}{2} \mathcal{L}_X \omega = \frac{1}{2} d \iota_{X} \omega = \frac{1}{2} d J \iota_{JX} \omega = -\frac{1}{2} d  J d f = i \partial \partialb f
	\end{align*}
	since $X=\nabla^g f$. 
	So, if $g_\varphi$ is the metric corresponding to $\omega_{\varphi}$, the triple $(M,g_\varphi, X)$ defines a steady K\"ahler-Ricci soliton. 
	
	For the second claim, assume $\varphi$ to be $JX$-invariant. It is not difficult to see that $\iota_{JX} \left(2i\partial \partialb \varphi \right) =-d X(\varphi)$, compare the proof of \cite{conlon2020steady}[Lemma 7.3] for instance. Then we conclude
	\begin{align*}
	\iota_{JX} \left(\omega + i \partial \partialb \varphi \right) =  -d\left( f + \frac{X}{2}(\varphi) \right)
	\end{align*}
	i.e. $X= \nabla^{g_\varphi} \left(f+ \frac{X}{2}(\varphi)\right)$ as claimed. 
\end{proof}
 
 We conclude this section by addressing the question when a given K\"ahler-Ricci soliton is gradient. It is not difficult to see that if it is gradient, then $JX$ is a Killing vector field for the metric. Under certain conditions, the converse is true as well. 
 
 \begin{lemma} \label{lemma when soliton is gradient}
 	Let $(M,g,X)$ be a steady K\"ahler-Ricci soliton and suppose that $JX$ is  Killing for $g$, where $J$ denotes the complex structure of $(M,g)$. If $H^1(M,\mathbb{R})=0$, then the soliton $(M,g,X)$ is gradient. 
  \end{lemma}

\begin{proof}
	This is a special case of \cite{conlon2020expanding}[Corollary A.7].
\end{proof}

In the special case of ACyl manifolds, we can replace the condition $H^1 (M,\mathbb{R})=0$ in Lemma \ref{lemma when soliton is gradient} by an asymptotic condition on the vector field $X$. In fact, there is the following statement for more general ACyl K\"ahler manifolds.

\begin{prop} \label{proposition when acyl metric is hamiltonian}
	Let $(M,g)$ be an ACyl K\"ahler manifold of rate $\delta >0$ with complex structure $J$ and ACyl map $\Phi$. Suppose $X$ is a real holomorphic vector field on $M$ such that 
	\begin{align} \label{in prop acyl implies hamilton: condition on the vector field}
	X= 2 \Phi_* \frac{\partial}{\partial t}
	\end{align}
	outside some compact domain. If $JX$ is Killing for $g$, then there exists a smooth function $f:M \to \mathbb{R}$ with 
	$
		f-2t \in C^{\infty}_\delta  (M)
	$
	such that $X= \nabla^g f$. 
\end{prop}

\begin{proof}
	The idea is to adapt a proof of Frankel for compact manifolds (\cite{frankel59}) to the ACyl setting. This is possible because there is a version of Hodge splitting on such manifolds, see for example \cite{nordstromPhd}[Section 2.3.3].
	
	Let $\omega$ be the K\"ahler form of $(M,g,J)$. First, since $JX$ is Killing and $X$ is real holomorphic, we have $\mathcal{L}_{JX} g=\mathcal{L}_{JX}J=0$ and so $\mathcal{L}_{JX} \omega =0$. In particular, the 1 form $\iota _{JX} \omega$ is closed. We would like to show that it is in fact exact, for which we need to understand its asymptotic behaviour. 
	
	Let $\Phi ^{-1} \circ t$ be a cylindrical coordinate function for $(M,g)$, whose smooth extension to all of $M$ is denoted by $\tau$. Then we claim that 
	\begin{align}\label{in proof of prop hamiltonian: estimate of omega plus d tau}
		\iota_{JX} \omega + d\tau  \in C^{\infty}_{\delta} (\Lambda^1 (M)).
	\end{align} 
	Indeed, outside of a sufficiently large domain so that (\ref{in prop acyl implies hamilton: condition on the vector field}) is satisfied, we can estimate
	\begin{align*}
	 |d\tau + \iota_{JX} \omega  |_g &= |\iota_{\Phi_* \partial_t}  (\Phi_*g_{cyl})- \iota_X g  |_g 
	 \leq  |X|_g \cdot |\Phi_* g_{cyl} - g|_g = O(e^{-\delta t})		  
	\end{align*}
	because $(M,g)$ is ACyl of rate $\delta>0$ and the norm of $X$ is uniformly bounded on $M$. Here we used that on the product $\mathbb{R} \times L$,  the tensors $dt$ and $g_{cyl}$ are related by $\iota_{\partial_t} g_{cyl} = dt$. A similar estimate holds for the first covariant derivative 
	\begin{align*}
		|\nabla^g \left( \iota_{\Phi_* \partial t} (\Phi_* g_{cyl} ) - \iota_X g   \right)    |_g &\leq |\nabla^g X|_g \cdot|\Phi_* g_{cyl} -g |_g + |X|_g \cdot |\nabla^g g_{cyl}|_g \\
		& = O(e^{-\delta t})
	\end{align*}
	since $|\nabla^g X|_g =O(1)$ and $|\nabla^g g_{cyl}|_g$ decays exponentially of rate $\delta$. Similarly, we can proceed by induction to obtain bounds on higher derivatives, which implies (\ref{in proof of prop hamiltonian: estimate of omega plus d tau}). 
	
	By the  ACyl version of Hodge splitting (\cite{nordstromPhd}[Theorem 2.3.27]), 
	there are 1-forms $h,\alpha, \beta \in C^{\infty}_{\varepsilon}(\Lambda^1 M)$ such that 
	\begin{align} \label{in proof of prop hamiltonian action: Hodge decomposition}
		\iota_{JX}\omega + d\tau = h + \alpha + \beta,
	\end{align}
	where $h$ is $\Delta_g$-harmonic, $\alpha$ exact and $\beta$ co-exact. Here, $0<\varepsilon<\min \{\delta,\lambda\}$, with $\lambda$ denoting the smallest (positive) critical weight of the Laplace operator $\Delta_g$ acting on 1-forms. Moreover, we can write 
	\begin{align*}
	\alpha = d\tilde{f} \;\;\text{ and } \;\; \beta = d^* \gamma
	\end{align*}
	 for some $\tilde{f} \in C^{\infty}_{\varepsilon} (M)$ and $\gamma \in C^{\infty}_0 (\Lambda ^2 M)$. (Note that the growth of $\gamma$  follows from \cite{nordstromPhd}[Theorem 2.3.27] since translation-invariant forms on the cylinder $\mathbb{R}\times L$ are bounded with respect to $g_{cyl}$, and  that we can indeed assume $\tilde{f}$ decays at infinity because the only translation-invariant harmonic functions are constants.) 
	 
	 We have to show that both $h$ and $\beta$ vanish identically. We begin by observing that $h$ is closed. Since $h$ is $\Delta_g$-harmonic and in $C^{\infty}_\varepsilon (\Lambda ^1 M)$, we may,  according to Lemma \ref{lemma integration by parts}, integrate by parts to obtain
	 \begin{align}\label{in prop hamiltonian function: h is closed}
0= \langle h, \Delta_g h \rangle_{L^2} = \langle dh, dh\rangle_{L^2} + \langle d^* h, d^*h\rangle_{L^2} ,
	 \end{align}
	 i.e. $dh=0$ and $d^*h=0$. 
	  Hence, we deduce immediately from the decomposition (\ref{in proof of prop hamiltonian action: Hodge decomposition}) that $\beta$ is also closed. Integrating by parts then yields
	 \begin{align*}
	 	\langle \beta,\beta \rangle_{L^2 } = \langle \beta, d^* \gamma\rangle_{L^2 } = \langle d \beta , \gamma  \rangle_{L^2}=0,
	 \end{align*}
	 so $\beta \equiv0$ as desired. 
	 
	 Next, we follow the proof of \cite{frankel59}[Lemma 2] to show that $h\equiv0$. By assumption, $JX$ is Killing for $g$ and so 
	 \begin{align*}
	 	\Delta_g \left( \mathcal{L}_{JX} h\right) = \mathcal{L}_{JX} \left( \Delta_g h \right)=0,
	 \end{align*}
	 but also $\mathcal{L}_{JX}h= d (\iota_{JX} h)$, i.e. $\mathcal{L}_{JX}h$ is a harmonic and exact  1-form in $C^{\infty}_{\varepsilon}(\Lambda^1 M)$. Using the orthogonality of Hodge's decomposition, we conclude $\mathcal{L}_{JX} h=0$
	 
	 Moreover, the 1-form $J h (\cdot):= h(J\cdot)$ is also harmonic since the Laplace operator on a K\"ahler manifold preserves the bi-degree decomposition of the cotangent bundle. Using the same computation as in (\ref{in prop hamiltonian function: h is closed}), we conclude that $Jh$ is closed, from which we further deduce that 
	 \begin{align*}
	 	d \left( \iota_{JX} (Jh) \right) =\mathcal{L}_{JX}(Jh)= \mathcal{L}_{JX}(J) h + J \mathcal{L}_{JX} h=0
	 \end{align*}
	 because $JX$ is real holomorphic, i.e. $\mathcal{L}_{JX} J=0$. In particular, the function $\iota_{JX} (Jh)=- h(X)$ is constant on $M$, and thus it can only be identically zero as $h(X)$ tends to zero at infinity. 
	 This, together with integration by parts, in turn gives
	 \begin{align*}
	 	\langle h, h\rangle _{L^2}& = \langle \iota _{JX} \omega + d\tau , h \rangle_{L^2} \\
	 	&= \langle \iota_{JX} \omega,h\rangle_{L^2 } + \langle \tau , d^* h \rangle_{L^2} \\
	 	&= -\int _M h(X) \operatorname{d V}_g \\
	 	&=0,
	 \end{align*}
	 where we used in the penultimate line that $\iota_{JX} \omega$ is the negative $g$-dual of $X$ and $d^*h=0$. We conclude $h\equiv0$, and consequently
	 \begin{align*}
	 	\iota_{JX} \omega = d \tilde{f} - d \tau
	 \end{align*}
	 with $\tilde{f} \in C^{\infty}_{\varepsilon}(M)$. It remains to improve the decay rate of $\tilde{f}$, i.e. we need to show $\tilde{f} \in C^{\infty}_{\delta}(M)$ instead of merely $\tilde{f} \in C^{\infty }_\varepsilon(M)$.  It clearly suffices to prove $\tilde{f}\in C^{0}_{\delta}(M)$ because we already know from (\ref{in proof of prop hamiltonian: estimate of omega plus d tau}) that $df \in C^{\infty} _{\delta}(\Lambda^1M)$.
	 
	 Working on the cylindrical end, we write $\tilde{f} (t,x):= \tilde{f}\circ \Phi (t,x)$ for $t\in \mathbb{R}_+$ and $x\in L$, and express $\tilde{f}$ as an integral of the radial derivative as follows:
	 \begin{align*}
	 	\tilde{f}(t,x)= -\int^{\infty} _t \partial_s \tilde{f}(s,x)ds.
	 \end{align*}
	 This, together with $d\tilde{f} (X)=O(e^{-\delta t})$, implies $\tilde{f}= O(e^{-\delta t})$ as required. Proposition \ref{proposition when acyl metric is hamiltonian} now follows by setting $f:= -\tilde{f} +\tau$.
\end{proof}
	 
	\section{The existence theorem} \label{section existence theorem}

The goal of this section  is to show the main result of this article (Theorem \ref{geometric existence theorem in introduction}). We begin by introducing a more general  setup and discussing the existence of ACyl K\"ahler metrics on the considered manifolds.  Step by step, we then add further assumptions and point out their importance for Theorem \ref{geometric existence theorem in introduction}. This discussion will also be accompanied by a simple, but illustrative example.

Throughout this section, let $D=D^{n-1}$ be a compact K\"ahler manifold of complex dimension $n-1$ and  assume that $\gamma : D \to D$ is a biholomorphism of order $m>1$.
Consider the orbifold $M_{orb}:= \left( \mathbb{C} \times D   \right) / \,  \Gamma$, where we set $\Gamma:= \langle \gamma \rangle \cong \mathbb{Z}_m$ and let $\gamma$ act on the product via
\begin{align}\label{gamma acts on product}
	\gamma (z,w):= \left(e^{\frac{2\pi i}{m}}  z, \gamma(w)  \right).
\end{align}
The  singular part $M_{orb}^{sing}$ of $M_{orb}$ is clearly contained in the slice $(\{0\} \times D) / \, \Gamma$ and corresponds to the fixed points of $\gamma$ on $D$. 

Let $\pi: M \to M_{orb}$ be a  resolution of $M_{orb}$, with exceptional set $E= \pi^{-1}(M_{orb}^{sing})$. Then we use $\pi$ to identify $M \setminus E \cong M_{orb} \setminus M^{sing}_{orb}$ and, in particular, we view $\left( \mathbb{C}^* \times D \right) / \, \Gamma$ as an (open) complex submanifold of $M$. 

It is instructive to keep the following example in mind.

\begin{example}[A first example] \label{example: a first example}Let $D= \mathbb{T}$ be the (real) 2-torus and define $\gamma=- \operatorname{Id}$. Then consider the orbifold $\left(\mathbb{C}\times D\right) / \langle \gamma \rangle$ with four isolated singular points contained inside the slice $\{0\} \times D /\langle \gamma \rangle$ and  locally isomorphic to a neighborhood of the origin in  $\mathbb{C}^2 / \mathbb{Z}_2$. Blowing-up each of these rational double points  then yields a resolution  $\pi:M \to \left(\mathbb{C}\times D\right) / \langle \gamma \rangle $. 
	
We point out that this complex manifold $M$ does admit K\"ahler metrics, and in fact, certain Calabi-Yau metrics (so-called ALG gravitational instantons)	 were constructed on $M$ in \cite{biquard2011kummer}[Theorem 2.3]. 
\end{example}

Before finding ACyl K\"ahler metrics on a resolution $\pi:M \to M_{orb}$, we have to fix  an asymptotic model 
$g_{cyl}$ on $\left( \mathbb{C}^* \times D \right) / \Gamma$. For this, we  choose a $\gamma$-invariant K\"ahler metric $g_D$ on $D$ and define the cylindrical parameter $t: \mathbb{C}^* \times D \to \mathbb{R}$ to be
\begin{align}\label{cylindrical parameter}
	t(z,w):= \log |z|.
\end{align}
If $g_{\mathbb{C}}$  denotes the standard flat metric on $\mathbb{C}$, then the product metric
\begin{align}\label{cylindrical metric in terms of t}
	g_{cyl} := e^{-2t} g_{\mathbb{C}} + g_{D}
\end{align}
is $\Gamma$-invariant and can thus be viewed as a metric on 
the quotient $\left( \mathbb{C}^* \times D \right) / \Gamma$. Note that if we let 
\begin{align}\label{ACyl map biholomorphism}
	\begin{split}
		\Phi: \mathbb{R}\times \mathbb{S}^1 \times D &\to \mathbb{C}^* \times D,\\
		(t,e^{2\pi i \theta},w) &\mapsto (e^{t+2\pi i \theta },w)
	\end{split}
\end{align}
then $\Phi^* (g_{cyl})= dt ^2+ g_{\mathbb S ^1} + g_D$, where $g_{\mathbb S ^1}$ denotes the metric on $\mathbb{S}^1$ of length 1. So $g_{cyl}$  is indeed a $\Gamma$-invariant cylinder with cross-section $(\mathbb{S}^1\times D, g_{\mathbb{S}^1} + g_D)$.
 The corresponding K\"ahler form $\omega_{cyl}$ on $\mathbb{C}^* \times D$ is given by 
\begin{align}\label{cylindrical kahler form}
	\omega_{cyl}= |z|^{-2} \frac{i}{2} dz \wedge d \bar z + \omega_{D} = i \partial \partialb t^2+ \omega_{D},
\end{align}
where $\omega_{D}$ is the K\"ahler form associated to $g_D$. 

We would like to understand how to construct ACyl K\"ahler metris on $M$ that are asymptotic to $g_{cyl}$ as in (\ref{cylindrical metric in terms of t})  for some choice of K\"ahler metric $g_D$ on $D$. Moreover, we wish to know which de Rham cohomology classes contain the corresponding K\"ahler forms.

To simplify notation, we introduce the following notion of K\"ahler class.

\begin{defi} \label{definition Kahler class}
	Let $\pi:M \to M_{orb}$ be as above. A class $\kappa\in H^2(M,\mathbb{R})$ is said to be \textit{K\"ahler} if there exists a K\"ahler form $\omega \in \kappa$. 
	
	A K\"ahler class is called \textit{ACyl} if it contains a K\"ahler form whose metric $g$ is ACyl and satisfies
	\begin{align}\label{ACyl Kahler class}
		|\left(\nabla^{g_{cyl}}\right)^k \left( g-g_{cyl} \right)  |_{g_{cyl}}= O(e^{-\delta t}) \;\; \text{ as } \;\; t\to \infty,
	\end{align}
	for some $\delta>0$ and all $k\in \mathbb{N}_0$, where $g_{cyl}$  is given by (\ref{cylindrical metric in terms of t}) for some $\gamma$-invariant K\"ahler metric $g_D$ on $D$. 
\end{defi}

We point out that this notion of ACyl K\"ahler classes is quite restrictive since we only allow ACyl metrics with ACyl diffeomorphism $\Phi$ defined by (\ref{ACyl map biholomorphism}). In particular, the ACyl K\"ahler metric $g$ and its asymptotic cylinder are K\"ahler with respect to the \textit{same} complex structure since $M\setminus E$ is biholomorphic to $\left( \mathbb{C}^* \times D  \right)/ \Gamma$.  

One way to describe ACyl classes is by introducing  a complex compactification $\overline{M}$ of $M$, whose construction we now describe. 

Recall that $\mathbb{C}$ can naturally be compactified to the Riemann sphere $\mathbb{CP}^1 $ by adding one point `at infinity'. We denote this point by $\infty$, i.e. $\mathbb{CP}^1 = \mathbb{C} \cup \{\infty\}$. Consequently,  the orbifold $M_{orb}$ is naturally  compactified by $\left( \mathbb{CP}^1 \times D\right) / \, \Gamma$ and, since $\left(\mathbb{C}^* \times D  \right)/\Gamma$ and $M$ are biholomorphic outside of the exceptional set $E$,  we also obtain a compactification $\overline{M}$ of $M$. 

In other words,  $\overline{M}$ is constructed from $M$ by gluing in the orbifold divisor $\overline{D} := \left(\{\infty \} \times D \right) / \, \Gamma$ at 'infinity'. We emphasize this by writing $\overline{M} = M \cup \overline{D}$. Then the following theorem provides equivalent characterisations of ACyl K\"ahler classes.

\begin{theorem} \label{theorem: characterisation of ACyl classes}
	Let $\pi :M \to M_{orb}$ be as introduced at the beginning of Section \ref{section existence theorem}, and suppose that $\overline{M}=M\cup \overline{D}$ is the compacification obtained by adding an orbifold divisor $\overline{D}$ at infinity.
	For a given $\kappa \in H^2(M,\mathbb{R})$, the following are equivalent:
	
	\begin{itemize}
		\item[(i)] $\kappa $ is an ACyl K\"ahler class.
		
		\item[(ii)] $\kappa= \kappa_{\overline{M}} |_{M}$ for some orbifold K\"ahler class $\kappa_{\overline{M}}$ on $\overline M$.  
	\end{itemize}
Moreover, if the $\mathbb{C}^*$-action $(\mathbb{C} \times D ) / \langle \gamma \rangle$ given by 
\begin{align}\label{in theorem characterisation: C* action}
	\lambda * (z,w)= (\lambda z, w), \;\;\; \lambda \in \mathbb{C}^*,
\end{align}
extends $\pi$-equivariantly to a holomorphic action of  $\mathbb{C}^*$ on $M$, then $(i)$ is equivalent to the following:
\begin{itemize}
	\item[(iii)] There exists some K\"ahler form $\omega_0 \in \kappa $ on $M$ such that the $1$-form $	\iota_{J \frac{\partial }{\partial t}} \omega_0 $ is defined on $M$ and the restriction of $	\iota_{J \frac{\partial }{\partial t}} \omega_0 $  to the open set  $\left(\mathbb{C}^*\times D\right) / \langle \gamma \rangle$ is exact, where $J$ denotes the complex structure on $M$ and $t$ is defined in (\ref{cylindrical parameter}).
	
\end{itemize}
\end{theorem}

The equivalence of $(i)$ and $(ii)$ originates in work on ACyl Calabi-Yau manifolds \cite{haskins2015asymptotically}, however, it is impractical to verify in concrete examples. This is why we introduce criterion $(iii)$. In fact, this condition allows us to prove:

\begin{corollary} \label{corollary: H1=0 implies every kahler class admits a nice ACyl metric}
	Let $\pi : M \to M_{orb}$ be as introduced at the beginning of Section \ref{section existence theorem} and assume that the $\mathbb{C}^*$-action given by (\ref{in theorem characterisation: C* action}) extends $\pi$-equivariantly to a holomorphic action on $M$. 
	
	If every closed, $\gamma$-invariant 1-form on $D$ is exact, then each K\"ahler class is ACyl. 
\end{corollary}

The proof of this corollary also partly justifies   extending the $\mathbb{C}^*$-action (\ref{in theorem characterisation: C* action}) to the resolution.  
\begin{proof}
	Let $\omega_0$ a  K\"ahler form on $M$. Since $\mathbb{S}^1$ is compact and connected, we can average $\omega_0$ over this group to obtain a new closed 2-form $\hat \omega_0$ such that $[\hat \omega_0] = [\omega_0] \in H^2(M,\mathbb{R})$. In fact, $\hat \omega_0$ is a positive (1,1)-form because $\mathbb{S}^1$ acts by biholomorphisms and the averaging does not affect the positivity. 
	
	As the $\mathbb{C}^*$-action (\ref{in theorem characterisation: C* action}) extends to $M$, the radial vector field $\partial /  \partial t$ also extends to a real holomorphic vector field $Y$ on $M$. In particular, 
	\begin{align}\label{genertor of S1 action in corollary}
		Y= \frac{\partial}{\partial t} \;\; \text{ on } \;\; \left( \mathbb{C}^*\times D\right) / \Gamma \subset M
	\end{align}
	and  $JY$ is a generator of the corresponding $\mathbb S ^1$-action, so that 
	\begin{align*}
		\mathcal{L}_{JY} (\hat \omega_0)=0.
	\end{align*}
	  Hence, the 1-form $\iota_{JY}(\hat \omega_0)$ is closed and to apply $(iii)$ of Theorem \ref{theorem: characterisation of ACyl classes}, we need to show that its restriction to $M\setminus E \cong \left( \mathbb{C}^* \times D \right) / \langle \gamma \rangle$ is exact.
	
	Observe that  it is sufficient for the lift of $\iota_{JY}(\hat \omega_0)$ to $\mathbb{C}^*\times D$ to be exact. This lift, in turn, is clearly exact if its restriction  to a slice $\{0\}\times \mathbb{S}^1 \times D \subset \mathbb{R} \times S^1 \times D \cong \mathbb{C^*} \times D$ is exact. Since $\hat \omega_0$ is $\mathbb{S}^1$-invariant and we have $\iota_{JY}(\hat \omega_0) (JY)=0$, this restriction, however, is of the form $p_D^* \alpha$ for some 1-form $\alpha$ on $D$, where $p_D: \mathbb{S}^1\times D \to D$ denotes the projection.  Using that $\iota_{JY}(\hat \omega_0)$ is also closed and $\gamma$-invariant, we conclude that $\alpha$ must  be closed and $\gamma$-invariant  as well, and hence exact by assumption. 
\end{proof}

The proof of  Theorem \ref{theorem: characterisation of ACyl classes} is postponed to Section \ref{subsection constructing background metric}.

\begin{remark} \label{remark on first example}
Let us examine the usefulness of this corollary by considering  Example \ref{example: a first example}. Recall that  in this case, the resolution $\pi : M \to \left( \mathbb{C}^*\times \mathbb{T} \right) / \langle \gamma \rangle$ is obtained by blowing-up the four  fixed points of  $\gamma= - \operatorname{Id}$ on $\mathbb{C} \times \mathbb{T}$. For showing that the $\mathbb{C}^*$-action given by (\ref{in theorem characterisation: C* action}) extends to the blow-up $M$, it suffices to do so locally near each singularity because these are isolated points. This, however, is clearly true because the blow-up
\begin{align} \label{blow-up of C2 mod Z2}
	\mathcal{O}_{\mathbb{CP}^1}(-2) \to \mathbb{C}^2 / \{ \pm \operatorname{Id_{\mathbb{C}^2}} \}
\end{align}
is a toric resolution (with respect to the standard action of $(\mathbb{C}^*)^2$ on $\mathbb{C}^2$). 

Verifying the condition in Corollary \ref{corollary: H1=0 implies every kahler class admits a nice ACyl metric} is also straight forward. Indeed, denoting the holomorphic coordinate of the universal cover $\mathbb{C}$ of $\mathbb{T}^1$ by $w=u+iv$, we see that the translation-invariant 1-forms $du$ and $dv$ are clearly \textit{not} fixed by the action of $-\operatorname{Id}$ on $\mathbb{C}$. 
Thus, \textit{every} K\"ahler class of the blow-up $M$ admits an ACyl K\"ahler metric.
\end{remark}

Having understood when a resolution $\pi : M \to M_{orb}$ admits ACyl K\"ahler metrics, we may continue adding further assumptions in order to find steady K\"ahler-Ricci solitons on $M$. Namely, assume that $D^{n-1}$ admits a nowhere-vanishing holomorphic $(n-1,0)$-form $\Omega_D$ such that
\begin{align*}
	\gamma^* \Omega_D= e^{-\frac{2\pi i }{m}} \Omega_D.
 \end{align*}
This, together with (\ref{gamma acts on product}), implies that the holomorphic $(n,0)$-form $\Omega:=dz\wedge \Omega_D$ is $\gamma$-invariant  and  descends to $M_{orb}$. Thus, we may require the resolution  $\pi : M \to M_{orb}$ to be \textit{crepant},  i.e. we assume that $\Omega$ extends to a nowhere-vanishing form on $M$. 

As in Theorem \ref{theorem: characterisation of ACyl classes}, we additionally assume the extension of the $\mathbb{C}^*$-action (\ref{in theorem characterisation: C* action}) from $M_{orb}$ to $M$. This guarantees that the infinitessimal generator $Y$ of the corresponding $\mathbb{R}_+$-action  is a real holomorphic vector field and thus, multiples of $Y$ are candidates for the soliton field of the desired solitons.  

With these conditions, we recall the main result of this article. 

\begin{theorem} \label{geometric existence theorem}
	Let $D^{n-1}$ be a compact K\"ahler manifold with nowhere-vanishing holomorphic $(n-1,0)$-form $\Omega_D$. Suppose $\gamma: D \to D$ is a complex automorphism of order $m>1$ such that 
	\begin{align} \label{in geometric existence thm: gamma acts on omega D}
		\gamma^* \Omega_D = e^{-\frac{2\pi i}{m}} \Omega_D,
	\end{align} 
	and consider the orbifold $(\mathbb{C} \times D ) / \langle \gamma \rangle$, where $\gamma$ acts on the product via 
	\begin{align}\label{in geometric existence thm: gamma acts on product}
		\gamma(z,w)= \left( e^{\frac{2\pi i }{m}}z ,\gamma(w)  \right).
	\end{align} 
	Let $\pi : M \to (\mathbb{C} \times D ) / \langle \gamma \rangle$ be a crepant resolution   such that the $\mathbb{C}^*$-action on $(\mathbb{C} \times D ) / \langle \gamma \rangle$ given by 
	\begin{align*}
		\lambda * (z,w)= (\lambda z, w), \;\;\; \lambda \in \mathbb{C}^*,
	\end{align*}
	extends $\pi$-equivariantly to a holomorphic action of  $\mathbb{C}^*$ on $M$. 
	
	Then every ACyl K\"ahler class admits a gradient steady K\"ahler-Ricci soliton. 
Moreover, the soliton metric is ACyl of rate $\varepsilon$ for each $0<\varepsilon<2$ and with asymptotic cylinder given by
\begin{align*}
	g_{cyl}=e^{-2t} g_\mathbb{C} + g_{RF},
\end{align*}
where $g_{RF}$ is a  Ricci-flat K\"ahler metric on $D$.
\end{theorem}

Looking back at our Example \ref{example: a first example}, we see that the resolution $\pi:M \to (\mathbb{C} \times \mathbb{T})/ \{\pm \operatorname{Id} \}$ satisfies all requirements  because the blow-up (\ref{blow-up of C2 mod Z2}) of each singularity is indeed crepant, and $\gamma= -\operatorname{Id}$ acts on the holomorphic 1-form on $\mathbb{T}^1$ by multiplication with $-1$. Hence, Theorem \ref{geometric existence theorem}, together with Remark \ref{remark on first example}, imply the existence of a steady K\"ahler-Ricci soliton in \textit{each} K\"ahler class on $M$.

 Following ideas of Conlon and Deruelle developed in \cite{conlon2020steady}[Section 4.2],
the strategy for proving Theorem \ref{geometric existence theorem} is reducing it to a complex Monge-Amp\`ere equation. As explain before Theorem \ref{geometric existence theorem}, the assumptions ensure the existence of a nowhere-vanishing holomorphic $(n,0)$-form as well as  suitable real holomorphic vector fields, so that Lemma \ref{lemma reducing to MA equation} may indeed be used to set up  a Monge-Amp\`ere equation for finding a steady K\"ahler-Ricci soliton. 
 The technical argument for solving the resulting  equation is then provided by Theorem \ref{analytic existence theorem} below, whose proof we postpone to Section \ref{section monge ampere equation}.

\begin{theorem} \label{analytic existence theorem}
	Let $(M,g,J)$ be an ACyl K\"ahler manifold of complex dimension $n$ with K\"ahler form $\omega$. Suppose that  $M$ admits a real holomorphic vector field $X$ such that 
	\begin{align*}
		X = 2 \Phi_*\frac{\partial}{\partial t}
	\end{align*}
	outside some compact domain, where $\Phi$ is the ACyl map and $t$ the cylindrical coordinate function. Moreover, assume that $JX$ is Killing for $g$. 
	
	If  $1 < \varepsilon<2$  and $F\in C^{\infty}_{\varepsilon}(M)$ is JX-invariant, then there exists a unique, $JX$-invariant $\varphi \in C^{\infty}_{\varepsilon}(M)$ such that $\omega + i \partial \partialb \varphi >0$ and 
	\begin{align}
		\left( \omega + i \partial \partialb \varphi \right) ^n  = e^{F - \frac{X}{2} (\varphi)} \omega^n 
	\end{align} 
\end{theorem}
The remainder of this section is structured as follows. In Section \ref{subsection constructing background metric}, 
we focus on proving Theorem \ref{theorem: characterisation of ACyl classes}. In fact, we  provide a detailed construction of the ACyl metrics, and thus obtain   more precise statements than those in Theorem \ref{theorem: characterisation of ACyl classes}. 

Having derived the necessary tools, we then present the proof of Theorem \ref{geometric existence theorem} by reducing it to Theorem \ref{analytic existence theorem}.
Further examples to which Theorem \ref{geometric existence theorem} may be applied are then discussed in Section \ref{subsection examples}.

\subsection{Constructing ACyl K\"ahler metrics} \label{subsection constructing background metric}

The goal is to prove Theorem \ref{theorem: characterisation of ACyl classes}, and we use the  notation introduced at the beginning of Section \ref{section existence theorem}. 

Let $\pi:M \to M_{orb}:= \left( \mathbb{C} \times D   \right) / \Gamma$ be a resolution, where $D$ denotes some compact K\"ahler manifold, and the action of $\Gamma=\langle \gamma \rangle \cong \mathbb{Z}_m$ is given by (\ref{gamma acts on product}). Also, recall that the cylindrical parameter $t: \mathbb{C}^* \times D\to \mathbb{R}$ is defined as $t(z,w)= \log |z|$. 

We begin by focusing on the equivalence of Conditions $(i)$ and $(iii)$ in Theorem \ref{theorem: characterisation of ACyl classes}
as this is most relevant to our purpose. That $(iii)$ implies $(i)$ is settled in the next proposition.

\begin{prop}\label{prop: construction of ACyl metric from a given Kahler class}
	Let $\pi : M \to M_{orb}$ be as introduced at the beginning of Section \ref{section existence theorem} and let the function $t$ be defined by (\ref{cylindrical parameter}). Suppose that $g_0$ is a K\"ahler metric on $M$, whose K\"ahler form $\omega_0$  satisfies
	\begin{align} \label{in proposition construction of ACyl metric: exactness condition}
		\iota_{J \frac{\partial }{\partial t}} \omega_0 = df \;\; \text{ on } \;\; \{ t\geq0 \} \subset \left(\mathbb{C}^* \times D  \right) / \Gamma,
	\end{align}
for some smooth $f: \{t\geq 0\} \to \mathbb{R}$, where $J$ denotes the complex structure on $M$. Then there exists an ACyl K\"ahler metric $g$ on $M$, with K\"ahler form $\omega$, such that $[\omega]=[\omega_0] \in H^2(M,\mathbb{R})$. 

Moreover, if $g$ is lifted to $\mathbb{C}^* \times D$, it is explicitly given by
\begin{align} \label{ACyl metric equal to cylinder outside compact domain}
	g= g_{cyl}= e^{-2t}g_{\mathbb{C}} + g_D \;\; \text{ on } \;\; \{t\geq t_0\}\subset \mathbb{C}^* \times D
\end{align}
for some $t_0>1$, where $g_\mathbb{C}$ denotes the Euclidean metric on $\mathbb C$ and $g_D$ is the restriction of $g_0$ to the slice $\{1\}\times D\subset \mathbb{C}^*\times D$.
\end{prop}

Interestingly, the ACyl metrics obtained  by the previous proposition are of \textit{optimal rate}, i.e. they are \textit{equal} to its asymptotic model $g_{cyl}$ outside some compact domain. This is an even stronger statement than claimed in Theorem \ref{theorem: characterisation of ACyl classes}. 

\begin{proof}
	Analogously to \cite{haskins2015asymptotically}[Section 4.2], the idea is to glue the K\"ahler form $\omega_0$ to a certain cylindrical K\"ahler form $\omega_{cyl}$ outside of some compact domain. Doing so, however, requires that the difference of these two (1,1)-forms  is $\partial \partialb$-exact. 
	
	Thus, before we can perform any gluing, we need to have a description of $\omega_0$ in terms of a K\"ahler potential, at least on the set $\{t\geq 0\}$. We begin by explaining the construction of such a potential function.
	
	Suppose that $\omega_0$ is a  K\"ahler form   satisfying
	\begin{align} \label{in proposition: construction of Kahler form from potential}  
		\iota_{J \frac{\partial }{\partial t}} \omega_0 = df \;\; \text{ on } \;\; \{ t\geq0 \} \subset \left(\mathbb{C}^* \times D  \right) / \Gamma,
	\end{align}
	for some smooth function $f$. Working on $\mathbb{C}^*\times D$, we lift $\omega_0$ and $f$ to $\Gamma$-invariant forms denoted by the same letters. We view $\mathbb{C}^*\times D$ as a (trivial) fibre bundle over $D$, and introduce two holomorphic maps
	\begin{align*}
		j : D \to \mathbb{C}^* \times D \;\; \text{ and } \;\; p: \mathbb{C}^* \times D \to D,
	\end{align*}
	where $j$ is the inclusion of the slice $\{1\} \times D \subset \mathbb{C}^* \times D$, and $p$  the projection onto $D$. Then we \textit{define} a K\"ahler form $\omega_D$ on $D$ by setting
	\begin{align*}
		\omega_D := j^* \omega_0.
	\end{align*}
	Using the cylindrical parameter $t$ as defined in (\ref{cylindrical parameter}), we identify $\mathbb{C}^* \cong \mathbb{R}\times \mathbb S ^1 $ and  define a new function $\varphi$ by 
	\begin{align*}
		\varphi (t,y):= 2 \int_0 ^ t f(s,y) ds \;\; \text{ for }  t \in  \mathbb{R}_{\geq 0}  \text{ and } y\in \mathbb{S}^1\times D.
	\end{align*}
	Then we claim that 
	\begin{align} \label{in prop: construction of ACyl metric:: omega0 admits a potential}
		\omega_0 = i \partial \partialb \varphi  + p^* \omega_D \;\; \text{ on } \;\; \{t\geq 0\} \cong \mathbb{R}_{\geq 0} \times \mathbb{S}^1\times D.
	\end{align}
	In other words, we have to show that the $(1,1)$-form $\alpha:= \omega_0 - i \partial \partialb \varphi$ is a basic form for the fibre bundle $p: \mathbb{C}^* \times D\to D$. This means that 
	\begin{align}\label{in prop construction of ACyl: basic forms definition}
		\mathcal{L}_V \alpha=0 \;\; \text{ and } \;\; \iota_V \alpha=0
	\end{align}
	for all vector fields $V$ on $\mathbb{C}^* \times D$ which are tangent to the fibres of the projection $p$. However, since $\alpha$ is $d$-closed, it suffices to show the second condition in (\ref{in prop construction of ACyl: basic forms definition}), and thus we only have to prove that 
	\begin{align} \label{in prop construction of ACyl metric: sufficient condition to proof claim}
		\iota_{\frac{\partial }{\partial t}} \alpha=0 \;\; \text{ and } \iota_{J \frac{\partial}{\partial t}} \alpha=0
	\end{align}
	since any vector field tangent to fibres of $p$ can be written in terms of $\partial / \partial t$ and $J \partial / \partial t$.
	
	Let us begin by considering the first equation in (\ref{in prop construction of ACyl metric: sufficient condition to proof claim}). Keeping in mind that $(J\partial / \partial t) (f)=0$ by (\ref{in proposition: construction of Kahler form from potential}), we split $df= d_t f+ d_D f$, where $d_t$ and $d_D$ are the differentials in direction of the $\mathbb{R}$- and $D$-factor, respectively. Using  the definition of $\varphi$ and the fact that $ \partial \partialb t =0$, we observe 
	\begin{align*}
		2 i  \partial \partialb \varphi =  dJd \varphi =  2 df\wedge Jdt +  d_t Jd_D \varphi + d_D Jd_D\varphi,
	\end{align*}
	so that we conclude from (\ref{in proposition: construction of Kahler form from potential})
	\begin{align*}
		\iota_{\frac{\partial}{\partial t}} \left( i  \partial \partialb \varphi  \right) 
		= \frac{\partial}{\partial t} f Jdt + \frac{1}{2}Jd_D  \frac{\partial}{\partial t}\varphi
		= Jdf
		= \iota_{\frac{\partial}{\partial t}} \omega_0,
	\end{align*}
	as claimed. The second equation in (\ref{in prop construction of ACyl metric: sufficient condition to proof claim}) follows similarly:
	\begin{align*}
		\iota_{J\frac{\partial}{\partial t}} \left( i  \partial \partialb \varphi  \right) = -df \cdot (Jdt)\left( J \frac{\partial}{\partial t}  \right) = df = \iota_{J \frac{\partial}{\partial t}} \omega_0.
	\end{align*}
	This finishes the proof of (\ref{in prop: construction of ACyl metric:: omega0 admits a potential}). 
	
	Let us define the cylindrical K\"ahler form $\omega_{cyl}$ on $\mathbb{C}^* \times D$ to be 
	\begin{align*}
		\omega_{cyl}:= i \partial \partialb t^2 + p^* \omega_D.
	\end{align*}
	The goal is to construct a new K\"ahler form $\omega$, cohomologous to $\omega_0$, such that
	\begin{align} \label{in prop: construction of Acyl metric:: omega inside and at infinity equal to omega cyl}
		\omega=	\begin{cases}
			\omega_{cyl} & \text{on}\;\; \{t\geq t_2\},\\
			\omega_0 & \text{on} \;\; \{t\leq t_1\}
		\end{cases}
	\end{align} 
	for some positive numbers $t_1<t_2$.
	The following gluing procedure is an adaptation of the one contained on \cite{haskins2015asymptotically}[p. 247]. For this construction, we first fix $t_0>1$ and choose a cut-off function $\chi=\chi(t)$ satisfying
	\begin{align*}
		\chi(t) = 
		\begin{cases}
			1 & \text{if}\;\; t\geq t_0,\\
			0 & \text{if} \;\; t\leq 1,
		\end{cases}
	\end{align*}
	and then we define a $\Gamma$-invariant ($1,1$)-form $\omega$ on $\{t\geq 0 \}$ by
	\begin{align*}
		\omega:= i \partial \partialb \left(\chi(t) \cdot t^2 + (1-\chi(t)) \cdot \varphi    \right) + \rho(t) dt\wedge d^c t + p^*\omega_D,
	\end{align*}
	where $\rho(t) dt\wedge d^ct$ is an exact bump-form supported inside a neighborhood of $[1,t_0]$, say $[1/2, t_0+ 1/2]$. Clearly, $\omega-\omega_0$ is exact and $\omega$ agrees with $\omega_0$ inside the region $\{t\leq 1/2\}$, so that $\omega$ extends to a (1,1)-form on $M$. 
	
	Moreover, we notice that $\omega= \omega_{cyl}$ if $t\geq t_0+1/2$, and thus, the only thing left to show is the positivity of $\omega$ on the region $\{1/2 \leq t\leq t_0 + 1/2 \}$. For $t\in [1/2,t_0+1/2] \setminus [1,t_0]$, this is clear because we have
	\begin{align*}
		\omega= \begin{cases}
			\omega_{cyl} + \rho dt \wedge d^ct & \text{on} \;\; \{t\geq t_0 \}, \\
			\omega_0 + \rho dt \wedge d^ct & \text{on} \;\; \{t\leq 1 \}
		\end{cases}
	\end{align*}
	and $\rho\geq 0$, so we only need to focus on the case $t\in [1,t_0]$.
	
	To show that $\omega>0$ on this region, it suffices to check that $\omega$ is positive in the direction of the $D$-factor since we can then compensate for potentially negative terms by choosing $\rho$ sufficiently large inside $[1,t_0]$. Hence, consider $0\neq v\in T_{\mathbb{C}}D$ and observe 
	\begin{align*}
		\omega(v,\overline{v}) &= (1-\chi(t)) \cdot \left( i \partial \partialb \varphi \right) (v,\overline{v}) + p^*\omega_D (v,\overline{v}) \\
		&=(1-\chi(t)) \cdot  \omega_0 (v,\overline{v}) + \chi (t) \cdot p^*\omega_D (v,\overline{v}) \\
		&>0,
	\end{align*}
	where we used in the first line that $\chi$ only depends on $t$, and the second equation follows from (\ref{in prop: construction of ACyl metric:: omega0 admits a potential}). 
	As explain before,  $\omega$ is positive on $\{ 1\leq t\leq t_0 \}$ once we choose $\rho\gg 1$ on $[1,t_0]$, and so we constructed a K\"ahler form $\omega$ on $M$ in the same cohomology class as $\omega_0$, which also satisfies (\ref{in prop: construction of Acyl metric:: omega inside and at infinity equal to omega cyl}). The corresponding ACyl metric $g$ then fulfills (\ref{ACyl metric equal to cylinder outside compact domain}), since both $g$ and $g_{cyl}$ are K\"ahler with respect to the same complex structure.

\end{proof}

For the converse to Proposition \ref{prop: construction of ACyl metric from a given Kahler class}, i.e. that $(i)$ of Theorem \ref{theorem: characterisation of ACyl classes} implies $(iii)$,  
 we additionally assume that the $\mathbb{C^*}$-action on $M_{orb}$ given by (\ref{in theorem characterisation: C* action})
%\begin{align} \label{C* action on orbifold}
%	\lambda * (z,w)= (\lambda z, w), \;\;\; \lambda \in \mathbb{C}^*
%\end{align}
extends $\pi$-equivariantly to a holomorphic action on the resolution $\pi : M \to M_{orb}$. Hence, the infinitesimal generators of this action extend to real holomorphic vector fields on all of $M$. Let $Y$ denote the generator of the induced $\mathbb{R}_+$-action (corresponding to translation in the cylindrial parameter $t$), i.e. 
\begin{align}\label{generator of the S1 action}
	Y= \frac{\partial}{\partial t} \;\; \text{ on } \;\; \left( \mathbb{C}^*\times D\right) / \Gamma \subset M.
\end{align}
Note that if $J$ is the complex structure on $M$, then $JY$ is generating the $\mathbb{S}^1$-action on $M$. 

Next, we show that Condition $(iii)$ in Theorem \ref{theorem: characterisation of ACyl classes} is in fact necessary for a K\"ahler class to be ACyl. 

\begin{prop}  \label{lemma: kahler potential for ACyl metric}
Let $\pi : M \to M_{orb}$ be as introduced at the beginning of Section \ref{section existence theorem} and assume that the $\mathbb{C}^*$-action given by (\ref{in theorem characterisation: C* action}) extends $\pi$-equivariantly to a holomorphic action on $M$. 

Then every ACyl K\"ahler class contains an  ACyl K\"ahler form $\hat{\omega}$ such that
	\begin{align*}
		\iota_{JY} \hat \omega = df,
	\end{align*}
where $JY$ is the infinitessimal generator of the $\mathbb
S ^1$-action. 
\end{prop}

\begin{proof} Let $g$ be an ACyl K\"ahler metric, with K\"ahler form $\omega$, such that (\ref{ACyl Kahler class}) holds.
	First, average $\omega$ over the $\mathbb S^1$-action to obtain a K\"ahler form $\hat{\omega}$ in the same cohomology class. Then observe that the averaging does not change the asymptotic behavior since both $g_{cyl}$ and $t$ are $\mathbb{S}^1$-invariant, so that the corresponding metric $\hat g$ is  ACyl and satisfies (\ref{ACyl Kahler class}). In particular, the function $t=\log |z|$ is also the cylindrical parameter for $\hat g$. 
	
	 Then we notice that $JY$, for $Y$ given by (\ref{generator of the S1 action}), is a Killing field for $\hat g$ because $\mathcal{L}_{JY}\hat \omega=0$. Thus, Proposition \ref{proposition when acyl metric is hamiltonian} implies that $Y$ is the gradient field of some function on $M$, or equivalently that $\iota_{JY} \hat \omega$ is exact. 
\end{proof}

It only remains to show the equivalence of $(i)$ and $(ii)$ in Theorem \ref{theorem: characterisation of ACyl classes}, i.e. that each ACyl K\"ahler class is the restriction of some orbifold K\"ahler class on the complex compactification $\overline{M}$. 

This goes back to a construction of Haskins, Hein and Nordstr\"om \cite{haskins2015asymptotically}. In fact, their ideas can be used to prove the following

\begin{prop}\label{proposition: restriction of orbifold class}
	Let $\pi :M \to M_{orb}$ be as introduced at the beginning of Section \ref{section existence theorem}, and suppose that $\overline{M}=M\cup \overline{D}$ is the compacification obtained by adding the orbifold divisor $\overline{D}= D/\Gamma$ at infinity.
	
	For a given $\kappa \in H^2(M,\mathbb{R})$, the following are equivalent:
	
	\begin{itemize}
		\item[(i)] $\kappa $ is an ACyl K\"ahler class.
		
		\item[(ii)] $\kappa= \kappa_{\overline{M}} |_{M}$ for some orbifold K\"ahler class $\kappa_{\overline{M}}$ on $\overline M$.  
	\end{itemize}

\end{prop}

\begin{proof}
	That $(i)$ implies $(ii)$ is a  direct consequence of \cite{haskins2015asymptotically}[Theorem 3.2], which can be applied here since $g$ and $g_{cyl}$ are K\"ahler with respect to the same complex structure.
	
	The construction required for the converse implication
 can be found on \cite{haskins2015asymptotically}[p. 247], so we only briefly sketch the idea.
	
	If $\omega_{\overline M}$ is a K\"ahler form on $\overline M$, then we define $\omega_D$ to be the restriction of $\omega_{\overline M}$ to the orbifold divisor $\overline D = \{ \infty\} \times D / \Gamma $. Note that $\omega_D$ lifts to a smooth $\Gamma$-invariant form on $D$, and so we can define the asymptotic model $\omega_{cyl}$ on $\mathbb{C}^*\times D$ to be
	\begin{align*}
		\omega_{cyl}:= i \partial \partial t^2 + \omega_D.
	\end{align*}
 The new ACyl K\"ahler form asymptotic to $\omega_{cyl}$ is then constructed as
\begin{align*}
	\omega:= \omega_{\overline M} + i \partial \partial \left(\chi \cdot t^2 \right) + \rho dt\wedge d^ct,
\end{align*}
for some cut-off function $\chi$ and a bump-function $\rho$. The cut-off $\chi$   is equal to 1 in a neighborhood of $\overline D$ and 0 if $t\leq 0$, and  $\rho$ is chosen sufficiently large to ensure positivity. 
\end{proof}

This concludes the proof of Theorem \ref{theorem: characterisation of ACyl classes}, and so we  focus on proving Theorem \ref{geometric existence theorem} next.

\subsection{Proof of Theorem \ref{geometric existence theorem}}
\label{subsection: proof of geometric existence theorem}
Let $D^{n-1},\Omega_{D}, \Gamma =\langle \gamma \rangle$ and $\pi:M \to (\mathbb{C} \times D ) / \langle \gamma \rangle$ be defined as in Theorem \ref{geometric existence theorem}. In particular, the discussion of the previous subsection applies and we use the same notation as introduced at the beginning of Section \ref{section existence theorem}. We also assume that the $\mathbb{C}^*$-action on $M_{orb}$ defined by  
\begin{align*}
	\lambda * (z,w):= (\lambda z,w), \;\; \lambda \in \mathbb{C}^*,
\end{align*}
extends $\pi$-equivariantly to a holomorphic action on $M$. As a consequence, the infinitesimal generators of this action extend to real holomorphic vector fields on $M$. Let $X$ be two-times the generator of the induced $\mathbb{R}$-action (corresponding to translation in the cylindrical parameter $t$), i.e.
\begin{align*}
	X= 2 \partialt \;\;\; \text{on }\;\; \left( \mathbb{C}^*\times D \right) / \Gamma \subset M. 
\end{align*}
Then $JX$ is  two-times the generator of the $\mathbb{S}^1$-action, where $J$ is the complex structure on $M$. 

Moreover, we point out that the action of $\gamma$ given by (\ref{in geometric existence thm: gamma acts on product}) preserves the holomorphic ($n,0$) form $\Omega$ on $\mathbb{C}^*\times D$ defined as
\begin{align*}
	\Omega:= dz \wedge \Omega_{D}
\end{align*}
since $\gamma$ satisfies (\ref{in geometric existence thm: gamma acts on omega D}). In particular, $\Omega$  descends to $M_{orb}$ and, because the resolution $\pi : M \to M_{orb}$ is crepant, $\Omega$ then extends to a holomorphic $(n,0)$-form on $M$, which  we also denote by $\Omega$.

Let $\kappa \in H^2(M,\mathbb{R})$ be an ACyl K\"ahler class, i.e. there exists an ACyl metric $g$ satisfying (\ref{ACyl Kahler class}) and with K\"ahler form $\omega \in \kappa$.
We need to find a different ACyl metric $g_0$ with K\"ahler form $\omega_0$ also contained in the given class $\kappa$, such that $X=\nabla^{g_0} f$ and  
\begin{align}\label{in existence theorem: final MA equation we desire}
	\left( \omega_0 + i \partial \partialb \varphi  \right) ^n =\alpha e^{-f-\frac{X}{2}(\varphi)} i^{n^2} \Omega\wedge \overline{\Omega},
\end{align}
for some $JX$-invariant functions $f,\varphi:M\to \mathbb{R}$ and some constant $\alpha \in \mathbb{R}$. According to Lemma \ref{lemma reducing to MA equation}, $\omega_0 + i \partial \partialb \varphi$ is then a gradient steady K\"ahler-Ricci soliton, as required.   To achieve this, we begin by modifying $\omega$ near infinity to improve the convergence rate and to ensure that it is asymptotic to a \textit{Ricci-flat} cylinder.

First, we improve the asymptotic behavior of $\omega$ by applying Proposition \ref{prop: construction of ACyl metric from a given Kahler class}, so that there exists an ACyl K\"ahler form $\omega_1 \in [\omega]$ which, if lifted to $\mathbb{C}^*\times D$, is of the form
\begin{align*}
	\omega_1=i \partial \partialb t^2+ \omega_D \;\; \text{ on } \;\; \{t\geq t_0  \}
\end{align*}
for some $t_0>0$. Here,  $\omega_D$ denotes the restriction of $\omega$ to the slice $\{1\}\times D$. 

In a second step, we modify $\omega_0$ so that it becomes Ricci-flat if restricted to $\{t\} \times D$ for $t\gg t_0$.
 Recall that by Yau's Theorem \cite{yau1978ricci}, there exists  $u_D: D\to \mathbb{R}$ such that $\omega_{RF}:= \omega_D + i \partial \partialb u_D>0$ and
 \begin{align}\label{in existence theorem: MA on D}
 	\left(\omega_{RF}  \right)^{n-1} = c i^{(n-1)^2} \Omega_D \wedge \overline{\Omega}_D.
 \end{align}
Moreover, the uniqueness of solutions to (\ref{in existence theorem: MA on D}) implies that $u_D$ is $\gamma$-invariant, because $\gamma$ preserves both $\omega_D$  and $\Omega_D \wedge \overline{\Omega}_D$.

Choosing a cut-off function $\chi$ with
\begin{align*}
	\chi(t) = \begin{cases}
		1 & \text{if } \; t\geq t_0+2\\
		0& \text{if }\; t\leq t_0+1,
	\end{cases}
\end{align*}
we then define a $\Gamma$-invariant $(1,1)$-form by
\begin{align*}
	\omega_0:= \omega_1 + i \partial \partialb \left( \chi\cdot u_D \right) + \rho dt\wedge d^ct,
\end{align*}
where $\rho$ is a bump-function supported in a small neighborhood of $[t_0+1,t_0+2]$. By the same reasoning as in the proof of Proposition \ref{prop: construction of ACyl metric from a given Kahler class}, $\omega_0$ is positive if $\rho$ is sufficiently large and thus, $\omega_0$ defines a K\"ahler metric on $M$ in the class $\kappa=[\omega]$. Note that by construction we have
\begin{align} \label{in existence theorem: omega0 equals Ricci flat cylinder}
	\omega_0 = i \partial \partialb t^2 + \omega_{RF} 
\end{align}
on the region $\{ t\geq t_0 +3\}$. 

The next step is to further modify $\omega_0$ so that it satisfies the requirements of Theorem \ref{analytic existence theorem}. 
Note that after averaging $\omega_0$ over the compact and connected group $\mathbb{S}^1$ we can assume that $\omega_0$ is invariant under the $\mathbb{S}^1$-action  because averaging  neither affects the cohomology class, nor the positivity of $\omega_0$. 
Hence, $JX$ is a Killing field for the corresponding K\"ahler metric $g_0$ and by Proposition \ref{proposition when acyl metric is hamiltonian}, there exists a function $f$ such that
\begin{align*}
	X= \nabla^{g_0} f \;\; \text{ and } \;\; f-2t \in C^{\infty}_{\delta}(M),
\end{align*}
for each $\delta >0$. In fact, we conclude from (\ref{in existence theorem: omega0 equals Ricci flat cylinder}) that 
\begin{align}\label{in existence theorem: }
	f=2t \;\; \text{ on } \;\; \{ t\geq t_0+3  \}.
\end{align}
In particular, we notice that $(M,g_0 )$ satisfies the assumptions of Theorem \ref{analytic existence theorem}. 

Let us define a $JX$-invariant function $F:M \to \mathbb{R}$ by
\begin{align*}
		F:= \log \frac{\alpha i^{n^2}\,\Omega\wedge\overline{\Omega}}{\omega_0^n} - f
\end{align*}
for some constant $\alpha$ to be fixed later. For an appropriate choice of $\alpha$, we claim that $F$ has compact support. 
To see this, first observe from (\ref{in existence theorem: MA on D}) and (\ref{in existence theorem: omega0 equals Ricci flat cylinder}) that the cylindrical volume form of $\omega_{cyl}$ can be computed as
\begin{align*}
\omega_{cyl} ^n = \frac{c\,n}{2} |z|^{-2} i^{n^2} dz \wedge \Omega_D \wedge d\bar{z} \wedge\overline\Omega_{D},
\end{align*}
so we set $\alpha:= cn/2$, and obtain
\begin{align*}
	F&= \log \frac{\alpha i ^{n^2 }\Omega \wedge \overline{\Omega}}{\omega_{cyl} ^n } + \log \frac{\omega_{cyl}^n}{\omega_0^n} -f \\
	&= 2t-f \\
	&=0,
\end{align*}
if $t\geq t_0+3$. Thus, $F$ is compactly supported.

If we fix \textit{some} $0<\varepsilon<2$,  
Theorem \ref{analytic existence theorem} yields a $JX$-invariant $\varphi\in C^{\infty}_{\varepsilon}(M)$ such that
\begin{align}\label{in geometric existence proof: MA equation with F}
	\left( \omega_0 + i \partial \partialb \varphi \right)^n = e^{F-\frac{X}{2}(\varphi)} \omega_0 ^n = \frac{cn}{2} e^{-f-\frac{X}{2}(\varphi)} i^{n^2}\Omega\wedge \overline{\Omega} ,
\end{align}
which is precisely (\ref{in existence theorem: final MA equation we desire}), so that  $\omega_0 + i \partial \partialb \varphi$  defines a gradient steady K\"ahler-Ricci soliton. The underlying K\"ahler metric is clearly ACyl  of rate $\varepsilon$. 

However, since $F\in C^{\infty}_\varepsilon (M)$ for \textit{all} $0<\varepsilon<2$ and since solutions   to  (\ref{in geometric existence proof: MA equation with F}) contained in $C^\infty_\varepsilon(M)$ are \textit{unique}, we may conclude that indeed $\varphi\in C^{\infty}_\varepsilon(M)$ for all $0<\varepsilon<2$, finishing the proof.

\subsection{Examples} \label{subsection examples}

We begin by providing further examples in complex dimension two. 
The   manifolds $M_k$ considered below are  defined as in \cite{biquard2011kummer}[Section 2.2], and their construction is similar to Example \ref{example: a first example}.

\begin{example}
	For $k=2,3,4,6$ we consider the maps $\gamma_k : \mathbb{C}^2 \to \mathbb{C}^2$ given by
	\begin{align*}
		\gamma_k(z_1,z_2) := 
		\left( e^{\frac{2\pi i}{k}  } z_1, e^{-\frac{2\pi i}{k}}  z_2  \right)
	\end{align*}
	If we let $\mathbb{T}$ be the (real) 2-torus, then $\gamma_k$ descends to $\mathbb{C} \times \mathbb{T}$, provided the lattice in $\mathbb{C}$ is chosen appropriately: For $k=2,4$, let $\mathbb{T}$ be obtained from the square-lattice, and for $k=3,6$ use the hexagonal one  instead. 
	
	In any case, we may define orbifolds $M_{orb}^k:= \left( \mathbb{C} \times \mathbb{T} \right) / \langle \gamma_k\rangle $ with isolated singular points which are locally modelled on a neighborhood of the origin in $\mathbb{C}^2/ \mathbb{Z}_j$, with $\mathbb{Z}_j$-action induced by the map
	\begin{align*}
		(z_1,z_2) \mapsto (e^{\frac{2\pi i }{j}}  z_1, e^{-\frac{2\pi i }{j}} z_2)
	\end{align*}
	for $j\in \{2,3,4,6\}$. More precisely,
	\begin{itemize}
		\item If $k=2$, $M_{orb}^2$ has four singularities, all isomorphic to $\mathbb{C}^2/\mathbb{Z}_2$. 
		
		\item If $k=4$,  the corresponding orbifold $M^4_{orb}$ has one $\mathbb{C}^2/\mathbb{Z}_2$ and two $\mathbb{C}^2/\mathbb{Z}_4$  singularities.
		
		\item If $k=3$, there are three singular points in $M_{orb}^3$ and all are isomorphic to $\mathbb{C}^2/\mathbb{Z}_3$. 
		
		\item If $k=6$, $M^6_{orb}$ also has three singularities: one $\mathbb{C}^2/\mathbb{Z}_2$, one  $\mathbb{C}^2/\mathbb{Z}_3$ and  one $\mathbb{C}^2/\mathbb{Z}_6$ singularity.
	\end{itemize}

In each case, condition (\ref{in geometric existence thm: gamma acts on omega D}) is fulfilled and the blow-up of all singularities results in a complex manifold denoted by $M_k$.
The corresponding resolution is indeed crepant since all singularities are isolated points and because blowing-up the origin in the  local models $\mathbb{C}^2 / \mathbb{Z}_j$ yields in fact a crepant resolution.
 Similar to the reasoning in Example \ref{example: a first example} and Remark \ref{remark on first example}, one can show that $M_k$ satisfies the requirements of both Theorem \ref{geometric existence theorem} and Corollary \ref{corollary: H1=0 implies every kahler class admits a nice ACyl metric}. 
 
Thus, there is a  steady K\"ahler-Ricci soliton in \textit{each} K\"ahler class of $M_k$. 
	Interestingly, these manifolds also admit ALG gravitational instantons by \cite{biquard2011kummer}[Theorem 2.3], for instance.  
\end{example}

For finding examples of complex dimension 3, we may take $D$ to be a  product $\mathbb{T}\times \mathbb{T}$, but then we consider a different resolution, as the next example shows. 

\begin{example}
Let $\mathbb T$ be constructed from the hexagonal lattice in $\mathbb{C}$. By setting $D:= \mathbb{T} \times \mathbb{T}$ we define $\gamma: \mathbb{C} \times D\to \mathbb{C} \times D$ by
\begin{align*}
	\gamma(z_1,z_2,z_3)= e^{\frac{2\pi i }{3}} (z_1,z_2,z_3)
\end{align*}
and note that $\gamma^* (dz_2\wedge dz_3)= e^{-\frac{2\pi i }{3}} dz_2\wedge dz_3$, i.e. (\ref{in geometric existence thm: gamma acts on omega D}) is satisfied. Each of the $3^2=9$ singularities of $\left(\mathbb{C} \times D \right) / \mathbb{Z}_3$ is modelled on $\mathbb{C}^3 / \mathbb{Z}_3$, and so we may consider the blow-up $M$ of all singular points. 

As before, this resolution is crepant and the $\mathbb{C}^*$-action on the first factor extends, because the same is true for the resolution 
\begin{align*}
	\mathcal{O}_{\mathbb{CP}^2} (-3) \to \mathbb{C}^3/  \mathbb{Z}_3.
\end{align*}
Moreover, the only closed, $\gamma$-invariant 1-forms on $D$ are clearly exact, so that again \textit{each}  K\"ahler class admits a steady K\"ahler-Ricci soliton. 

\end{example}

We conclude this section by discussing another class of examples with $D$ a K3-surface and $\gamma$ an antisymplectic involution. 
Explicit examples of such K3-surfaces can for instance be obtain form the Kummer's construction.

\begin{example}
	Let $D$ be a K3-surface together with  a trivialisation $\Omega_D$ of the canonical bundle. Suppose that $\gamma_D$ is a holomorphic involution on $D$ such that
	\begin{align*}
		\gamma_D^ *\Omega_D  =- \Omega_D.
	\end{align*}
Also assume that the fixed point set $\operatorname{Fix}(\gamma_D)$ is non-empty. This implies that $\operatorname{Fix}(\gamma_D)$ is the disjoint union of smooth, complex curves. (In fact, there is a classification for all possibilities of $\operatorname{Fix}(\gamma_D)$, compare \cite{nikulin1983factor}.)

At any $p\in \operatorname{Fix}(\gamma_D)$, we may linearise  $\gamma_D$ so that its action in a suitable chart is given by  
\begin{align}\label{in example K3: Z2 action}
	\begin{split}
	\mathbb{C}^2&\to \mathbb{C}^2 \\
	(z_1,z_2) &\to (-z_1,z_2)
	\end{split}
\end{align}
In particular, the singular set of the orbifold $D/\langle \gamma_D \rangle$ locally corresponds to  $\{ z_1=0 \} $ inside $ \mathbb{C}^2 / \mathbb{Z}_2$, with   $\mathbb{Z}_2$-action defined by (\ref{in example K3: Z2 action}).

As in Theorem \ref{geometric existence theorem}, we let $\gamma: \mathbb{C} \times D \to \mathbb{C} \times D$ be 
\begin{align*}
	\gamma(z_0,z):= (-z_0,\gamma_D (z)).
\end{align*}
Then the singularities of $M_{orb}=(\mathbb{C}\times D)/ \langle \gamma \rangle $ are locally isomorphic to $\mathbb{C}^3/ \mathbb{Z}_2 \cong \mathbb{C}^2 / \mathbb{Z}_2 \times \mathbb{C}$, where $\mathbb{Z}_2$ acts by $-1$ in the first two factors, and trivially in the third one. 
This orbifold, however, admits a \textit{unique} crepant resolution
\begin{align}\label{in example K3: local resolution}
		\mathcal{O}_{\mathbb{CP}^1} (-2) \times \mathbb{C} \to \mathbb{C}^2/ \mathbb{Z}_2   \times \mathbb{C},
\end{align}
so that the local resolutions may be patched together to yield a crepant resolution $M \to M_{orb}$. 
Moreover, the $\mathbb{C}^*$-action by multiplication in the first factor extends to $M$, because this is clearly true for the local model (\ref{in example K3: local resolution}).

Since $H^1(D,\mathbb{R})=0$, we deduce that each K\"ahler class on $M$  admits a steady K\"ahler-Ricci soliton, thanks to Theorem \ref{geometric existence theorem} and Corollary \ref{corollary: H1=0 implies every kahler class admits a nice ACyl metric}.

 \end{example}

	\section{The Monge-Amp\`ere equation} \label{section monge ampere equation}

In this section, we present the proof of Theorem \ref{analytic existence theorem}. We consider a more general setting as in Theorem \ref{geometric existence theorem} in order to clarify  which assumptions are used for the a priori estimates below. The following list of properties is assumed throughout this section:

\begin{assumption} \label{section MA: assumptions on ACyl manifold}
	Let $(M,g)$ be an ACyl manifold of (real) dimension $2n$ in the sense of Definition \ref{definition ACyl manifold}.
	\begin{enumerate}[label=\textbf{A.\arabic*},ref=A.\arabic*]
	
		\item \label{section MA: Acyl metric shall be kahler} Suppose there exists a complex structure $J$ on $M$, so that $(M,g,J)$ is K\"ahler and denote the K\"ahler form by $\omega$.  
		
		\item \label{section MA: analytic + cpt support-- definition of X} There exists a real holomorphic vector field $X$ on $M$ such that 
		\begin{align*} 
		 X = 2 \Phi_* \frac{\partial}{\partial t},
		\end{align*}
		where $\Phi$ denotes the ACyl map  and $t$ the cylindrical coordinate function of $(M,g)$.
		
		\item \label{section MA: asympotitc behaviour of hamiltonian function} $JX$ is a Killing field of $g$. In particular, $\mathcal{L}_{JX} \omega =0$ and according to Proposition \ref{proposition when acyl metric is hamiltonian}, there exists a smooth $\tilde f:M\to \mathbb{R}$ such that $X= \nabla^g \tilde f$ and 
		\begin{align*}
		\tilde f- 2t \in C^{\infty}_\delta (M),
		\end{align*}
	where $\delta>0$ is the convergence rate of $(M,g)$ to its asymptotic model.
		We normalise the proper function $\tilde f$ by choosing a $c>0$ such that $f:=\tilde f +c \geq 1$ so that we still have $X= \nabla ^g f$. 
	\end{enumerate}
\end{assumption}

The reader may recall that the ACyl metric constructed in Section \ref{subsection: proof of geometric existence theorem} satisfies all of these requirements. 

We define new function spaces $C^{\infty}_{\varepsilon,JX}(M)$ consisting of all elements in $C^{\infty}_{\varepsilon}(M)$ which are $JX$-invariant, i.e.
\begin{align*}
	C^{\infty}_{\varepsilon,\,JX}(M) :=\left\{  u \in C^{\infty}_{\varepsilon}(M) \, | \, JX(u)=0   \right\}.
\end{align*}

Using this notation, the main result of this section is the next
\begin{theorem} \label{analytic existence theorem for compactly supported RHS}
	Let $(M,g)$ be an ACyl manifold of real dimension $2n$ satisfying the assumptions \ref{section MA: Acyl metric shall be kahler}, \ref{section MA: analytic + cpt support-- definition of X} and \ref{section MA: asympotitc behaviour of hamiltonian function}. Given $F\in C^{\infty}_{\varepsilon,\,JX}(M)$ for some $1<\varepsilon<2$, there exists a unique $\varphi \in C^{\infty}_{\varepsilon,\,JX}(M)$ such that $\omega + i \partial \partialb \varphi$ is K\"ahler and satisfies
	\begin{align}\label{section MA: analytic + epsilon RHS -- the MA equation}
	\left(   \omega+ i \partial \partialb   \varphi  \right) ^n = e^{F-\frac{X}{2}(\varphi)} \omega^n .
	\end{align} 
\end{theorem}

This theorem is analogue to \cite{conlon2020steady}[Theorem 7.1], and we also follow the same strategy as in \cite{conlon2020steady}[Section 7] to prove it, i.e.
  we set up  a continuity method.
 
  For given $k\in \mathbb{N}_0$, $\alpha \in (0,1)$ and $F\in C^{\infty}_{\varepsilon,\,JX}(M)$ with $1<\varepsilon<2$, we define the Monge-Amp\`ere operator on the set $\mathcal{U}$  containing all $\varphi \in  C^{k+2,\alpha}_{\varepsilon, \, JX}(M)$ with $\omega+ i \partial \partialb \varphi >0$ as follows:
\begin{align}
\label{monge ampere operator for continuity method}
	\begin{split}
	\mathcal{M} : \mathcal{U}  \times [0,1] &\to C^{k,\alpha}_{\varepsilon, \, JX}(M) \\
	(\varphi,s) &\mapsto \log \frac{(\omega+ i \partial \partialb\varphi)^n}{\omega^n} + \frac{X}{2} (\varphi) -s F
	\end{split}
\end{align}
It is worth mentioning that the function $\mathcal{M}(\varphi,s)$ is indeed $JX$-invariant since $ F$ is assumed to be invariant under $JX$, and also $\mathcal{L}_{JX} \omega =0$ by \ref{section MA: asympotitc behaviour of hamiltonian function}. 
Before applying the implicit function theorem, we need to compute the linearization  of $\mathcal{M}$, i.e. the derivative at the point $(\varphi,s)$ in direction of $(u,0)$:
\begin{align} \label{monge ampere operator derivative}
D \mathcal{M} _{(\varphi,s)} (u,0) = \frac{1}{2} \Delta_{g_{\varphi}} (u) + \frac{X}{2}(u).
\end{align}
Here $\Delta_{g_{\varphi}}$ denotes the Riemannian Laplace operator of the metric $g_{\varphi}$ associated to the K\"ahler form $\omega + i \partial \partialb \varphi$.

As in \cite{conlon2020steady}, the first step is to show that the
  linearized operator is an isomorphism $C^{k+2,\alpha}_{\varepsilon, \, JX}(M) \to C^{k,\alpha}_{\varepsilon, \, JX}(M)$, which is covered in the next

\begin{prop}\label{linearization is an isomorphism}
 Let $(M,g)$ be an ACyl manifold of real dimension $2n$ satisfying the assumptions \ref{section MA: Acyl metric shall be kahler}, \ref{section MA: analytic + cpt support-- definition of X} and \ref{section MA: asympotitc behaviour of hamiltonian function}. Given $k\in \mathbb{N}_0$, $\alpha\in (0,1)$ and $0<\varepsilon<2$, the operator 
 \begin{align*}
 	\Delta_g + X : C^{k+2,\alpha}_{\varepsilon, \, JX}(M) \to C^{k,\alpha}_{\varepsilon, \, JX}(M)
 \end{align*}
 is an isomorphism. 
\end{prop}

Here, our arguments differ from those in \cite{conlon2020steady}[Theorem 6.6], because the metrics we consider have a different asymptotic behavior.  Instead, we reduce the proof to Theorem \ref{theorem ACyl drift laplace is iso}.
\begin{proof}
	First, we observe by Assumption \ref{section MA: asympotitc behaviour of hamiltonian function}, that $\Delta_g +X$ is an ACyl drift operator in the sense of Definition \ref{definition ACyl drift Laplacian}. Thus, according to Theorem \ref{theorem ACyl drift laplace is iso}, the map 
	\begin{align*}
		\Delta_g +X : C^{k+2,\alpha}_{\varepsilon}(M) \to C^{k,\alpha}_{\varepsilon}(M)
	\end{align*}
	is an isomorphism for $k\in \mathbb{N}_0, \alpha\in (0,1)$ and $0<\varepsilon<2$. Consequently, it only remains to show that $u \in C^{k+2,\alpha}_{\varepsilon}(M)$ is $JX$-invariant, provided $(\Delta_g +X) (u)$ is invariant under $JX$. 
	To see this, we use that $X$ is real holomorphic and obtain
	\begin{align*}
	[X,JX]=J [X,X]=0,
	\end{align*}
	so that $JX(X(u))=X(JX(u))$. Moreover, we have $JX(\Delta_g u )= \Delta_g (JX(u))$ which follows directly from the relation 
	\begin{align*}
		\frac{1}{2} \Delta_{g} u \, \omega^n = n\, i \partial \partialb u \wedge \omega^{n-1}
	\end{align*}
	by applying $\mathcal{L}_{JX} \omega=0$. Hence, we conclude that if $(\Delta_g +X)(u)$ is $JX$-invariant for some $u \in C^{k+2,\alpha}_\varepsilon (M)$, then
	\begin{align*}
		0=JX( (\Delta_g +X)(u) ) = (\Delta_g +X)(JX(u)).
	\end{align*}
	As $|X|_g $ is bounded,  $JX(u) $ tends to $0$ as $t\to \infty$, and the maximum principle yields $JX(u)=0$, as desired. 
\end{proof}

\begin{remark}[on the decay rate  $\varepsilon$]\label{remark on rate of convergence}
	The reader may notice that Proposition \ref{linearization is an isomorphism} holds for all $0<\varepsilon<2$, whereas Theorem \ref{analytic existence theorem for compactly supported RHS} only includes the case $F\in C^{\infty}_{\varepsilon, \, JX}$ with $1<\varepsilon<2$. This is because Conlon and Deruelle's approach to the uniform $C^0$-estimate requires the convergence of certain weighted  functionals, compare Definition \ref{definition of I and J} below. 
	
	However, it seems plausible to use Theorem \ref{analytic existence theorem for compactly supported RHS} together with  ideas contained in \cite{conlon2020steady}[Section 9] to cover the case $0<\varepsilon\leq 1$ as well, but we do not pursue this further in this article.
\end{remark}

We also obtain the following regularity statement for the Monge-Amp\`ere operator.

\begin{prop}[Regularity] \label{proposition regularity}
Let $(M,g), F \in C^{\infty}_{\varepsilon, JX}(M)$ and $1<\varepsilon<2$ be as in Theorem \ref{analytic existence theorem for compactly supported RHS}. Suppose that  $\varphi \in C^{3,\alpha}_{\varepsilon', \, JX}(M)$  for some $0<\varepsilon'\leq \varepsilon$ satisfies 
\begin{align*}
(\omega+ i \partial \partialb \varphi)^n = e^{F- \frac{X}{2}(\varphi)} \omega^n.
\end{align*}
Then $\varphi \in C^{\infty}_{\varepsilon,\, JX}(M)$. 
\end{prop}

Note that this statement only gives \textit{qualitative} information about the function $\varphi$, i.e. it does \textit{not} provide uniform estimates for the $C^{\infty}_{\varepsilon}(M)$-norm of $\varphi$. The 
 crucial part of the continuity method, however, is precisely to obtain uniform  a priori bounds on  $||\varphi||_{C^{k,\alpha}_\varepsilon}$.  This is achieved in the next 

\begin{theorem}[A priori estimates]\label{theorem APRIORI estimates and regularity}
Let $(M,g)$, $F \in C^{\infty}_{\varepsilon, \, JX}(M)$ and $1<\varepsilon<2$ be as in Theorem \ref{analytic existence theorem for compactly supported RHS}. Suppose that $(\varphi_{ s})_{0\leq s\leq 1}$ is a family in
 $C^{\infty}_{\varepsilon, \, JX}(M)$ 
such that $\omega + i \partial \partialb \varphi_{ s}$ is  K\"ahler for each $s\in [0,1]$ and satisfies
	\begin{align}\label{one-parameter family of MA}
	\left(   \omega+ i \partial \partialb   \varphi_s  \right) ^n =  e^{s\cdot F-\frac{X}{2}(\varphi_s)} \omega^n.
\end{align}
Then, for given $k\in \mathbb{N}_0, \alpha \in (0,1)$, there exists a constant $C>0$ such that
\begin{align*}
	\sup _{s\in [0,1]} ||\varphi_s ||_{C^{k,\alpha}_\varepsilon } \leq C,
\end{align*}
where $C$ only depends on $k,\alpha, F$ and the geometry of $(M,g)$. 
\end{theorem}

The strategy for proving Proposition \ref{proposition regularity} and Theorem \ref{theorem APRIORI estimates and regularity} is to follow, up to some minor adjustments, the arguments provided by Conlon and Deruelle (\cite{conlon2020steady}[Section 7]). In particular, we use  their idea to achieve the uniform $C^0$-bound, but we present a variation of their arguments which allows us to immediately assume $F\in C^{\infty}_{\varepsilon, \, JX}(M)$ with $1<\varepsilon<2$, instead of first considering functions $F$ with compact support as in \cite{conlon2020steady}[Theorem 7.1].

We postpone the proofs of both Proposition \ref{proposition regularity} and Theorem \ref{theorem APRIORI estimates and regularity} to subsequent sections and for now assume these results to conclude Theorem \ref{analytic existence theorem for compactly supported RHS}.

\begin{proof}[Proof of Theorem \ref{analytic existence theorem for compactly supported RHS}]
First, we point out that we only need to show the existence statement since the  uniqueness part is a direct consequence of the maximum principle, see \cite{Biquard17}[Proposition 1.2]. 

For the proof of existence, assume we are  given $F\in C^{\infty}_{\varepsilon, \, JX}(M)$, and  consider the set 
\begin{align*}
	S:= \{ s\in [0,1] \,|\, \text{there exists a } \varphi_{ s} \in C^{\infty}_{\varepsilon, \, JX}(M) \text{ satisfying } (\ref{one-parameter family of MA})  \}.
\end{align*}
Clearly, $0 \in S$ and so it is sufficient to show that $S$ is both open and closed.

The openness is a consequence of Proposition \ref{linearization is an isomorphism}. To see this,   let  $\mathcal{U}$ be the set   of all $\psi \in C^{3,\alpha}_{\varepsilon, \, JX}(M)$ such that $\omega + i \partial \partialb \psi>0$ and consider the Monge-Amp\`ere operator $\mathcal{M}$ defined by 
\begin{align*}
	\mathcal{M} : \mathcal{U}  \times [0,1] &\to C^{1,\alpha}_{\varepsilon, \, JX}(M) \\
	(\psi,s) &\mapsto \log \frac{(\omega+ i \partial \partialb\psi)^n}{\omega^n} + \frac{X}{2} (\psi) -s F
\end{align*}
Suppose we are given $s_0 \in S$, i.e. $\varphi_{s_0} \in C^{\infty}_{\varepsilon, \, JX}(M)$ solving (\ref{one-parameter family of MA}). Since $\varphi_{ s_0}$ is $JX$-invariant and $\varphi_{ s_0}\in C^{\infty}_{\varepsilon}(M)$, the Riemannian metric $g_{\varphi_{ s_0}}$ corresponding to $\omega + i \partial \partialb \varphi_{ s_0}$ is ACyl, with the same ACyl map as $g$, and satisfies Assumptions \ref{section MA: Acyl metric shall be kahler}, \ref{section MA: analytic + cpt support-- definition of X} and \ref{section MA: asympotitc behaviour of hamiltonian function}. Hence, the linearization of $\mathcal{M}$ at the point $(\varphi_{ s},s)$, which is given by (\ref{monge ampere operator derivative}), is  injective if restricted  to the subspace $C^{3,\alpha}_{\varepsilon,\, JX}(M)$ and also surjective according to Proposition \ref{linearization is an isomorphism}.
Thus, the implicit function theorem implies the existence of a $\delta_0>0$ such that for all $s\in (s_0 -\delta_0, s_0+\delta_0)$ there exists a $\varphi_{ s}\in C^{3,\alpha}_{\varepsilon} (M)$ solving (\ref{one-parameter family of MA}). But then $\varphi_{ s} \in C^{\infty}_{\varepsilon, \, JX}(M)$ by Proposition \ref{proposition regularity}, and consequently $(s_0-\delta_0 , s_0 + \delta_0) \cap [0,1] \subset S$. 

That $S$ is closed follows from Theorem \ref{theorem APRIORI estimates and regularity}. Indeed, consider a sequence $(s_k)_{k\in \mathbb{N}}$ in $S$ which converges to some $s_{\infty}\in \mathbb [0,1]$, and denote the corresponding sequence in $C^{\infty}_{\varepsilon, \, JX}(M)$ of solutions to (\ref{one-parameter family of MA})  by $(\varphi_{ s_k})$. 
According to Theorem \ref{theorem APRIORI estimates and regularity}, this sequence $(\varphi_{ s_k})$ is uniformly bounded in $C^{3,\alpha}_{\varepsilon}(M)$. Choosing $\varepsilon' \in (0,\varepsilon)$ and $\beta \in (0,\alpha)$, the inclusion $C^{3,\alpha}_{\varepsilon} (M ) \subset C^{3,\beta}_{\varepsilon'}(M)$ is compact (by \cite{marshallphd}[Theorem 4.3] for instance), so that we can extract a subsequence of $(\varphi_{ s_k})$ converging in $C^{3,\beta}_{\varepsilon'}(M)$ to some limit $\varphi_{s_\infty} \in C^{3,\beta}_{\varepsilon'}(M)$.  Note that we must have $JX(\varphi_{s_\infty}) =0$ and that $\varphi_{s_\infty}$ satisfies
\begin{align*}
	(\omega+ i \partial \partialb \varphi_{s_\infty}) ^n = e^{s_{\infty} F - \frac{X}{2}(\varphi_{s_\infty})} \omega^n,
\end{align*}
as we can take the point-wise limit $k\to \infty$ in (\ref{one-parameter family of MA}). From this, we immediately see that $\omega + i \partial \partialb \varphi_{s_\infty}$ is a K\"ahler form, and applying Proposition \ref{proposition regularity} then implies $\varphi_{s_\infty} \in C^{\infty}_{\varepsilon, \,  JX}(M)$, i.e. $s_\infty \in S$. This concludes the proof. 
\end{proof}

The rest of this section is devoted to proving Proposition \ref{proposition regularity} and Theorem \ref{theorem APRIORI estimates and regularity}. We begin in Section \ref{subsection: C0 estimate} by deriving the $C^0$-estimate which is the key part of the proof. Then we move on to higher-order estimates in Section \ref{subsection: higher order estimates} to finish the proof of Theorem \ref{theorem APRIORI estimates and regularity}. Afterwards, we conclude by verifying Proposition \ref{proposition regularity}.

\subsection{The $C^0$-estimate} \label{subsection: C0 estimate}
Throughout this section, let $(M,g)$ satisfy Assumptions \ref{section MA: Acyl metric shall be kahler}, \ref{section MA: analytic + cpt support-- definition of X} and \ref{section MA: asympotitc behaviour of hamiltonian function}. The goal is to obtain uniform estimates for solutions $(\varphi_{ s})_{0\leq s\leq 1}$ to (\ref{one-parameter family of MA}), among which the $C^0$-bound is the most difficult one to achieve.

The proof of the $C^0$-estimate  is split into three parts: First, we obtain a weighted upper bound on $ \varphi_{s }$, then an $L^2$-bound with a certain weight and finally, we can conclude a lower bound on $\inf_M \varphi_{ s}$. The last two steps closely follow the ideas developed in \cite{conlon2020steady}[Section 7.1]. Before beginning with the preparations, let us fix some notation.

\begin{notation}
	We denote the metric associated with $\omega + i \partial \partialb \varphi_s$ by $g_{\varphi_s}$, and $\nabla^{g_{\varphi_s}}$, $\Delta_{g_{\varphi_s}}$, etc. denote the Levi-Civita connection, the Laplace operator, etc. of $g_{\varphi_s}$. We point out that $\Delta_{g_{\varphi_s}}$ is the \textit{Riemannian} Laplace operator, i.e. it satisfies
	\begin{align}\label{riemannian laplace operator kähler relation}
		\frac{1}{2} \Delta_{g_{\varphi_s}} u \, \omega_{\varphi_s}^n = n\, i \partial \partialb u \wedge \omega_{\varphi_s}^{n-1}
	\end{align}
	for each $C^2$-function $u$. 
\end{notation}

\subsubsection{An upper bound on $\varphi_s$}
We begin by estimating $\varphi_s$ from above:
\begin{prop}[Weighted upper bound on $\varphi_s$]\label{upper-C^0 estimate for varphi_s from above}
	Let $1<\varepsilon <2$ and suppose $(\varphi_s)_{0\leq s\leq 1}$ is  a family in $C^{\infty}_{\varepsilon, \, JX}(M)$ solving  (\ref{one-parameter family of MA}). Then there exists a constant $C>0$ such that 
	\begin{align*}
		\sup_{s\in[0,1]}\sup_M e^{\varepsilon t}\varphi_s \leq C,
	\end{align*}
	where $C$ only depends on  $F\in C^{\infty}_{\varepsilon, JX}(M)$ and the geometry of $(M,g)$.
\end{prop}

We  present a proof  based on the use of a  barrier function, so our argument differs from the one given in \cite{conlon2020steady}[Proposition 7.9].

\begin{proof}
	
	We begin by observing that  $\varphi_s$ satisfies
	\begin{align}\label{in proposition on C0 upper bound: inequality for Delta_g varphi}
	\frac{1}{2}\Delta_g (\varphi_s) + \frac{X}{2} (\varphi_s) \geq sF.
	\end{align}
Indeed, consider any $p\in M$ and  holomorphic coordinates $(z_1,\dots, z_n)$ such that
\begin{align*}
	g_{i\bar j} = \delta_{i\bar j} \;\; \text{ and } \;\; \frac{\partial ^2 u}{\partial z_i \partial \bar z_j}   = \lambda _i \delta_{i\bar j} \;\; \text{ at } \;\;p
\end{align*}
for some $\lambda_i \in \mathbb{R}$ with $1+\lambda_i>0$, where $g_{i\bar j}$ are the local components of $g$ and $\delta_{i\bar j}$ denotes the Kronecker delta. Starting from (\ref{one-parameter family of MA}), we compute at $p$ that
\begin{align*}
s F -\frac{X}{2}(\varphi_s)&=\log  \frac{ \left(\omega + i \partial \partialb \varphi_{ s}   \right)^n}{\omega ^n } \\
&= \log (1+\lambda_1) \cdots (1+\lambda_n) \\
&= \sum_{j=1}^n \log (1+\lambda_j) \\
&\leq \sum_{j=1}^n \lambda_j \\
&= \operatorname{tr}_\omega (i \partial \partialb \varphi_{ s})=\frac{1}{2} \Delta_{g} (\varphi_s),
\end{align*}
where $\operatorname{tr}_\omega (i \partial \partialb \varphi_{ s})$ denotes the trace of $i \partial \partialb \varphi_{ s}$ with respect to $\omega$ and
 we used $\log (1+\tau)\leq \tau $ if $\tau >-1$ to obtain the inequality in the fourth line. This finishes the proof of  (\ref{in proposition on C0 upper bound: inequality for Delta_g varphi}).

Moreover, since $F\in C^{\infty}_{\varepsilon}(M)$ with $0<\varepsilon<2$ and because of Assumption \ref{section MA: asympotitc behaviour of hamiltonian function}, Theorem \ref{theorem ACyl drift laplace is iso} implies the existence of a  function $u_F \in C^{\infty}_{\varepsilon}(M)$ such that
\begin{align*} %\label{barrier function differential inequality}
	\frac{1}{2} \Delta_g (u_F) + \frac{X}{2}(u_F) = F,
\end{align*}
which, in combination with (\ref{in proposition on C0 upper bound: inequality for Delta_g varphi}),  leads to
\begin{align*}
	(\Delta_g + X)(\varphi_{ s}  - su_F) \geq 2 s(F-F)=0. 
\end{align*}
Choosing a sequence $(t_k)_{k\in \mathbb{N}}$ with $t_k \to \infty$ and applying 
   Hopf's maximum principle to a sequence of domains of the form $\{ t \leq t_k\} \subset M$ then yields
\begin{align*}
	\sup_M (\varphi_{ s} - s u_F) \leq \lim_{t\to \infty} (\varphi_{ s} - s u_F) =0,
\end{align*}
i.e. $\varphi_{ s} \leq s u_F$ holds on all of $M$. In particular, we observe that
\begin{align*}
	e^{\varepsilon t} \varphi_{ s} \leq su_F e^{\varepsilon t} \leq ||u_F||_{C^0 _\varepsilon} =:C,
\end{align*}
which proves the claim. 
\end{proof}

For obtaining a lower bound on $\varphi_s$, we need to work considerably harder. The important idea in \cite{conlon2020steady}  is to first obtain a weighted $L^2$-bound.

\subsubsection{A weighted $L^2$-bound} As in \cite{conlon2020steady}[Subsection 7.1.1.], we consider two functionals which were used by Tian and Zhu \cite{tian2000uniqueness} to study shrinking K\"ahler-Ricci solitons on compact Fano manifolds.

\begin{defi}\label{definition of I and J}
	Consider $1<\varepsilon<2$ and let $(\psi_\tau)_{0\leq \tau \leq 1}$ be  a $C^1$-path in $C^{\infty}_{\varepsilon,\, JX}(M)$ from $\psi_0=0$ to $\psi_1=\psi$ and assume  for each $\tau \in [0,1]$ that $\omega_{\psi_{\tau}}:=\omega+ i \partial \partialb \psi_\tau >0$.  Define: 
	\begin{align*}
		I_{\omega, \,X}(\psi)&:= \int_M \psi \left(e^f\omega^n - e^{f + \frac{X}{2}(\psi)} \omega^n _{\psi}     \right),\\
		J_{\omega, \, X}(\psi)	&:= \int_0 ^1 \int_M \dot{\psi_{\tau}} \left(e^f \omega ^n - e^{f+\frac{X}{2}(\psi_\tau)} \omega_{\psi_{\tau}} ^n    \right) \wedge d \tau,
	\end{align*} 
	where $\dot{\psi_{\tau}}= \frac{\partial}{\partial \tau} \psi_{\tau}$. 
\end{defi}

Since $M$ is non-compact, we need to show that $I_{\omega, \, X}$ and $J_{\omega, \, X}$ are well-defined, i.e. that the resulting integrals are finite. Given $\psi \in C^{\infty}_{\varepsilon,\, JX}(M)$ with $1<\varepsilon<2$, we deduce from (\ref{section MA: asympotitc behaviour of hamiltonian function}) that $\psi \, e^{f}=O(e^{(2-\varepsilon)t})$, so it suffices to show
\begin{align}\label{subsection 4.1.1: I is finite}
	|\omega^n  - e^{\frac{X}{2}(\psi)}\omega_{\psi}^n |_g =O(e^{-\varepsilon t}) 
\end{align}
since $\varepsilon>1$. To see that this is true, we expand $\omega_\psi ^n $ and obtain
\begin{align*}
	\omega^n  - e^{\frac{X}{2}(\psi)}\omega_{\psi}^n = \left( 1-e^{\frac{X}{2}(\psi)}  \right) \omega^n - e^{\frac{X}{2}(\psi)} \sum_{k=1}^{n} \binom{n}{k} \left( i \partial \partialb \psi   \right) ^k \wedge \omega^{n-k}
\end{align*}
from which (\ref{subsection 4.1.1: I is finite}) follows because $\frac{X}{2}(\psi)=O(e^{-\varepsilon t})$ and $|i \partial \partialb \psi|_g =O(e^{-\varepsilon t})$ by definition of $C^{\infty}_\varepsilon(M)$. Thus $I_{\omega, \, X} (\psi)$ is finite, and the same argument also proves that $J_{\omega,\, X}$ is well-defined.
The crucial starting point is the next
\begin{theorem} \label{theorem functional J is independent of path}
	Let $(\psi_{\tau})_{0\leq \tau \leq 1}$  be a $C^1$-path as in Definition \ref{definition of I and J}. Then the first variation of the difference $I_{\omega, \, X} - J_{\omega, \, X}$ is given by 
	\begin{align*}
		\frac{d}{d\tau} (I_{\omega, \, X} - J_{\omega, \, X})(\psi_{\tau})=- \int_M \psi_{\tau} \left( \frac{1}{2} \Delta_{g_{\psi_{\tau}}} (\dot{\psi}_\tau) + \frac{X}{2} (\dot \psi_\tau)  \right) e^{f+ \frac{X}{2}(\psi_{\tau})} \omega^n _{\psi_{\tau}},
	\end{align*}
	where $g_{\psi_{\tau}}$ is the metric with K\"ahler form $\omega_{\psi_{\tau}}= \omega + i \partial \partialb \psi_{\tau}$. Moreover, $J_{\omega,\,X}$ does not depend on the choice of path $(\psi_{\tau})_{0\leq \tau \leq 1}$, but only on the end points $\psi_0=0$ and $\psi_1= \psi$. 
\end{theorem}

\begin{proof}
 This is \cite{conlon2020steady}[Theorem 7.5], whose proof in turn relies on \cite{tian2000uniqueness}. The reader may observe that this proof is a completely formal calculation, which applies word-by-word to our case if Stokes theorem holds. This, however, is only used once on \cite{conlon2020steady}[p. 50]. Given our asymptotics, it is clear from Lemma \ref{lemma integration by parts} that we as well can integrate by parts because the integrands decay exponentially in the parameter $t$.  
\end{proof}

Before we can continue with the weighted $L^2$ bounds, we need another  lemma as preparation.

\begin{lemma}[A first  bound on $\inf _M X(\varphi_s)$]   \label{lemma rough lower bound on X(varphi) in terms of f}
	Let $1<\varepsilon <2$ and suppose $(\varphi_s)_{0\leq s\leq 1}$ is  a family in $C^{\infty}_{\varepsilon, \, JX}(M)$ solving  (\ref{one-parameter family of MA}). Then there exists a constant $C>0$ such that
	\begin{align}\label{in lemma a first bound on X(varphi)}
		\inf_{s\in[0,1]} \inf _M \left( f + \frac{X}{2}(\varphi_s)   \right)\geq 1,
	\end{align}
	where $C$ only depends on the geometry of $(M,g)$. 
\end{lemma}

\begin{proof}
	Since both $f$ and $\varphi_s$ are $JX$-invariant, the argument \cite{conlon2020steady}[(7.6)] applies and we obtain that 
	\begin{align}\label{in Lemma: a first bound on X(varphi); X is gradient}
		X= \nabla^{g_{\varphi_s}} \left( f+ \frac{X}{2} (\varphi_s)   \right).
	\end{align}
	Also observe that $\frac{X}{2}(\varphi_s) \to 0$ as $t\to \infty$ because $X$ is bounded with respect to the norm $g_{\varphi_s}$. Thus, we conclude from (\ref{section MA: asympotitc behaviour of hamiltonian function}) that $f+\frac{X}{2}(\varphi_s)$ converges to the function $2t+c$ with $c>0$ and consequently, $f+\frac{X}{2}(\varphi_s)$ attains a global minimum at some point $p\in M$. By (\ref{in Lemma: a first bound on X(varphi); X is gradient}), we see that $X$ must vanish at $p$,  so we conclude that
	\begin{align*}
		\inf_M \left(f + \frac{X}{2} (\varphi_s)\right) = \min_{\{  X=0\}  } \left(f + \frac{X}{2} (\varphi_s)\right) =  \min_{\{  X=0\}  } f
	\end{align*}
	holds for all $s\in [0,1]$. In particular, (\ref{in lemma a first bound on X(varphi)}) follows since we normalised $f$ such that $f\geq 1$ on $M$. 
\end{proof}

\begin{prop}[A priori bound on weighted energy] \label{weighted L^2 estimate}
	Let $1<\varepsilon <2$ and suppose $(\varphi_s)_{0\leq s\leq 1}$ is  a family in $C^{\infty}_{\varepsilon, \, JX}(M)$ solving  (\ref{one-parameter family of MA}).  Then there exists a constant $C>0$  such that
\begin{align}
	\sup _{0\leq s \leq 1} \int _M |\varphi_s|^2 \frac{e^f}{f^2} \operatorname{dV}_g \leq C ,
\end{align}
where $C$ only depends on $F\in C^{\infty}_{\varepsilon,\,JX}(M)$ and on the geometry of $(M,g)$.
\end{prop}
\begin{proof}
	We follow \cite{conlon2020steady}[Proposition 7.7]. The idea is to consider two different paths in $C^\infty_{\varepsilon,JX}(M)$ with $1<\varepsilon<2$ and to use Theorem \ref{theorem functional J is independent of path} for obtaining the required bound.
	
	We begin by considering a linear path from $0$ to $\varphi_s$. Given $s\in [0,1]$, define this path $(\psi_{\tau})_{0\leq \tau \leq 1}$ by 
	$\psi_{\tau}:= \tau \varphi_s $. Since $\omega + i \partial \partialb \psi_\tau >0$, Theorem \ref{theorem functional J is independent of path} implies that 
	\begin{align}
	(	I_{\omega,X}-J_{\omega,X} )(\varphi_s)= - \int _{0}^1 \int _M \frac{\tau \varphi_s}{2} \left( \Delta_{g_{\tau\varphi_s}} + X  \right)(\varphi_s) e^{f+ \tau\frac{X}{2}(\varphi_s) } \omega^n _{\tau \varphi_s}\wedge d\tau
	\end{align}
	Recalling that $X= \nabla^{g_{\tau\varphi_s}} (f+ \frac{X}{2}(\varphi_s)) $, we integrate by parts and obtain
	\begin{align} \label{in proposition energy bound: I-J estimated from below}
&(	I_{\omega,X}-J_{\omega,X} )(\varphi_s) = n \int _0^1 \int _M \tau e^{f+\tau \frac{X}{2}(\varphi_s)} i \partial \varphi_s  \wedge \partialb \varphi_s \wedge \omega^{n-1}_{\tau \varphi_s} \wedge d\tau \\ \nonumber
=& n \int _0 ^1 \int _M \tau e^{f+\tau \frac{X}{2}(\varphi_s)} i \partial \varphi_s  \wedge \partialb \varphi_s \wedge ((1-\tau)\omega+ \tau \omega_{\varphi_s})^{n-1} \wedge d\tau \\ \nonumber
\geq & n \int _0^1 \int _M \tau (1-\tau)^{n-1}  e^{f+\tau \frac{X}{2}(\varphi_s)} i \partial \varphi_s  \wedge \partialb \varphi_s \wedge \omega^{n-1} \wedge d\tau \\ \nonumber
\geq & n \int _0^1 \int _M \tau (1-\tau)^{n-1}  e^{(1-\tau)f} i \partial \varphi_s  \wedge \partialb \varphi_s \wedge \omega^{n-1} \wedge d\tau  \\ \nonumber
 =&n  \int _M \left(\int _0^1  \tau (1-\tau)^{n-1}  e^{(1-\tau)f} d\tau\right) \wedge i \partial \varphi_s  \wedge \partialb \varphi_s \wedge \omega^{n-1} ,
	\end{align}
	where the penultimate line holds since $\frac{X}{2}(\varphi_s) \geq - f$ by Lemma \ref{lemma rough lower bound on X(varphi) in terms of f}. Thanks to \cite{conlon2020steady}[Claim 7.8], there exists a constant $C>0$ such that 
	\begin{align}\label{in proposition energy bound: tau and f integral from below}
n \int _0^1  \tau (1-\tau)^{n-1}  e^{(1-\tau)f} d\tau \geq C \frac{e^f}{f^2},
	\end{align}
	which, in combination with (\ref{in proposition energy bound: I-J estimated from below}), then leads to 
	\begin{align} \label{in proposition engery estimate: I-J final lower bound}
		(	I_{\omega,X}-J_{\omega,X} )(\varphi_s) \geq  C \int _M \frac{e^f}{f^2}i \partial \varphi_s  \wedge \partialb \varphi_s \wedge \omega^{n-1}.
	\end{align}
	To estimate $(	I_{\omega,X}-J_{\omega,X} )(\varphi_s)$ from above, we recall from Theorem \ref{theorem functional J is independent of path} that $	J_{\omega,X}$ is independent of the choice of  path from $0$ to $\varphi_s$. Thus, we can compute $(	I_{\omega,X}-J_{\omega,X} )(\varphi_s)$ by defining a new path $(\psi_\tau )_{0\leq \tau \leq 1}$ as $\psi_\tau:= \varphi_{\tau s}$. We point out that $\psi_0= \varphi_0\equiv 0$ follows from the maximum principle applied to the Monge-Amp\`ere equation (\ref{one-parameter family of MA}). For calculating $\dot\psi_{\tau}$, differentiate (\ref{one-parameter family of MA}) with respect to $s$ and obtain
	\begin{align*}
	n \,i \partial \partialb \dot{\varphi}_s \wedge \omega_{\varphi_{ s}}^{n-1} = \left(F-\frac{X}{2}(\dot{\varphi}_s)   \right) \omega_{\varphi_s} ^n.
	\end{align*}
	Combining with (\ref{riemannian laplace operator kähler relation}) and using $\dot \psi _\tau = s\dot \varphi_{\tau s}$, we arrive at 
	\begin{align*}
		\frac{1}{2}\Delta_{\psi_{\tau}} \dot \psi_{\tau} + \frac{X}{2} (\dot \psi_{\tau}) = s F,
	\end{align*}
 to which we further apply Theorem (\ref{theorem functional J is independent of path}) and continue:
 \begin{equation}\label{in proposition energy estimate: I_J from above}
 \begin{aligned} \nonumber
  (I_{\omega, \, X} - J_{\omega, \, X})(\varphi_s)&=-\int_0^1 \int_M \psi_{\tau} \cdot sF e^{f+ \frac{X}{2}(\psi_{\tau})} \omega^n _{\psi_{\tau}} \wedge d\tau  \\ 
  &=- \int_0 ^1 \int_M \psi_\tau \cdot  sF e^{f+\tau s F} \omega^n \wedge d\tau \\ 
  & \leq  \int_0 ^1 \int_M |\psi_\tau|    |F| e^{f+ |F|} \omega^n \wedge d\tau \\
  & = \int_0 ^1 \int_M   f |F| e^{\frac{f}{2}+ |F|}  \cdot  |\psi_\tau| \frac{e^{\frac{f}{2}}}{f} \omega^n \wedge d\tau  \\
  &\leq C \int_0 ^1 \left( \int_M |\psi_\tau |^2 \frac{e^f}{f^2} \omega^n \right)^{\frac{1}{2}}  d\tau.
 \end{aligned}
 \end{equation}
 Here, we applied (\ref{one-parameter family of MA}) in the second line, Cauchy-Schwarz in the last one and the uniform constant $C>0$ is given by  $C^2=\int_M   f^2 |F|^2 e^{f+ 2|F|} \omega^n $, which is finite since $f^2e^f =O(t^2 e^{2t})$ and $F^2=O(e^{-2\varepsilon t})$ with $\varepsilon>1$. 
  From the previous estimate together with (\ref{in proposition engery estimate: I-J final lower bound}), we thus conclude
\begin{equation} \nonumber
 \begin{aligned}
 \int _M |\nabla^g \varphi_s |^2 _g \frac{e^f}{f^2} \operatorname{dV}_g &\leq C \int _0 ^1\left( \int _M |\varphi_{\tau s}|^2\frac{e^f}{f^2} \operatorname{d V}_g\right)^{\frac{1}{2}} d\tau\\
 &= \frac{C}{s} \int _0 ^s \left( \int _M |\varphi_{\tau }|^2\frac{e^f}{f^2} \operatorname{d V}_g\right)^{\frac{1}{2}} d\tau.
 \end{aligned}
 \end{equation}
 Together with Proposition \ref{poincare inequality}, we finally arrive at
 \begin{align}\label{in proposition energy estimate: grönwall inequality}
\lambda \int _M |\varphi_{s }|^2\frac{e^f}{f^2} \operatorname{d V}_g \leq \frac{C}{s} \int _0 ^s \left( \int _M |\varphi_{\tau }|^2\frac{e^f}{f^2} \operatorname{d V}_g\right)^{\frac{1}{2}} d\tau. 
 \end{align}
 As observed by Conlon and Deruelle \cite{conlon2020steady}[Proposition 7.7], this is a Gr\"onwall-type differential inequality for the function $U:(0,1] \to \mathbb{R}_+$ defined by 
 \begin{align*}
 	U(s):= \int _0 ^s \left( \int _M |\varphi_{\tau }|^2\frac{e^f}{f^2} \operatorname{d V}_g\right)^{\frac{1}{2}} d\tau.
 \end{align*}
 Indeed, it is immediate that (\ref{in proposition energy estimate: grönwall inequality}) becomes
 \begin{align*}
 	\frac{\dot U(s)}{\sqrt{U(s)}} \leq \frac{C}{\sqrt{s}},
 \end{align*}
 so that we integrate to obtain $\sqrt{U(s)}\leq C\sqrt{s}$ with $s\in (0,1]$. Hence,
 \begin{align*}
 	 \left( \int _M |\varphi_{s }|^2\frac{e^f}{f^2} \operatorname{d V}_g\right)^{\frac{1}{2}} d\tau = \dot{U}(s) \leq C,
 \end{align*}
 where $C=C(M,g,F)$ is independent of $s\in [0,1]$, as claimed. 
\end{proof}

\subsubsection{A lower bound on $\varphi_{s }$} For proving a uniform bound on $\sup_M |\varphi_s|$, it remains to bound $\inf_M \varphi_{s }$ from below. This is the main result of this subsection:

\begin{prop}[Lower bound on $\inf_M \varphi_{s }$]\label{proposition: lower bound on inf varphi}
	Let $1<\varepsilon <2$ and suppose $(\varphi_s)_{0\leq s\leq 1}$ is  a family in $C^{\infty}_{\varepsilon, \, JX}(M)$ solving  (\ref{one-parameter family of MA}). Then there exists a constant $C>0$ such that
	\begin{align*}
		\inf_{s\in [0,1]} \inf _M \varphi_{s } \geq -C,
	\end{align*}
	where $C$ only depends $F \in C^{\infty}_{\varepsilon,\, JX}(M)$ and on the geometry of $(M,g)$.
\end{prop}

If we assumed that $F$ was compactly supported, the same argument as in \cite{conlon2020steady}[Proposition 7.10] would go through verbatim and provide the required bound on $\inf_M \varphi_s$, since we already obtained  uniform bounds on $\sup _M \varphi_{s } $ (Proposition  \ref{upper-C^0 estimate for varphi_s from above}) and on the weighted $L^2$-norm (Proposition \ref{weighted L^2 estimate}).

In our situation, however, we do \textit{not} assume that $F$ has compact support, but merely $F\in C^{\infty}_{\varepsilon, \,JX}(M)$ with $1<\varepsilon<2$. Thus, we proceed as follows.

First, we construct a compact domain $K\subset M$ so that we obtain a suitable barrier function on its complement $M\setminus K$, which will be useful for arguments relying on the maximum principle. 
In a second step, the argument in  \cite{conlon2020steady}[Proposition 7.10] gives a lower bound on 
$\inf_K \varphi_s$. And finally, we will see that the maximum principle yields a lower bound on $\inf_{M\setminus K} \varphi_{ s}$.

In other words, our strategy is to prove the following lemma, as well as the next two propositions:

\begin{lemma}[Construction of $K$] \label{lemma: construction of compactum} 
Let $1<\varepsilon <2$ and suppose $(\varphi_s)_{0\leq s\leq 1}$ is  a family in $C^{\infty}_{\varepsilon, \, JX}(M)$ solving  (\ref{one-parameter family of MA}). Then there exists a constant $0<\varepsilon_0<1$ and a compact domain $K\subset M$ such that for all $s\in [0,1]$, we have
\begin{align} \label{in proposition lower weighted bound: inequality for weight function}
	\left( \Delta_{g_{\varphi_s}} +X  \right) \left(e^{-\varepsilon_0 \left(  f+ \frac{X}{2}(\varphi_{ s})  \right)} \right) \leq - \frac{\varepsilon_0}{2} e^{-\varepsilon_0 \left(  f+ \frac{X}{2}(\varphi_{ s})  \right)} <0 \;\; \text{ on } \;\; M\setminus K,
\end{align}
where both $\varepsilon_0$  and $K$ only depend on $F\in C^{\infty}_{\varepsilon, \, JX}(M)$ and the geometry of $(M,g)$.
\end{lemma}

\begin{prop}[Lower bound on a compact set]\label{prop: lower bound for varphi on compact sets}
	Let $1<\varepsilon <2$ and suppose $(\varphi_s)_{0\leq s\leq 1}$ is  a family in $C^{\infty}_{\varepsilon, \, JX}(M)$ solving  (\ref{one-parameter family of MA}). For the compact domain $K\subset M$ given by Lemma \ref{lemma: construction of compactum},  there exists a constant $C>0$ such that
	\begin{align*}
		\inf_{s\in [0,1]} \inf _K \varphi_{s } \geq -C,
	\end{align*}
	where $C$ only depends on $K$, $F \in C^{\infty}_{\varepsilon,\, JX}(M)$ and  the geometry of $(M,g)$.
\end{prop}

\begin{prop}[Lower bound outside of a compact set] \label{prop: lower bound OUTSIDE a compact set}
Let $1<\varepsilon <2$ and suppose $(\varphi_s)_{0\leq s\leq 1}$ is  a family in $C^{\infty}_{\varepsilon, \, JX}(M)$ solving  (\ref{one-parameter family of MA}). For the compact domain $K\subset M$ constructed in Lemma \ref{lemma: construction of compactum}, there exists a constant $C>0$ such that
\begin{align*}
	\inf_{s\in[0,1]} \inf_{M\setminus K} \varphi_{ s} \geq -C,
\end{align*}
where  $C$ only depends on $K$, $F \in C^{\infty}_{\varepsilon,\, JX}(M)$ and  the geometry of $(M,g)$.
\end{prop}

Clearly, Proposition \ref{prop: lower bound OUTSIDE a compact set}, together with Proposition \ref{prop: lower bound for varphi on compact sets},  yield a uniform lower bound on $\inf _M \varphi_{ s}$, as claimed in Proposition \ref{proposition: lower bound on inf varphi}.

Since  Lemma \ref{lemma: construction of compactum}  requires some preparation, let us for the moment assume that we are given the compact set $K\subset M$ from Lemma \ref{lemma: construction of compactum} and see how this implies the lower bound on $\inf_K \varphi_{ s}$, i.e. Proposition \ref{prop: lower bound for varphi on compact sets}.

\begin{proof}[Proof of Proposition \ref{prop: lower bound for varphi on compact sets}]
	We follow the proof of \cite{conlon2020steady}[Proposition 7.10], which in turn relies on B\l ocki's  local argument \cite{blocki2005uniform}.

	Let $K\subset M$ be the compact domain constructed in Lemma \ref{lemma: construction of compactum}.
	 For each  $p\in K$, let $V$ be a chart around $p$ so that $\omega$ can be written as $\omega= i \partial \partialb G$. According to the proof of \cite{blocki2005uniform}[Theorem 4],  there are constants $a,r>0$ only depending on the local geometry of $(M,g)$ around $p$ such that $G<0$ on $B_g(p,2r)$, $G$ is minimal at $p$ and $G\geq G(p)+a $ on $B_g(p,2r) \setminus B_g(p,r)$, where $B_g(p,2r) \subset V$ is the  geodesic ball of radius $2r$ around $p$. Since $K$ is compact, we can cover $K$ by a \textit{finite}  number of such balls $B_g(p,2r)$.
	 
	 For a given $s\in [0,1]$, we consider $\varphi_{ s} $ solving (\ref{one-parameter family of MA}) and point out that there exists a  $p_s\in K$ such that $\varphi_{ s}(p_s)= \inf _K \varphi_{ s}$. Then $p_s \in B_g(p,2r)$ for one of the balls constructed above.   Define a plurisubharmonic function $u: B_g(p,2r) \to \mathbb{R}_{\leq 0}$ by
	\begin{align*}
		u= \begin{cases}
		\varphi_s + G & \text{if } \sup _M \varphi_{ s} \leq 0, \\
		\varphi_s - \sup _M \varphi_{ s} + G& \text{otherwise},
		\end{cases}
	\end{align*}
	so that \cite{blocki2005uniform}[Proposition 3] implies the following estimate
	\begin{align}\label{in proposition lower bound for varphi: blockis esimate for u}
		\sup _{B_g(p,2r)} |u| \leq a + \left(c_n \cdot 2r \cdot a^{-1}   \right)^{2n} \int _{B_g(p,2r)}|u| \operatorname{dV}_g \cdot \left(\sup _{B_g(p,2r)} \frac{\omega^n_{\varphi_{ s}}}{\omega^n } \right)^2
	\end{align}
	where $\omega_{\varphi_{ s}}= \omega+ i \partial \partialb \varphi_{ s}$ and $c_n>0$ is a constant only depending on the  dimension $n$ of $M$.
	
	 We now explain how to estimate the terms appearing on the right hand side of (\ref{in proposition lower bound for varphi: blockis esimate for u}). 
	We begin by using (\ref{one-parameter family of MA}) together with Lemma \ref{lemma rough lower bound on X(varphi) in terms of f} to obtain
	\begin{align}\label{in proposition lower bound for varphi: bounding volume form fraction from above}
		\sup_{B_g(p,2r)} \frac{\omega_{\varphi_{ s}}^n }{\omega^n} = \sup_{B_g(p,2r)} e^{s\cdot F - \frac{X}{2}(\varphi_{ s})} 
		 \leq \sup_{N_{2r}(K) }e^{|F|+f}=:C_1.
	\end{align}
	Here $N_{2r}(K)$ denotes the tabular neighborhood of radius $2r$ around $K$. Note that since $K$ is compact, the constant $C_1$ is indeed finite. 
	
	Next, we focus on the integral appearing in (\ref{in proposition lower bound for varphi: blockis esimate for u}) and first consider  the case $\sup _M \varphi_{ s}\leq 0$. We continue:
	\begin{align}
	\nonumber &	\int _{B_g(p,2r)} |u|\operatorname{dV}_g \\
\nonumber	&\leq \int _{B_g(p,2r)} |\varphi_{ s}| \operatorname{dV}_g+ \sup _M \varphi_{ s} - G(p)  \\
\nonumber & \leq \max \left\{ 1, \operatorname{Vol}(B_g(p,2r)) \right\}\left( \left(\int _{B_g(p,2r)} |\varphi_{ s}|^2 \operatorname{dV}_g\right)^{\frac{1}{2}}   + C-G(p) \right)\\
	\nonumber	&\leq \max \left\{ 1, \operatorname{Vol}(N_{2r}(K)) \right\} \left( \sup _M \frac{e^{-f}}{f^2 } \cdot \left(\int _M |\varphi_{ s}|^2 \frac{e^f}{f^2 } \operatorname{dV}_g\right)^{\frac{1}{2}}   + C-G(p) \right) \\
\nonumber &	\leq \max \left\{ 1, \operatorname{Vol}(N_{2r}(K)) \right\} \left( \sup _M \frac{e^{-f}}{f^2 } \cdot C  + C-G(p) \right) =:C_2,
	\end{align}
	where we used Cauchy-Schwarz and Proposition \ref{upper-C^0 estimate for varphi_s from above} in the second line and Proposition \ref{weighted L^2 estimate} in the last one. Combining this estimate with (\ref{in proposition lower bound for varphi: blockis esimate for u}) and (\ref{in proposition lower bound for varphi: bounding volume form fraction from above}) then leads to
	\begin{align}
	\nonumber	-\inf _K \varphi_{ s} = - \varphi_{ s}(p_s)&= -u(p_s) - \sup_M \varphi_{ s} + G(p_s) \\ 
		&\leq \sup_{B_g(p,2r)} |u|    \label{in prop: the final lower bound on compact subset}   \\
		&\leq a+\left(c_n \cdot 2r \cdot a^{-1}   \right)^{2n} \cdot C_2 \cdot C_1 ^2.  \nonumber
	\end{align}
Note that a priori, the constants in the last line of (\ref{in prop: the final lower bound on compact subset}) depend on the ball $B_g(p,2r)$  containing the point in which $\varphi_{ s}$ attains its minimum inside $K$. However, since $K$ is covered by only \textit{finitely}  many of such balls $B_g(p,2r)$, (\ref{in prop: the final lower bound on compact subset}) does indeed prove the required uniform lower bound on  $\inf _K \varphi_{ s}$. 
Observing that  the above estimates hold in the case  $\sup_M\varphi_{ s}>0 $ as well then finishes the proof. 
\end{proof}

Thus, it only remains to show Lemma \ref{lemma: construction of compactum} and Proposition \ref{prop: lower bound OUTSIDE a compact set}. We begin with the following crucial observation. 

\begin{lemma}[Uniform bound on $X^2(\varphi_{ s})$] \label{lemma: unform bound on X^2 of varphi}
Let $1<\varepsilon <2$ and suppose $(\varphi_s)_{0\leq s\leq 1}$ is  a family in $C^{\infty}_{\varepsilon, \, JX}(M)$ solving  (\ref{one-parameter family of MA}). Then there exists a constant $C>0$ such that 
\begin{align*}
	\sup _{s\in[0,1]} \sup_M |X(X(\varphi_{ s})) |\leq C,
\end{align*}
where $C$ only depends on $F\in C^{\infty}_{\varepsilon, \, JX}(M)$ and the geometry of $(M,g)$.
\end{lemma}

\begin{proof}
	The idea is to obtain a differential equality to which the maximum principle applies, so that the desired estimate follows. 
	
	First, we differentiate (\ref{one-parameter family of MA}) in the direction of $\frac{X}{2}$, i.e. apply $\mathcal{L}_{\frac{X}{2}}$, which leads to
	\begin{align} \label{in lemma X^2 varphi: first derivative of MA equation}
	n i \partial \partialb \left( f + \frac{X}{2}(\varphi_{ s})  \right) \wedge \omega_{\varphi_{ s}}^{n-1}  = &\left( \frac{X}{2}(sF) - \frac{X^2}{4} (\varphi_{ s}) + \frac{1}{2} \Delta_g f \right) \omega_{\varphi_{ s}}^ n 
	\end{align}
where we abbreviated $X(X(\cdot))=X^2(\cdot)$. Here, we also used two formulas, $\mathcal{L}_{\frac{X}{2}} \omega = i \partial \partialb f$ and $\mathcal{L}_{\frac{X}{2}} \omega_{\varphi_{ s}}= i \partial \partialb f + \frac{X}{2}(\varphi_{ s})$, whose computations can be found in the proof Lemma \ref{lemma reducing to MA equation}. Next, recall that for any real $(1,1)$-form $\alpha$, we have
\begin{align}\label{in lemma X^2 bound: formula for any alpha}
	n (n-1) \alpha^2 \wedge\omega_{\varphi_{ s}} ^{n-2} = \left( (\operatorname{tr}_{\omega_{\varphi_{ s}}} (\alpha))^2 - |\alpha|^2_{g_{\varphi_{ s}}}  \right) \omega_{\varphi_{ s}}^{n},
\end{align}
where $\operatorname{tr}_{\omega_{\varphi_{ s}}} (\alpha)$ is  defined by 
\begin{align} \label{in lemma definition of trace}
  n \alpha \wedge \omega_{\varphi_{ s}} ^{n-1}=\operatorname{tr}_{\omega_{\varphi_{ s}}} (\alpha) \,\omega_{\varphi_{ s}} ^n .
\end{align}
Setting $\alpha := \mathcal{L}_{\frac{X}{2}} \omega_{\varphi_{ s}} = i \partial \partialb f + \frac{X}{2} (\varphi_{ s})$ and applying $\mathcal{L}_{\frac{X}{2}}$ to the left-hand side of (\ref{in lemma definition of trace}) then yields
\begin{equation}
	\begin{aligned} \label{in lemma differentiate LHS}
		\mathcal{L}_{\frac{X}{2}} \left( n \alpha \wedge \omega_{\varphi_{ s}} ^{n-1}  \right) 
		=& n \left( \mathcal{L}_{\frac{X}{2}} \alpha \right) \wedge \omega^{n-1}_{\varphi_{ s}} + n (n-1) \alpha^2 \wedge \omega^{n-2}_{\varphi_{ s}} \\
		=& \frac{n}{2} i \partial \partialb \left( X(f) + \frac{X^2}{2} (\varphi_{ s})  \right) \wedge \omega_{\varphi_{ s}}^{n-1} \\
		&+ \left( (\operatorname{tr}_{\omega_{\varphi_{ s}}} (\alpha))^2 - |\alpha|^2_{g_{\varphi_{ s}}}  \right) \omega_{\varphi_{ s}}^{n},
	\end{aligned}
\end{equation}
where we used $\mathcal{L}_{X} \alpha = i \partial \partialb X(f) + \frac{X^2}{2}(\varphi_{ s})$ and (\ref{in lemma X^2 bound: formula for any alpha}) to conclude the second inequality. 

If we differentiate the right-hand side of (\ref{in lemma definition of trace}) in direction of $\frac{X}{2}$, we obtain
\begin{equation}
	\begin{aligned}\label{in lemma differentiate RHS}
		\mathcal{L}_{\frac{X}{2}} \left(\operatorname{tr}_{\omega_{\varphi_{ s}}} (\alpha )\omega_{\varphi_{ s}} ^n\right)=& \frac{X}{2} \left(\operatorname{tr}_{\omega_{\varphi_{ s}}} (\alpha ) \right) \omega_{\varphi_{ s}} ^n + \operatorname{tr}_{\omega_{\varphi_{ s}}} (\alpha) n \alpha \wedge \omega^{n-1}_{\varphi_{ s}} \\
		=& \left(    \frac{X^2}{4}(sF) - \frac{X^3}{8} (\varphi_{ s}) + \frac{X}{4}(\Delta_g f)    \right) \omega^n _{\varphi_{ s}} \\
		&+ \left( \operatorname{tr}_{\omega_{\varphi_{ s}} } (\alpha) \right)^2 \omega_{\varphi_{ s}} ^n,
	\end{aligned}
\end{equation}
where the second equality follows from (\ref{in lemma definition of trace}) together with the expression of $\operatorname{tr}_{\omega_{\varphi_{ s}}} (\alpha)$ provided by (\ref{in lemma X^2 varphi: first derivative of MA equation}).

Since (\ref{in lemma differentiate LHS}) equals (\ref{in lemma differentiate RHS}), we see that the $\operatorname{tr}_{\omega_{\varphi_{ s}}} (\alpha) ^2$-term is canceled and, after dividing by $\omega_{\varphi_{ s}}^n$, we conclude that
\begin{align*}
	\operatorname{tr}_{\omega_{\varphi_{ s}}} i \partial \partialb \left( \frac{X}{2}(f)   + \frac{X^2}{4}(\varphi_{ s}) \right) - |\alpha|^2_{g_{\varphi_{ s}}} =   \frac{X^2}{4}(sF) - \frac{X^3}{8} (\varphi_{ s})+ \frac{X}{4} (\Delta_g f).
\end{align*}
Multiplying  by $4$, adding $X^2(f)$ on both sides and keeping in mind that $2\operatorname{tr}_{\omega_{\varphi_{ s}}} i \partial \partialb = \Delta_{g_{\varphi_{ s}}}$, we may rearrange the previous equation to finally arrive at
\begin{align}\label{in lemma: differential equality for X^2 varphi}
(	\Delta_{g_{\varphi_{ s}}   } + X )\left(  X(f) + \frac{X^2}{2}(\varphi_{ s})    \right) =  H_1+ 4 \left|  \partial \partialb f + \frac{X}{2}(\varphi_{ s})\right|^2_{g_{\varphi_{ s}}}
\end{align}
with $H_1:= X^{2} (sF) + X(\Delta_g f) + X^2(f) $. We continue to estimate the right-hand side of (\ref{in lemma: differential equality for X^2 varphi}) from below :
\begin{align*}
	4 \left|  \partial \partialb f + \frac{X}{2}(\varphi_{ s})\right|^2_{g_{\varphi_{ s}}}&\geq \frac{1}{n} \left( 	\Delta_{g_{\varphi_{ s}} } \left( f+ \frac{X}{2}(\varphi_{ s})  \right) \right) ^2 \\
	 &= \frac{1}{n} \left(  X(f) + \frac{X^2}{2} (\varphi_{ s}) - H_2 \right)^2,
\end{align*}
where $H_2:= X(f) + X(sF)+ \Delta_g f$ and we made use of (\ref{in lemma X^2 varphi: first derivative of MA equation}) in the second line. Combining the previous inequality with (\ref{in lemma: differential equality for X^2 varphi}), we then obtain
\begin{align*}
 (	\Delta_{g_{\varphi_{ s}}   } + X )\left(  X(f) + \frac{X^2}{2}(\varphi_{ s})    \right)	\geq H_1+\frac{1}{n} \left(  X(f) + \frac{X^2}{2} (\varphi_{ s}) - H_2 \right)^2.
\end{align*}
Note that by Assumption \ref{section MA: asympotitc behaviour of hamiltonian function}, the function $ X(f) + \frac{X^2}{2}(\varphi_{ s})$ tends to $1$ as $t\to \infty$, and so either $ X(f) + \frac{X^2}{2}(\varphi_{ s})\leq 1$, or $ X(f) + \frac{X^2}{2}(\varphi_{ s})$ attains its maximum at some point. In the first case, we are done so we assume that  $X(f) + \frac{X^2}{2}(\varphi_{ s})$  is maximal at $p_{\max} \in M$. Then we observe that the previous inequality gives at $p_{\max}$
\begin{align*}
	X(f) + \frac{X^2}{2}(\varphi_{ s}) \leq \sqrt{ n \, \sup _M |H_1|} + \sup_M |H_2| <\infty,
\end{align*} 
i.e.  $X(f) + \frac{X^2}{2} (\varphi_{ s})$ is uniformly bounded from above. This, in turn, implies the required uniform upper bound on $X^2(\varphi_{ s})$ since $X(f)$ is bounded.

 For the lower bound on $X^2(\varphi_{ s})$ we recall from (\ref{in Lemma: a first bound on X(varphi); X is gradient}) that 
\begin{align*}
	X= \nabla^{g_{\varphi_{ s}}}\left( f+\frac{X}{2}(\varphi_{ s})  \right),
\end{align*}
so that we can estimate as follows:
\begin{align*}
	X	\left(	\frac{X}{2}(\varphi_{ s}) \right) = - X(f) + |X|^2_{g_{\varphi_{ s}}} \geq - \sup_M |X(f)| ,
\end{align*}
which is finite. This completes the proof.

\end{proof}

With the previous lemma, we can finish the proof of Lemma \ref{lemma: construction of compactum}.

\begin{proof}[Proof of Lemma \ref{lemma: construction of compactum}]
	Let us define the barrier function  $v:=e^{\varepsilon_0 \left(f+ \frac{X}{2}(\varphi_{ s})   \right)}$ for some $0<\varepsilon_0<1$ to be chosen later on.  Since $X= \nabla^{g_{\varphi_{ s}} } \left( f+ \frac{X}{2}(\varphi_{ s}) \right)$, we compute
	\begin{align*}
		(\Delta_{g_{\varphi_{ s}}} + X ) \left(  v^{-1}   \right) 
		&= \varepsilon_0 v^{-1} \left( (\varepsilon_0 -1) |X|^2_{g_{\varphi_{ s}}}   - \Delta_{g_{\varphi_{ s}}} \left( f+ \frac{X}{2} (\varphi_{ s})\right)  \right)\\
		&= \varepsilon_0 v^{-1} \left(  (\varepsilon_0 -1) |X|^2_{g_{\varphi_{ s}}} + \frac{X^2}{2}(\varphi_{ s}) - X(sF)  - \Delta_g f  \right)  
	\end{align*}
where we used (\ref{in lemma X^2 varphi: first derivative of MA equation}) in the second line. Recalling the identity
\begin{align}\label{bounding |X|^2}
	|X|^2_{g_{\varphi_{ s}}} = \frac{X^2}{2}(\varphi_{ s}) + X(f),
\end{align}
we may further simplify the previous equation to
\begin{equation}\label{another barrier function v}
	\begin{aligned}
(\Delta_{g_{\varphi_{ s}}} + X ) \left(  v^{-1}   \right)
 &=   \varepsilon_0 v^{-1} \left(  \varepsilon_0  |X|^2_{g_{\varphi_{ s}}} -X(f) - X(sF)  - \Delta_g f  \right)\\
	&\leq \varepsilon_0 v^{-1} \left(  \varepsilon_0 C  - X(f)- X(sF)  - \Delta_g f  \right)
	\end{aligned}
\end{equation} 
for some uniform constant $C>0$ only depending on $\sup_M X(f) $ and the uniform bound on $X^2 (\varphi_{ s})$ from Lemma \ref{lemma: unform bound on X^2 of varphi}. Note that this estimate again uses (\ref{bounding |X|^2}).

 Since $X(f) \to 1$, and $X(F), \Delta_g f \to 0$ as $t\to \infty$, there exists a compact domain $K \subset M$  such that 
\begin{align*}
	X(f)\geq \frac{3}{4} \;\; \text{ and } \;\; |\Delta_g f| + |X(F)| \leq \frac{1}{8} \;\; \text{ on } \;\; M\setminus K.
\end{align*}
Moreover, we can assume that $K$ is of the form
\begin{align} \label{K explicit form}
	K= \{ x\in M \,|\, t(x) \leq t_0 \} 
\end{align}
for some $t_0>0$.
Choosing $\varepsilon_0>0$ sufficiently small so that 
\begin{align*}
	\varepsilon_0 C \leq \frac{1}{8},
\end{align*}
we thus obtain from (\ref{another barrier function v}) that
\begin{align*}
	(\Delta_{g_{\varphi_{ s}}} + X ) \left(  v^{-1}   \right) 
	&\leq  - \frac{\varepsilon_0}{2} v^{-1} \;\; \text{ on } \;\; M\setminus K,
\end{align*}
as claimed.
\end{proof}

Before obtaining Proposition \ref{prop: lower bound OUTSIDE a compact set}, we  require yet another

\begin{lemma}[Bounding $X(\varphi_{ s})$ on  a compact set] \label{lemma: bound of X varphi from above on compact set}
Let $1<\varepsilon<2$ and suppose $(\varphi_{ s})_{0\leq s\leq 1}$ is a family in $C^{\infty}_{\varepsilon,\, JX}(M)$ solving (\ref{one-parameter family of MA}). For the compact domain $K\subset M$ given by Lemma \ref{lemma: construction of compactum}, there exists a constant $C>0$ such that 
\begin{align*}
	\sup_{s\in [0,1]} \sup_K X(\varphi_{ s}) \leq C \;\; \text{ and } \;\; \sup_{s\in [0,1]} X(\varphi_{ s}) \leq C t + C \text{ on } M\setminus K,
\end{align*}
where $C$ only depends on $K$, $F\in C^{\infty}_{\varepsilon, \, JX}(M)$ and the geometry of $(M,g)$.
\end{lemma}

\begin{proof}
	For the first part of the statement, we essentially follow the argument in \cite{conlon2020steady}[Proposition 7.11].
	
	Consider the flow $(\Phi_{\tau})_{\tau \in \mathbb{R}}$ of the complete vector field $\frac{X}{2}$. In particular, the map $\Phi_{\tau}$ corresponds to translation  by $\tau$ in the radial parameter $t$ on the cylindrical end $[0,\infty) \times L$.
	Then we let $\psi_x(\tau):= \varphi_{ s}(\Phi_\tau(x))$ for $(x,\tau)\in M\times \mathbb{R}$ and observe that for each fixed $x\in M$, the limit $\lim_{\tau \to \pm \infty} \psi_x(\tau)$ always exists because $\varphi_s$ tends to zero as $t\to \infty$. 
	Keeping this in mind, we consider $\eta_{-}(\tau):= e^{-\tau}$ and integrate  by parts as follows
	\begin{align} \label{in lemma: bounding X varphi on K, integration by parts eta minus}
		\begin{split}
		\int _0^{\infty} \eta_{-}''(\tau) \psi_x(\tau) d\tau &= - \int_{0}^{\infty} \eta_{-}'(\tau) \psi_{x}'(\tau )d\tau + \psi_{x}(0) \\
		&=\int _{0}^{\infty} \eta_{-}(\tau) \psi_x''(\tau) d\tau + \psi_x'(0)+ \psi_x(0).
			\end{split}
	\end{align} 
By choosing $x\in K$, rearranging (\ref{in lemma: bounding X varphi on K, integration by parts eta minus}) and using $\frac{X}{2}(\varphi_{ s}) (x) = \psi_x'(0)$, we consequently estimate 
\begin{align*}
	\frac{X}{2}(\varphi_{ s}) (x) &\leq - \inf _K \varphi_{ s} + \sup_M \varphi_{ s} \int_0 ^{\infty} e^{-\tau} d\tau - \inf_M \frac{X^2 }{4}(\varphi_{ s}) \int_0 ^{\infty} e^{-\tau} d\tau\\
	&\leq C_1,
\end{align*}
where $C_1>0$ only depends on $F$ and the geometry of $(M,g)$, thanks to Propositions \ref{prop: lower bound for varphi on compact sets} and \ref{upper-C^0 estimate for varphi_s from above} as well as Lemma \ref{lemma: unform bound on X^2 of varphi}. This shows the first part of this lemma.

For the second part, recall from (\ref{K explicit form}), that we can identify $M\setminus K \cong (t_0,\infty) \times L$ for some $t_0>0$. To emphasize this splitting, we write  $x=(t,y)$ for points $x\in M\setminus K$. Under this identification, $X= 2 \partial /\partial t$ and so we can write
\begin{align}\label{in lemma bound on compact set X varphi: rough bound outside K}
	\begin{split}
	X(\varphi_{ s}) (t,y) &= \int_{t_0}^{t} \frac{X^{2}}{2} (\varphi_{ s}) (\sigma,y) d\sigma + X(\varphi_{ s}) (t_0,y)\\
	&\leq C_2 (t-t_0) + C_1,
\end{split}
\end{align}
since $(0,y) \in K$ and $X^2(\varphi_{ s}) \leq C_2$ for some uniform constant $C_2>0$ given by Lemma \ref{lemma: unform bound on X^2 of varphi}. As the right-hand side of (\ref{in lemma bound on compact set X varphi: rough bound outside K}) is independent of $s\in [0,1]$, the lemma follows.
\end{proof}

Now we can deduce Proposition \ref{prop: lower bound OUTSIDE a compact set}.

\begin{proof}[Proof of Proposition \ref{prop: lower bound OUTSIDE a compact set}]
		As in \cite{conlon2020steady}[Proposition 7.20], we use a barrier function to show the claim.
		Let $0<\varepsilon_0<1$ and $K\subset M$ be given by Lemma \ref{lemma: construction of compactum}, i.e. on $M\setminus K$, we have
		\begin{align}\label{in proposition lower bound OUTSIDE: weight function DE}
			\left( \Delta_{g_{\varphi_s}} +X  \right) \left(e^{-\varepsilon_0 \left(  f+ \frac{X}{2}(\varphi_{ s})  \right)} \right) \leq - \frac{\varepsilon_0}{2} e^{-\varepsilon_0 \left(  f+ \frac{X}{2}(\varphi_{ s})  \right)} <0 .
		\end{align}
	The reader may observe from the proof that (\ref{in proposition lower bound OUTSIDE: weight function DE}) holds as long as $0<\varepsilon_0 \ll 1$ is sufficient small. In particular, we are free to choose $\varepsilon_0>0$ as small was we require and (\ref{in proposition lower bound OUTSIDE: weight function DE})  is still valid. 
	
	Similar to (\ref{in proposition on C0 upper bound: inequality for Delta_g varphi}), which was used for proving the upper bound, the Monge-Amp\`ere equation (\ref{one-parameter family of MA}) implies
	\begin{align*}
		(\Delta_{g_{\varphi_{ s}}} + X) (\varphi_{ s}) \leq |F|,
	\end{align*}
and so for some $A>0$ to be specified later on, we obtain 
\begin{align}\label{in proposition lower bound OUTSIDE: weight function summed up to varphi}
	\left( \Delta_{g_{\varphi_s}} +X  \right) \left(   \varphi_{ s} + Ae^{-\varepsilon_0  \left(  f+ \frac{X}{2}(\varphi_{ s})  \right)} \right) 
	&\leq |F| - A\frac{\varepsilon_0}{2} e^{-\varepsilon_0 \left(  f+ \frac{X}{2}(\varphi_{ s})  \right)} .
\end{align}
The idea is to  choose $A\gg 1$ sufficiently large so that the right term in (\ref{in proposition lower bound OUTSIDE: weight function summed up to varphi}) becomes negative. 

Note that by Lemma \ref{lemma: bound of X varphi from above on compact set} and Assumption \ref{section MA: asympotitc behaviour of hamiltonian function} there exists a constant $C>0$ only depending on $F$ and the geometry of $(M,g)$ such that 
\begin{align}
	\begin{split} \label{in lemma bound of varphi OUTSIDE: preparation for max principle}
	\varepsilon_0 \left(f+ \frac{X}{2}(\varphi_{ s})       \right) \leq \varepsilon_0 C t + C 
	\leq \varepsilon t + C,
	\end{split}
\end{align}
where the second inequality holds if we fix some $\varepsilon_0>0$ with
\begin{align*}
	C\varepsilon_0 < \varepsilon. 
\end{align*}
Applying (\ref{in lemma bound of varphi OUTSIDE: preparation for max principle}) to the right-hand side of (\ref{in proposition lower bound OUTSIDE: weight function summed up to varphi}) then yields
\begin{align*}
	 |F| - A\frac{\varepsilon_0}{2} e^{-\varepsilon_0 \left(  f+ \frac{X}{2}(\varphi_{ s})  \right)} 
	&\leq e^{-\varepsilon_0 \left(  f+ \frac{X}{2}(\varphi_{ s})  \right)} \left(  e^{\varepsilon_0 \left(  f+ \frac{X}{2}(\varphi_{ s})  \right) - \varepsilon t} ||F||_{C^0_\varepsilon}  -A\frac{\varepsilon_0}{2}  \right) \\
	&\leq  e^{-\varepsilon_0 \left(  f+ \frac{X}{2}(\varphi_{ s})  \right)} \left( e^C ||F||_{C^0 _\varepsilon} - A \frac{\varepsilon_0}{2}    \right).
\end{align*}
Thus, choosing $A>0$ sufficiently large so that 
\begin{align*}
	A>\frac{2}{\varepsilon_0} e^{C} ||F||_{C^0_\varepsilon},
\end{align*}
and plugging this back into (\ref{in proposition lower bound OUTSIDE: weight function summed up to varphi}), we arrive at
\begin{align*}
	\left( \Delta_{g_{\varphi_s}} +X  \right) \left(   \varphi_{ s} + Ae^{-\varepsilon_0  \left(  f+ \frac{X}{2}(\varphi_{ s})  \right)} \right) \leq 0 \;\; \text{ on } \;\; M\setminus K.
\end{align*}
Hence, Hopf's maximum principle states that   
\begin{align} \label{in proposition: lower bound OUTSIDE, AFTEr applying max principle}
\nonumber \varphi_{ s} + Ae^{-\varepsilon_0  \left(  f+ \frac{X}{2}(\varphi_{ s})  \right)} &\geq \min \left\{ 0 , \min_{\partial K } \left(\varphi_{ s} + Ae^{-\varepsilon_0  \left(  f+ \frac{X}{2}(\varphi_{ s})  \right)} \right) \right\} \\
&\geq \min \left\{ 0, \min_K \varphi_{ s}  \right\}\\
\nonumber &\geq -C
\end{align}
holds on $M\setminus K$ because $\varphi_{ s} + Ae^{-\varepsilon_0  \left(  f+ \frac{X}{2}(\varphi_{ s})  \right)} $ goes to $0$ as $t\to \infty$ and $\min_K \varphi_{ s}$ is, according to Lemma \ref{lemma: bound of X varphi from above on compact set}, uniformly bounded from below by some constant $-C<0$. To conclude  the Proposition, we observe that 
\begin{align*}
	f+ \frac{X}{2}(\varphi_{ s}) \geq 1
\end{align*}
by Lemma \ref{lemma rough lower bound on X(varphi) in terms of f} and consequently,
\begin{align*}
	\varphi_{ s} \geq - C - A e^{-\varepsilon_0} \;\; \text{ on } \;\; M\setminus K,
\end{align*}
as claimed.

\end{proof}

Having finally finished the proof of Proposition \ref{proposition: lower bound on inf varphi}, we can now strengthen the estimates  in Lemma \ref{lemma: bound of X varphi from above on compact set}, i.e. achieve a uniform bound on the radial derivative of $ \varphi_{ s}$.

\begin{corollary}[Uniform bound on $X(\varphi_{ s})$]\label{proposition: estimate on radial derivative}
	Let $1<\varepsilon<2$ and suppose $(\varphi_{ s})_{0\leq s \leq 1}$ is a family in $C^{\infty}_{\varepsilon, \, JX}(M)$ solving (\ref{one-parameter family of MA}). Then there exists a constant $C>0$ such that
	\begin{align*}
		\sup_{s\in [0,1]} \sup_M |X(\varphi_{ s})| \leq C,
	\end{align*}
where $C$ only depends on $F\in C^{\infty}_{\varepsilon, \, JX}(M)$ and the geometry of $(M,g)$.
\end{corollary}

\begin{proof}
We apply the same idea as in the proof of Lemma \ref{lemma: bound of X varphi from above on compact set}. Namely, by (\ref{in lemma: bounding X varphi on K, integration by parts eta minus}), we can estimate for each $x\in M$:
	\begin{align*}
		\frac{X}{2} (\varphi_{s}) (x) \leq - \inf _M \varphi_{ s} + \sup _M \varphi_{ s} - \inf _M \frac{X^2}{4}(\varphi_{ s}),
	\end{align*}
so that the uniform upper bound follows from Propositions \ref{upper-C^0 estimate for varphi_s from above}, \ref{proposition: lower bound on inf varphi} and Lemma \ref{lemma: unform bound on X^2 of varphi}. The lower bound is similar. Using $\eta_+=e^{\tau}$ instead of $\eta_-= e^{-\tau}$ leads to 
\begin{align*}
	\int_{-\infty}^0 \eta_+ ''(\tau ) \psi_x(\tau) d\tau = \int _{-\infty}^0 \eta_{+}(\tau ) \psi''_x(\tau) d\tau - \psi'_x(0) + \psi_x(0) ,
\end{align*}
and estimating as before then yields
\begin{align*}
	\frac{X}{2}(\varphi_{ s})(x)&\geq \inf _M \varphi_s -\sup _M \varphi_{ s}    + \inf _M \frac{X^2 }{4} (\varphi_{ s}),  
\end{align*}
finishing the proof. 
\end{proof}

This  new bound on $X(\varphi_{ s})$ enables us to conclude a weighted lower bound on $\varphi_{ s}$, at least for \textit{some} $\varepsilon_0<\varepsilon$. 

\begin{prop}[A first weighted lower bound on $\varphi_{ s}$]\label{proposition lower weighted bound for varepsilon0}
	Let $1<\varepsilon<2$ and suppose $(\varphi_{ s})_{0\leq s\leq 1}$ is a family in $C^{\infty} _{\varepsilon, \, JX}(M)$ solving (\ref{one-parameter family of MA}). Then there exist two constants $0<\varepsilon_0<1$ and  $C>0$ such that
	\begin{align*}
		\inf _{s\in [0,1]} \inf _M e^{\varepsilon_0 t} \varphi_s \geq -C,
	\end{align*}
where both $\varepsilon_0$ and $C$ only depend on $F\in C^{\infty}_{\varepsilon, \, JX}(M)$ and the geometry of $(M,g)$.
\end{prop}

\begin{proof}
 Using Corollary \ref{proposition: estimate on radial derivative}, the proof of Proposition \ref{proposition: lower bound on inf varphi} can be refined by following the argument in \cite{conlon2020steady}[Proposition 7.20].
 
 We repeat the proof until arriving at (\ref{in proposition: lower bound OUTSIDE, AFTEr applying max principle}), so that we have on $M \setminus K$:
 \begin{align} \label{in prop: first weighted lower bound starting point of modification}
 	\varphi_{ s} + Ae^{-\varepsilon_0  \left(  f+ \frac{X}{2}(\varphi_{ s})  \right)} &\geq \min \left\{ 0 , \min_{\partial K } \left(\varphi_{ s} + Ae^{-\varepsilon_0  \left(  f+ \frac{X}{2}(\varphi_{ s})  \right)} \right) \right\} .
 \end{align}
By Corollary  \ref{proposition: estimate on radial derivative} and Assumption \ref{section MA: asympotitc behaviour of hamiltonian function}, there is a uniform constant $C>0$ such that 
\begin{align} \label{in prop: first weighted lower bound, some inequality}
	C^{-1} e^{-2\varepsilon_0 t} \leq e^{-\varepsilon_0  \left(  f+ \frac{X}{2}(\varphi_{ s})  \right)} \leq C e^{-2\varepsilon_0 t}
\end{align} 
holds on $M$. In particular, since $\inf_K \varphi_{ s} $ is uniformly bounded from below by Proposition \ref{prop: lower bound for varphi on compact sets}, we can choose $A\gg1$ even larger, so that
\begin{align*}
\min _{\partial K} \left( \varphi_{ s} +Ae^{-\varepsilon_0  \left(  f+ \frac{X}{2}(\varphi_{ s})  \right)}\right) \geq  \inf _K \varphi_{ s} + A \inf_K C^{-1}e^{-2\varepsilon_0 t} \geq 0.
\end{align*}
Consequently, we arranged that (\ref{in prop: first weighted lower bound starting point of modification}) becomes 
\begin{align*}
	\varphi_{ s} \geq - Ae^{-\varepsilon_0  \left(  f+ \frac{X}{2}(\varphi_{ s})  \right)} \geq -A C e^{-2\varepsilon_0 t} \;\; \text{ on } \;\; M \setminus K,
\end{align*}
because of (\ref{in prop: first weighted lower bound, some inequality}). This is precisely what we wanted to prove. 
\end{proof}

Before improving the weighted bound from $\varepsilon_0$ to $\varepsilon$, we require uniform bounds on all derivatives of $\varphi_{ s}$, which is the content of the subsequent section.

\subsection{Higher order estimates} \label{subsection: higher order estimates}

In the previous section, we obtained uniform bounds on $\varphi_{ s}$ and its radial derivative up to second order.
Using these results, we begin by deriving bounds on the $C^2$- and $C^3$-norms of $\varphi_{ s}$, which then leads to estimates for all derivatives. We purse essentially the same strategy as in \cite{conlon2020steady}[Section 7], but occasionally we present different computations.

\subsubsection{The $C^2$-estimate}

The $C^2$-estimate for $\varphi_{ s}$ is equivalent to bounding the associated metric $g_{\varphi_{ s}}$ uniformly in terms of $g$.

\begin{prop}[Uniform bound on the metric] \label{proposition: uniform bound on metrics}
	Let $1<\varepsilon<2$ and suppose $(\varphi_{ s})_{0\leq s\leq 1}$ is a family in $C^{\infty}_{\varepsilon, \, JX}(M)$ solving (\ref{one-parameter family of MA}). If $g_{\varphi_{ s}}$ denotes the metric associated to the K\"ahler form $\omega+ i \partial \partialb \varphi_{ s}$, then there exists a constant $C>0$ such that
	\begin{align}\label{in proposition: equivalent metrics}
		C^{-1} g\leq g_{\varphi_{ s}} \leq C g,
	\end{align}
	where  $C$ only depends on $F\in C^{\infty}_{\varepsilon, \, JX}(M)$ and the geometry of $(M,g)$.
\end{prop}

  Before proceeding with the proof, we immediately obtain a uniform bound on the volume form by 
 looking at (\ref{one-parameter family of MA}) and applying Corollary \ref{proposition: estimate on radial derivative}.
 
 \begin{corollary}[Uniform bound on volume form] \label{corollary: uniform volume bound}
 	Let $1<\varepsilon<2$ and suppose $(\varphi_{ s})_{0\leq s\leq 1}$ is a family in $C^{\infty}_{\varepsilon, \, JX}(M)$ solving (\ref{one-parameter family of MA}). Then there exists a constant $C>0$ such that 
 	\begin{align*}
 		C^{-1} \omega^n  \leq (\omega+i\partial \partialb \varphi_{ s})^n \leq C\omega^n ,
 	\end{align*}
 	where $C>0$ only depends on $F \in C^{\infty}_{\varepsilon, \, JX}(M)$  and the geometry of $(M,g)$.
 \end{corollary}

\begin{proof}[Proof of Proposition \ref{proposition: uniform bound on metrics}]
	We argue as in \cite{conlon2020steady}[Proposition 7.14], but present different calculations.
	The bound (\ref{in proposition: equivalent metrics}) amounts to bounding both $\operatorname{tr}_{\omega} \omega_{\varphi_{ s}}$ and $\operatorname{tr}_{\omega_{\varphi_{ s}}} \omega$ uniformly from above. However, there is the well-known formula
	\begin{align*}
		\operatorname{tr}_{\omega_{\varphi_{ s}}} \omega \leq n\cdot  \frac{\omega^n}{\omega_{\varphi_{ s}}^n} \left( \operatorname{tr}_{\omega} \omega_{\varphi_{ s}} \right) ^{n-1}
	\end{align*}
	compare for example (\cite{boucksom2013introduction}[Lemma 4.1.1]). Thus, it suffices to estimate $\operatorname{tr}_{\omega} \omega_{\varphi_{ s}}$ since the volume form $\omega_{\varphi_{ s}} ^n$ is uniformly bounded by Corollary \ref{corollary: uniform volume bound}. 
	
	In this proof, $C>0$ denotes a uniform constant, which may increase from line to line but only depends on the geometry of $(M,g)$ and the $C^{\infty}$-norm of $F$. 
	
	Recall that a standard computation yields the following inequality
	\begin{align}\label{in propositon bound on metrics: laplace estimate from below: step 1}
	\frac{1}{2}	\Delta_{g_{\varphi_{ s}}} \log \operatorname{tr}_{\omega} \omega_{\varphi_{ s}} \geq - \frac{\operatorname{tr}_\omega \operatorname{Ric}(\omega_{\varphi_s})}{\operatorname{tr}_{\omega} \omega_{\varphi_{ s}}} - C \operatorname{tr}_{\omega_{\varphi_{ s}}} \omega,
	\end{align}
	where $\operatorname{Ric}(\omega_{\varphi_s})$ is the Ricci form of $\omega_{\varphi_{ s}}$ and $C>0$ a constant such that the holomorphic bisectional curvature of $g$ is bounded from below by $-C$. For a proof of this inequality, we refer the reader to  \cite{boucksom2013introduction}[Proposition 4.1.2]. Also observe that in our case the bisectional curvature of $g$ is bounded since $g$ is asymptotically cylindrical. 
	Starting  from (\ref{one-parameter family of MA}), we compute the Ricci form of $\omega_{\varphi_{ s}}$ :
	\begin{align} \label{in proposition bound on metrics: Ric of omega varphi}
		\operatorname{Ric}(\omega_{\varphi_{ s}}) = - i \partial \partialb \log \omega_{\varphi_{ s}} ^n = \operatorname{Ric}(\omega) - i \partial \partialb sF + i \partial \partialb\left( \frac{X}{2}(\varphi_{ s}) \right).
	\end{align}
	As both  $||F||_{C^2 }$ and the curvature of $g$ are uniformly bounded, we continue to estimate
	\begin{align} \label{in proposition bound on metrics: trace of ricci form bounded from below}
		- \operatorname{tr}_\omega \operatorname{Ric}(\omega_{\varphi_{ s}}) \geq - C- \operatorname{tr}_{\omega} i \partial \partialb\left( \frac{X}{2}(\varphi_{ s}) \right).
	\end{align}
	Also recall from \cite{boucksom2013introduction}[Lemma 4.1.1] that 
	\begin{align}\label{in proposition bound on metrics: trace bounded from below}
		\operatorname{tr}_\omega \omega_{\varphi_{ s}} \geq n \cdot \left( \frac{\omega_{\varphi_{ s}}^n}{\omega ^n}\right)^{\frac{1}{n}} \geq C^{-1}>0,
	\end{align}
	where the lower bound again follows from Corollary \ref{corollary: uniform volume bound}. 
	Combining (\ref{in proposition bound on metrics: trace bounded from below}) and (\ref{in proposition bound on metrics: trace of ricci form bounded from below}) with (\ref{in propositon bound on metrics: laplace estimate from below: step 1}), we consequently arrive at
	\begin{align} \label{in proposition bound on metrics: laplace bound from below: final step}
	\frac{1}{2}	\Delta_{g_{\varphi_{ s}}} \log \operatorname{tr}_{\omega} \omega_{\varphi_{ s}} \geq - \frac{\operatorname{tr}_\omega i \partial \partialb \left(\frac{X}{2}(\varphi_{ s})    \right) }{\operatorname{tr}_\omega \omega_{\varphi_{ s}}} - C - C\operatorname{tr}_{\omega_{\varphi_{ s}}} \omega.
	\end{align}
	Next, we calculate the radial derivative of $\operatorname{tr}_\omega \omega_{\varphi_{ s}}$ by considering its defining equation:
	\begin{align*}
		\operatorname{tr}_\omega \omega_{\varphi_{ s}} \cdot \omega ^n = n \cdot \omega_{\varphi_{ s}} \wedge \omega ^{n-1}.
	\end{align*}
	Taking the Lie derivative in direction $\frac{X}{2}$ on both sides of this equation and then dividing by $\omega^n$ leads to 
	\begin{align}\label{in proposition bound on metrics: compute X of trace omega omega varphi}
\nonumber	&\frac{X}{2}(\operatorname{tr}_\omega \omega_{\varphi_{ s}})  + \operatorname{tr}_\omega \omega_{\varphi_{ s}} \cdot   \operatorname{tr}_\omega \mathcal{L}_{\frac{X}{2}} (\omega)  \\
	=&  \operatorname{tr}_\omega \mathcal{L}_{\frac{X}{2}}(\omega_{\varphi_{ s}})  + n(n-1)\cdot \frac{ \omega_{\varphi_{ s}} \wedge \mathcal L_{\frac{X}{2}}(\omega) \wedge  \omega^{n-2} }{\omega^n} \\
	=& \operatorname{tr}_\omega \mathcal{L}_{\frac{X}{2}}(\omega_{\varphi_{ s}})  + \operatorname{tr}_\omega
\nonumber \omega_{\varphi_{ s}} \cdot \operatorname{tr}_\omega \mathcal{L}_{\frac{X}{2}} (\omega) - \langle  	\omega_{\varphi_{ s}}, \mathcal{L}_{\frac{X}{2}} (\omega) \rangle_{g},
	\end{align}
	or equivalently,
	\begin{align}\label{in proposition bound on metrics: final formula for X of tr omega omega varphi}
		\frac{X}{2}(\operatorname{tr}_\omega \omega_{\varphi_{ s} }) = \operatorname{tr}_\omega \mathcal{L}_{\frac{X}{2}} (\omega_{\varphi_{ s} }) - \langle \omega_{\varphi_{ s} }, \mathcal{L}_{\frac{X}{2}}(\omega) \rangle_g
	\end{align}
	Here, the last equation in (\ref{in proposition bound on metrics: compute X of trace omega omega varphi}) is a straight forward computation, which can be found in \cite{Szek}[Lemma 4.6], and $\langle,\rangle_g$ denotes the metric on $2$-forms induced by $g$. 
	 We recall that $\mathcal{L}_{\frac{X}{2}} (\omega) = i \partial \partialb f$ since $X=\nabla^g f$ and also that the norm $|i\partial \partialb f|_g \leq C$ is uniformly bounded by some $C>0$ because of  \ref{section MA: asympotitc behaviour of hamiltonian function}. Applying these observations to  the previous equation, we obtain
	 \begin{align*}
	 	\frac{X}{2}(\log \operatorname{tr}_\omega \omega_{\varphi_{ s}}) 
	 	&= \frac{\operatorname{tr}_\omega i \partial \partialb \left(\frac{X}{2}(\varphi_{ s}) \right)}{\operatorname{tr}_\omega \omega_{\varphi_{ s}}} + \frac{\operatorname{tr}_\omega i \partial \partialb f }{\operatorname{tr}_\omega \omega_{\varphi_{ s}}} - \frac{\langle  \omega_{\varphi_{ s}}, \mathcal{L}_{\frac{X}{2}} (\omega) \rangle_{g}}{\operatorname{tr}_\omega \omega_{\varphi_{ s}}}\\
	 	&\geq \frac{\operatorname{tr}_\omega i \partial \partialb \left(\frac{X}{2}(\varphi_{ s}) \right)}{\operatorname{tr}_\omega \omega_{\varphi_{ s}}} - \frac{C}{\operatorname{tr}_\omega \omega_{\varphi_{ s}}} - \frac{C\cdot |\omega_{\varphi_{ s}}|_g }{\operatorname{tr}_\omega \omega_{\varphi_{ s}}} \\
	 	&\geq \frac{\operatorname{tr}_\omega i \partial \partialb \left(\frac{X}{2}(\varphi_{ s}) \right)}{\operatorname{tr}_\omega \omega_{\varphi_{ s}}} -C ,
	 \end{align*}
	where we used the bound on $|i\partial \partialb f|_g$ in the second line and (\ref{in proposition bound on metrics: trace bounded from below}) in the last one. Altogether, we finally arrive at
	\begin{align}\label{in proposition bound on metric: inequality for delta plus X over 2 applied to trace}
	\frac{1}{2}	\left(  X + \Delta_{g_{\varphi_{s}}} \right)  \log \operatorname{tr}_\omega \omega_{\varphi_{ s}} \geq -C -C \operatorname{tr}_{\omega_{\varphi_{ s}} } \omega.
	\end{align}
	From there, it is  standard to conclude an upper bound on $\operatorname{tr}_\omega \omega_{\varphi_{ s}}$. We begin by  considering the following inequality 
	\begin{align*}
\frac{1}{2}	\left(  X + \Delta_{g_{\varphi_{s}}} \right) \varphi_{ s} \leq C + n -\operatorname{tr}_{\omega_{\varphi_{ s}} }\omega
	\end{align*}
	where we used the upper bound on $X(\varphi_{ s})$ from Proposition \ref{proposition: estimate on radial derivative} and the definition of $\omega_{\varphi_{ s}}$. In combination with (\ref{in proposition bound on metric: inequality for delta plus X over 2 applied to trace}), we then obtain
	\begin{align}\label{in proposition bound on metrics: final inequality for max principle set up}
	\frac{1}{2}	\left(  X + \Delta_{g_{\varphi_{s}}} \right)  (\log \operatorname{tr}_\omega \omega_{\varphi_{ s}} - (C+1) \varphi_{ s}) \geq -C + \operatorname{tr}_{\omega_{\varphi_{ s}} } \omega.
	\end{align}
	Applying the maximum principle to this equation, yields the desired estimate for $\operatorname{tr}_\omega \omega_{\varphi_{ s}}$ as follows. Note that we can assume  $\log \operatorname{tr}_\omega \omega_{\varphi_{ s}} - (C+1) \varphi_{ s}>n$ at least somewhere on $M$, because otherwise we are done by the uniform upper bound on $\varphi_{ s}$ (Proposition \ref{upper-C^0 estimate for varphi_s from above}). Thus, there exists  $p_{\max}\in M$ such that $\log \operatorname{tr}_\omega \omega_{\varphi_{ s}} - (C+1) \varphi_{ s}$ is maximal at $p_{\max}$. Then at this point, we obtain from (\ref{in proposition bound on metrics: final inequality for max principle set up}) that $\operatorname{tr}_{\omega_{\varphi_{ s}}} \omega \leq C$, so that at $p_{max}$:
	\begin{align*}
		\ \operatorname{tr}_\omega \omega_{\varphi_{ s}} \cdot e^{- (C+1) \varphi_{ s}} \leq  n  e^{- (C+1) \varphi_{ s}} \cdot \frac{\omega_{\varphi_{ s}} ^n }{\omega^n } (\operatorname{tr}_{\omega_{\varphi_{ s}}} \omega)^{n-1}  \leq C.
	\end{align*}
	Hence,  $ \log \operatorname{tr}_\omega \omega_{\varphi_{ s}} - (C+1) \varphi_{ s}$ is uniformly bounded from above, and so is $\operatorname{tr}_\omega \omega_{\varphi_{ s}}$,  finishing the proof. 
\end{proof}

\subsubsection{The $C^3$-estimate}

\begin{prop}[Uniform  $C^3$-estimate] \label{proposition uniform C3 estimate}
Let $1<\varepsilon<2$ and suppose $(\varphi_{ s})_{0\leq s\leq 1}$ is a family in $C^{\infty}_{\varepsilon,\, JX}(M)$ solving (\ref{one-parameter family of MA}). 
  If $g_{\varphi_{ s}}$ denotes the metric associated to the K\"ahler form $\omega+ i \partial \partialb \varphi_{ s}$, then there exists a constant $C>0$ such that 
	\begin{align*}
		\sup_{s\in[0,1]} \sup _M |\nabla^g g_{\varphi_{ s}}|_g \leq C,
	\end{align*}
	where the constant $C$ only depends on $F\in C^{\infty}_{\varepsilon,\, JX}(M)$ and the geometry of $(M,g)$.
\end{prop}

\begin{proof}
	We define 
	\begin{align*}
		S:=|\nabla^g g_{\varphi_{ s}}|_g ^2,
	\end{align*}
	and then the computation in \cite{conlon2020steady}[Proposition 7.16] goes through verbatim. In particular, if $\operatorname{Rm}(g)$ denotes the curvature tensor of $g$, there exists a constant $C>0$, which only depends on the constant in Proposition \ref{proposition: uniform bound on metrics} as well as on bounds for covariant derivatives of both $F$ and $\operatorname{Rm}(g)$, such that 
	\begin{align}\label{in proposition C3 estimate: differential inequality for S}
		\frac{1}{2}\left( \Delta_{g_{\varphi_{ s}}}- X \right) S \geq -C(S+1).
	\end{align}
	Moreover, recall that the standard Schwarz-Lemma calculation in holomorphic coordinates yields
	\begin{align}\label{in proposition C3 estimate: precise formula in coordinates for laplace of metric}
		\frac{1}{2}\Delta_{g_{\varphi_{s}}} \operatorname{tr}_\omega \omega_{\varphi_{ s}}= - \operatorname{tr}_\omega\operatorname{Ric}(\omega_{\varphi_{ s}}) + g_{\varphi_{ s}}^{\bar l k} R_{k\bar l} ^{\;\;\;j \bar i} g^{\varphi_{ s}} _{i\bar j} + g^{\bar j i} g^{\bar q p} _{\varphi_{ s}} g^{\bar l k} _{\varphi_{ s}} \nabla^g _i g^{\varphi_{ s}} _{p \bar l} \nabla^g _{\bar j} g^{\varphi_{ s}} _{k\bar q},
			\end{align}
	where  $g^{\varphi_{ s}}_{i\bar j}$ denotes the components of $g_{\varphi_{ s}}$ in coordinates, with inverse $g_{\varphi_{ s} } ^{\bar j i}$, and $R_{k \bar l i \bar j}$ is the local expression of $\operatorname{Rm}(g)$. For the computation, we refer the reader to \cite{boucksom2013introduction}[(3.67)], for example.
	Starting from (\ref{in proposition C3 estimate: precise formula in coordinates for laplace of metric}), and keeping Proposition \ref{proposition: uniform bound on metrics} as well as (\ref{in proposition bound on metrics: trace of ricci form bounded from below}) in mind, we estimate
	\begin{align*}
		\frac{1}{2}\Delta_{g_{\varphi_{s}}} \operatorname{tr}_\omega \omega_{\varphi_{ s} }\geq - \operatorname{tr}_\omega i \partial \partialb \left( \frac{X}{2} (\varphi_{ s})\right) - C + C^{-1}S.
	\end{align*}
	Proposition \ref{proposition: uniform bound on metrics} applied to (\ref{in proposition bound on metrics: final formula for X of tr omega omega varphi}) also leads to 
	\begin{align*}
		\frac{X}{2} (\operatorname{tr}_\omega \omega_{\varphi_{ s}}) \geq \operatorname{tr}_\omega i \partial \partialb \left( \frac{X}{2}(\varphi_{ s})\right) - C,
	\end{align*}
	and hence,
	\begin{align} \label{in proposition C3 estimate: differential inequality for trace omega omega varphi}
	\frac{1}{2} \left( \Delta_{g_{\varphi_{ s}}} +X   \right) \operatorname{tr}_\omega \omega_{\varphi_{ s} } \geq -C + C^{-1} S
	\end{align}
	for some constant $C>0$ only depending on $||\operatorname{Rm}(g)||_{C^0(M)}$, $||\partial \partial f||_{C^0(M)}$, $||F||_{C^2(M)}$ and the constant in Proposition \ref{proposition: uniform bound on metrics}.  
	If we choose a sufficiently large constant $C_1>0$ and then add  (\ref{in proposition C3 estimate: differential inequality for S}) to $C_{1}$-times (\ref{in proposition C3 estimate: differential inequality for trace omega omega varphi}), we can arrange that
	\begin{align}\label{in proposition C3 estimate: final differential inequality for maximum principle}
		\frac{1}{2}(\Delta_{g_{\varphi_{ s}}} - X) \left( S + C_1 \operatorname{tr}_\omega \omega_{\varphi_{ s} }   \right) \geq -C + S .
	\end{align}
	Again, there are two cases to consider. If $S + C_1 \operatorname{tr}_\omega \omega_{\varphi_{ s} } \leq \lim_{t\to \infty} S + C_1 \operatorname{tr}_\omega \omega_{\varphi_{ s} }  =n$, there is nothing to show, so we can assume $S + C_1 \operatorname{tr}_\omega \omega_{\varphi_{ s} } >n$. Thus, there exists a point $p_{\max} \in M$, where $S + C_1 \operatorname{tr}_\omega \omega_{\varphi_{ s} } $ is maximal. Applying the maximum principle to (\ref{in proposition C3 estimate: final differential inequality for maximum principle}), we have at $p_{\max}$:
\begin{align}
S + C_1 \operatorname{tr}_\omega \omega_{\varphi_{ s} } \leq C + C_1 \sup_M \operatorname{tr}_\omega \omega_{\varphi_{ s} } \leq C.
\end{align}
This implies a uniform upper bound on $S$, as claimed. 
\end{proof}

Since the $C^1$-norm of $g_{\varphi_{ s}}$ is uniformly bounded, we  obtain a uniform $C^{0,\alpha}$-bound on  $g_{\varphi_{ s}}$, as in \cite{conlon2020steady}[Corollary 7.17].
\begin{corollary}\label{uniform C0,alpha estimate of metric}
	Let $1<\varepsilon<2$, $\alpha \in (0,1)$ and suppose $(\varphi_{ s})_{0\leq s\leq 1}$ is a family in $C^{\infty}_{\varepsilon,\, JX}(M)$ solving (\ref{one-parameter family of MA}). If  $g_{\varphi_{ s}}$ denotes the Riemannian metric corresponding to $\omega_{\varphi_{ s} }$, and $g_{\varphi_{ s}}^{-1}$ the induced metric on 1-forms, then there exists a constant $C>0$ such that
	\begin{align*}
		\sup_{s\in[0,1]} \left( ||g_{\varphi_{ s}} ||_{C^{0,\alpha}} + ||g_{\varphi_{ s} }^{-1}||_{C^{0,\alpha}} \right) \leq C,
	\end{align*}
	where  $C$ only depends on $\alpha$, $ F\in C^{\infty}_{\varepsilon, \, JX}(M)$ and the geometry of $(M,g)$. 
\end{corollary}

\begin{proof}
	Recall that the natural embedding
	\begin{align*}
		C^{1}(TM \otimes TM) \subseteq C^{0,\alpha}(TM \otimes TM)
	\end{align*}
	is continuous, so that the operator norm of this inclusion only depends on $(M,g)$ and $\alpha$. 
	Hence, the $C^{0,\alpha}$-norm of $g_{\varphi_{ s}}$ is uniformly bounded from above by  $ ||g_{\varphi_{ s}}||_{C^1}$, which in turn is uniformly bounded according to Proposition \ref{proposition: uniform bound on metrics} and \ref{proposition uniform C3 estimate}. 
	
	Similarly, we find a uniform $C>0$, only depending on $(M,g)$ and $\alpha$, such that
	\begin{align} \label{in proposition C0 alpha norm of metric: g-1 bound}
		||g_{\varphi_{ s}}^{-1} ||_{C^{0,\alpha}} \leq C \left( ||g_{\varphi_{ s}} ^{-1} ||_{C^0} + ||\nabla^g g_{\varphi_{ s}} ^{-1} ||_{C^0}   \right).
	\end{align}
	Moreover, there is the following point-wise estimate
	\begin{align}\label{in proposition C0 alpha norm of metric: nabla g-1 bound}
		|\nabla^g g_{\varphi_{ s}} ^{-1} |_g \leq |g_{\varphi_{ s} }|_g^2 \cdot   |\nabla^g g_{\varphi_{ s}}  |_g.
	\end{align}
	Indeed, this inequality  follows immediately by using holomorphic normal coordinates and differentiating the relation 
	\begin{align*}
		g^{\bar j i}_{\varphi_{ s}} \cdot g^{\varphi_{ s}} _{k \bar j} = \delta^{i}_k,
	\end{align*}
	where $g^{\bar j i}_{\varphi_{ s}}$ and $g^{\varphi_{ s}} _{i \bar j}$ are the components of $g^{-1}_{\varphi_{ s}}$ and $g_{\varphi_{ s}}$, respectively. 
	 
	 Thus, we conclude the required uniform bound on $||g_{\varphi_{ s}}^{-1} ||_{C^{0,\alpha}}$ from (\ref{in proposition C0 alpha norm of metric: g-1 bound}) and (\ref{in proposition C0 alpha norm of metric: nabla g-1 bound}), together with  Proposition \ref{proposition: uniform bound on metrics} and \ref{proposition uniform C3 estimate}. 
\end{proof}

The standard Schauder theory for ACyl metrics then implies uniform $C^{2,\alpha}$-bounds for $\varphi_{ s}$.

\begin{prop}[Uniform $C^{2,\alpha}$-bound on $\varphi_{ s}$] \label{prop: C2 alpha bound for varphi}
	Let $1<\varepsilon<2$, $\alpha \in (0,1)$ and suppose $(\varphi_{ s})_{0\leq s\leq 1}$ is a family in $C^{\infty}_{\varepsilon, \, JX}(M)$ solving (\ref{one-parameter family of MA}). Then there exists a constant $C>0$ such that 
	\begin{align*}
		\sup_{s\in [0,1]} ||\varphi_{ s}||_{C^{2,\alpha}} \leq C,
	\end{align*}
where $C$ only depends on $\alpha$, $F\in C^{\infty}_{\varepsilon, \, JX}(M)$ and the geometry of $(M,g)$. 
\end{prop}

\begin{proof}
 Recall that $\Delta_g$ is an asymptotically translation-invariant operator of order $2$, and so the Schauder estimates (Theorem \ref{theorem: ACyl schauder estimates}) apply and yield a constant $C>0$, only depending on $(M,g)$, such that
 \begin{align*}
 ||\varphi_{ s}||_{C^{2,\alpha}} \leq C\left( || \varphi_{ s}||_{C^0} + ||\Delta_{g} \varphi_{ s}||_{C^{0,\alpha}} \right).
 \end{align*}
As $||\Delta_{g} \varphi_{ s}||_{C^{0,\alpha}} $ can be bounded from above in terms of $||g_{\varphi_{ s}}||_{C^{0,\alpha}}$, the claim then follows immediately from Corollary \ref{uniform C0,alpha estimate of metric} and the  uniform bound on $\sup_M |\varphi_{ s}|$ given by Propositions \ref{upper-C^0 estimate for varphi_s from above} and \ref{proposition: lower bound on inf varphi}.
\end{proof}

\subsubsection{Local $C^{k,\alpha}$-estimates} Uniform higher order estimates can be obtained similarly to the compact case considered by Yau \cite{yau1978ricci}. The idea is to use \textit{local} Schauder estimates to conclude higher regularity in a uniform way. Since our manifold $(M,g)$ is non-compact, we require the use of special coordinates in which the metric $g$, and all its derivatives, are uniformly bounded. This is provided by the following
\begin{theorem}\label{Theorem: uniform coordinates for (M,g)}
	Let $(M,g)$ be an $n$-dimensional ACyl K\"ahler manifold and $i(g)>0$ the corresponding injectivity radius. For each $q\in \mathbb{N}_0$, suppose that  $C_q>0$ is a constant such that the curvature tensor $\operatorname{Rm}(g)$ satisfies
	\begin{align*}
	\sup_M |\nabla^q\operatorname{Rm}(g)|_g \leq C_q.
	\end{align*}
	Then there are two constants $r_2>r_1>0$,  depending only on $n, i(g), C_q$, such that for each $x\in M$, there exists a chart $\phi: U \subset \mathbb{C}^n \to M$ satisfying the following properties:
	\begin{itemize}
		\item[(i)] $B_{\mathbb{C}^n}(0,r_1) \subset U \subset B_{\mathbb{C}^n}(0,r_2)$ and $\phi(0)=x $, where $B_{\mathbb{C}^n}(0,r_i)$ denotes the Euclidean ball of radius $r_i$ around the origin.
		\item [(ii)] There exists a constant $C>0$, depending only on $r_1,r_2$ such that the Euclidean metric $g_{\mathbb{C}^n}$ satisfies
			\begin{align*}
			C^{-1} g_{\mathbb{C}^n} \leq \phi^* g \leq C g_{\mathbb{C}^n} \;\;\; \text{ on} \;\; U.
			\end{align*}
		\item[(iii)] For each $l\in \mathbb{N}_0$, there exist constants $A_l>0$, depending only on $l,r_1,r_2$, such that 
			\begin{align*}
			\sup_U \left| \frac{\partial ^{|\mu|+|\nu|}  g_{i\bar j}}{\partial z ^{\mu} \partial \bar z ^{\nu} }  \right| \leq A_l \;\;\; \text{ for all } \; |\mu|+|\nu|\leq l,
			\end{align*}
			where $g_{i\bar j}$ are the components of $g$ in the holomorphic coordinates $(z_1,\dots,z_n)$ induced by $\phi$, and $\mu,\nu$ are multi-indices with $|\mu|= \mu_1+\dots +\mu_n$. 
	\end{itemize}
\end{theorem} 

This theorem follows because the asymptotic cylinder is given explicitly. Similar results   have previously been used to solve complex Monge-Amp\`ere equations on non-compact K\"ahler  manifolds, see for instance \cite{cheng1980existence} and \cite{tian1990complete}. More generally, Theorem \ref{Theorem: uniform coordinates for (M,g)} is also valid for every non-compact K\"ahler manifold of positive injectivity radius and bounded geometry, compare \cite{wu2020invariant}[Theorem 9].

Using the coordinates given by Theorem \ref{Theorem: uniform coordinates for (M,g)}, we can apply the local Schauder theory and conclude estimates on the $C^{k,\alpha}_{\operatorname{loc}}$-norm of $\varphi_s$. This argument is by induction on $k$ starting at $k=3$.

\begin{prop}[Local $C^{3,\alpha}$-bound on $\varphi_{ s}$]\label{proposition uniform local C3alpha bound}
Let $1<\varepsilon<2$, $\alpha \in (0,1)$ and suppose $(\varphi_{ s})_{0\leq s\leq 1}$ is a family in $C^{\infty}_{\varepsilon, \, JX }(M)$ solving (\ref{one-parameter family of MA}). Then there exists a constant $C>0$ such that
	\begin{align*}
		\sup_{s\in [0,1]} ||\varphi_{ s}||_{C^{3,\alpha}_{\operatorname{loc}}} \leq C,
	\end{align*}
	where  $C$ only depends on $\alpha$, $F\in C^{\infty}_{\varepsilon, \, JX}(M)$ and  the geometry of $(M,g)$. 
\end{prop}

\begin{proof}
	As in \cite{conlon2020steady}[Proposition 7.19], we follow the argument given in the compact case \cite{yau1978ricci}. 
	
	We consider $x\in M$ and work in the holomorphic chart $\phi: U \to M$ as in Theorem \ref{Theorem: uniform coordinates for (M,g)}. To simplify notation, we suppress $\phi$ and simply view $U$ as a subset of $M$. 
	The conditions $(ii)$ and $(iii)$ ensure that the Euclidean H\"older norm $||\cdot||_{C^{k,\alpha}(B_x)}$ on the ball $B_x:= B(0,r_1)\subset M$ is uniformly equivalent to $||\cdot ||_{C^{k,\alpha} (B_x, g)}$, the H\"older norm on $B_x$ induced by the ACyl metric $g$. In other words, there exists a constant $C_1>0$, only depending on $k,\alpha$ and the constants in Theorem $\ref{Theorem: uniform coordinates for (M,g)}$, such that 
	\begin{align}\label{in proposition uniform C3alpha bound: equivalence of holder norms}
		C_1 ^{-1} ||\cdot  ||_{C^{k,\alpha} (B_x)} \leq ||\cdot||_{C^{k,\alpha}(B_x,g)} \leq C_1 ||\cdot||_{C^{k,\alpha}(B_x)}. 
	\end{align}
	In particular, the interior Schauder estimates (\cite{gilbarg2015elliptic}[Theorem 6.2, 6.17]) on $B_x$ are valid for the norms $||\cdot||_{C^{k,\alpha}(B_x,g)} $. The goal is to apply these estimates to the equation
		\begin{align}\label{in proposition uniform local Ckalpha bound: equation to whcih we apply schauder}
		\frac{1}{2} \Delta _{g_{\varphi_{ s}}} (\partial_j \varphi_{ s}) = \partial_j \left( sF-\frac{X}{2}(\varphi_{ s})   \right) + (\operatorname{tr}_\omega - \operatorname{tr}_{\omega_{\varphi_{ s} }}) \mathcal{L}_{\partial_j} (\omega),
	\end{align}
where $\partial_j$ denotes the coordinate field $\partial/ \partial z_j$ induced by the chart $\phi$ and $j=1,\dots,n$. 
Observe that (\ref{in proposition uniform local Ckalpha bound: equation to whcih we apply schauder}) is obtained by applying the Lie derivative $\mathcal{L}_{\partial_j}$ to the Monge-Amp\`ere equation $(\ref{one-parameter family of MA})$ and dividing by $\omega^n _{\varphi_{ s}}$. 

	 Recall that in holomorphic coordinates, we have
	\begin{align*}
		\Delta_{g_{\varphi_{ s}}} = g_{\varphi_{ s}}^{\bar  j i } \partial_i \partial_{\bar j},
	\end{align*}
so that applying Schauder requires to bound the coefficients of $\Delta_{g_{\varphi_{ s}}}$ uniformly in $C^{0,\alpha}(B_x)$, i.e. we have to find a constant $D>0$, only depending on $\alpha$, $F$ and the geometry of $(M,g)$, such that 
\begin{align} \label{in proposition uniform C3alpha bound: holder norm of g bar j i bounded}
	||g_{\varphi_{ s}} ^{\bar j i } ||_{C^{0,\alpha}(B_x)} \leq D \;\; \text{ and } \;\; g_{\varphi_{ s}}^{-1} \geq D \, g_{\mathbb{C}^n}.
\end{align}
The first inequality is clear by (\ref{in proposition uniform C3alpha bound: equivalence of holder norms}) together with Corollary \ref{uniform C0,alpha estimate of metric} and the second bound follows immediately from Proposition \ref{proposition: uniform bound on metrics} and condition ($ii$) in Theorem \ref{Theorem: uniform coordinates for (M,g)}.
Thus, interior Schauder estimates provide a constant $C_2>0$, only depending on $n$, $\alpha$ and $D$, such that
\begin{align}
	\begin{split} \label{in prop C3alpha bound: Schauder estimate}
	||\partial_j \varphi_{ s} ||_{C^{2,\alpha}(B_x)}  
	\leq& C_2 \left( ||\Delta_{g_{\varphi_{ s}}} \partial _j \varphi_{ s}||_{C^{0,\alpha}(B_x)} + ||\partial_j \varphi_{ s}||_{C^0(B_x)}   \right) \\
	\leq & C_2\left( ||\Delta_{g_{\varphi_{ s}}} \partial _j \varphi_{ s}||_{C^{0,\alpha}(B_x)}  + C_1 ||\varphi_{ s}||_{C^{2,\alpha}(M,g)}   \right),
	\end{split}
\end{align}
where we used (\ref{in proposition uniform C3alpha bound: equivalence of holder norms}) for the second inequality. We continue to estimate the first term on the right-hand side of (\ref{in prop C3alpha bound: Schauder estimate}) as follows
\begin{align}
	\begin{split} \label{in prop: estimate laplace term in schauder estimate}
	& ||\Delta_{g_{\varphi_{ s}}} \partial _j \varphi_{ s}||_{C^{0,\alpha}(B_x)}\\
	\leq&  ||\varphi_{ s}|| _{C^{2,\alpha}(B_x)} + ||F||_{C^{1,\alpha}(B_x)} 
	+  ||(\operatorname{tr}_\omega -\operatorname{tr}_{\omega_{\varphi_{ s} }}) \mathcal{L}_{\partial_j } (\omega) ||_{C^{0,\alpha}(B_x)}   \\
	\leq&C_1 \left( ||\varphi_{ s}||_{C^{2,\alpha}(M,g)} + ||F||_{C^{1,\alpha}(M,g)}   
	+A \cdot ||g^{-1} - g_{\varphi_{ s}}^{-1}  ||_{C^{0,\alpha}(M,g)} \right),
	\end{split}
\end{align}
for some constant $A>0$ determined by condition ($iii$) of Theorem \ref{Theorem: uniform coordinates for (M,g)}. 
Here, the first inequality is a consequence of (\ref{in proposition uniform local Ckalpha bound: equation to whcih we apply schauder}) and the second one is obtained from (\ref{in proposition uniform C3alpha bound: equivalence of holder norms}).
In combination with (\ref{in prop: estimate laplace term in schauder estimate}), inequality (\ref{in prop C3alpha bound: Schauder estimate}) then becomes
\begin{align} \label{in prop: uniform C3alpha estimate: final estimate on partial varphi}
	||\partial _j \varphi_{ s} ||_{C^{2,\alpha}(B_x)} \leq C_3
\end{align}
for some constant $C_3>0$ which only depends on $C_1$, $C_2$, $A$, $F$ and the uniform bounds on $||\varphi_{ s}||_{C^{2,\alpha}}$ and $||g_{\varphi_{ s}} ^{-1}||_{C^{0,\alpha}}$ given by Proposition \ref{prop: C2 alpha bound for varphi} and Corollary \ref{uniform C0,alpha estimate of metric}, respectively. 

To conclude the proof, we point out that the same arguments for (\ref{in prop: uniform C3alpha estimate: final estimate on partial varphi}) also yield  
\begin{align*}
	||\partial_{\bar j} \varphi_{ s}||_{C^{2,\alpha}(B_x)} \leq C_3,
\end{align*}
and hence
\begin{align*}
	||\varphi_{ s}||_{C^{3,\alpha}(B_x)}  &\leq \sum_{j=1}^n ||\partial _j \varphi_{ s} ||_{C^{2,\alpha}(B_x)} + ||\partial _{\bar j} \varphi_{ s} ||_{C^{2,\alpha}(B_x)} + ||\varphi_{ s}|| _{C^{0}(B_x)}, \\
	&\leq 2n C_3 + C_1 ||\varphi_{ s}||_{C^{2,\alpha}(M,g)}\\
	&\leq C_4
\end{align*}
with $C_4>0$ only depending on $n$, $C_1$, $C_3$ and the uniform bound on $||\varphi_{ s}||_{C^{2,\alpha}}$. In particular, the constant $C_4$ is independent of both $x\in M$ and $s\in [0,1]$, so that the proposition then follows. 

\end{proof}

The standard bootstrapping argument then leads to uniform $C^{k,\alpha}$-estimates. 

\begin{prop}[Local $C^{k,\alpha}$-bounds on $\varphi_{ s}$]\label{proposition uniform local Ckalpha bound}
Let $1<\varepsilon<2$, $\alpha \in (0,1)$, $k\in \mathbb{N}_{\geq1}$ and suppose $(\varphi_{ s})_{0\leq s\leq 1}$ is a family in $C^{\infty}_{\varepsilon, \, JX}(M)$ solving (\ref{one-parameter family of MA}). Then there exists a constant $C>0$ such that
	\begin{align} \label{in proposition uniform local Ckalpha bound: the actual bound}
		\sup_{s\in [0,1]} ||\varphi_{ s}||_{C^{k+2,\alpha}_{\operatorname{loc}}} \leq C ,
	\end{align}
	where $C$ only depends on $k$, $\alpha$, $F\in C^{\infty}_{\varepsilon, \, JX}(M)$ and the geometry of $(M,g)$. 
\end{prop}

\begin{proof}
	As in \cite{conlon2020steady}[Proposition 7.19], the proof is by induction on $k\geq1$, with the $k=1$ case being settled by Proposition \ref{proposition uniform local C3alpha bound}. Thus, we consider $k\geq 2$ and can assume that the statement holds for $k-1$, i.e. that there is a $C_{k-1}>0$, only depending on $k$, $\alpha$, $F$ and the geometry of $(M,g)$, such that
	\begin{align}\label{in proposition uniform local Ckalpha bound: the induction hypothesis}
     || \varphi_{ s} ||_{C^{k+1,\alpha}_{\operatorname{loc}}} \leq C_{k-1}. 
	\end{align}
	Using the same notation as in the previous proof, we work near a given $x\in M$ in the chart $\phi: U \to M$ given by Theorem \ref{Theorem: uniform coordinates for (M,g)}. Because of (\ref{in proposition uniform C3alpha bound: equivalence of holder norms}), it suffices to show (\ref{in proposition uniform local Ckalpha bound: the actual bound}) for the Euclidean ball $B_x:= B(0,r_1)$ and the Euclidean H\"older norm $||\cdot||_{C^{k,\alpha}(B_x)}$.
	
	This time, we aim at applying interior Schauder estimates (of higher order) to equation (\ref{in proposition uniform local Ckalpha bound: equation to whcih we apply schauder}), for which we require a constant $D_{k-1}$, depending only on $k$, $\alpha$, $F$ and the geometry of $(M,g)$, such that
	\begin{align} \label{in proposition uniform Ckalpha bound: holder norm of g bar j i bounded}
		||g_{\varphi_{ s}} ^{\bar j i } ||_{C^{k-1,\alpha}(B_x)} \leq D_{k-1} \;\; \text{ and } \;\; g_{\varphi_{ s}}^{-1} \geq D_{k-1}\, g_{\mathbb{C}}.
	\end{align}
	The second inequality is again clear by Proposition \ref{proposition: uniform bound on metrics} and condition ($ii$) in Theorem \ref{Theorem: uniform coordinates for (M,g)}, and for the first, recall that 
	\begin{align*}
		g^{\varphi_{ s}} _{i \bar j } = g_{i\bar j} + \partial _i \partial_{\bar j} \varphi_{ s}.
	\end{align*}
	Together with condition $(iii)$ in Theorem \ref{Theorem: uniform coordinates for (M,g)}, we obtain
	\begin{align*}
	\nonumber	||g_{i\bar j }^{\varphi_{ s}}||_{C^{k-1,\alpha}(B_x)} &\leq ||\partial \partialb \varphi_{ s} ||_{C^{k-1,\alpha}} (B_x) + A_{k-1} \\
		&\leq ||\varphi_{ s}||_{C^{k+1,\alpha}(B_x)} + A_{k-1} \\
	\nonumber	&\leq C_{k-1} + A_{k-1},
	\end{align*}
	where we used the induction hypothesis (\ref{in proposition uniform local Ckalpha bound: the induction hypothesis}) in the last line. Consequently, the entries of the inverse matrix can be bounded as well since there exists a $C_0>0$, depending only on the uniform bound on $||g^{-1}_{\varphi_{ s}}||_{C^0(M)}$ from Proposition \ref{proposition: uniform bound on metrics}, such that
	\begin{align*}
	||g_{\varphi_{ s}} ^{\bar j i } ||_{C^{k-1,\alpha}(B_x)}\leq C_0 ||g^{\varphi_{ s}} _{\bar i j }  ||_{C^{k-1,\alpha}(B_x)} .
	\end{align*}
	Note that this follows by differentiating the identity
	\begin{align*}
		g_{\varphi_{ s}}^{\bar j i} g^{\varphi_{ s}}_{l \bar j} = \delta^i_l
	\end{align*}
	and using the fact that for functions $u$ with $\inf u>0$, one has
	\begin{align*}
		||u||_{C^{0,\alpha}} \leq (\inf u )^{-1} \left( 1+ ||u||_{C^{0,\alpha}} (\inf u)^{-1}     \right).
	\end{align*}
	Thus, (\ref{in proposition uniform Ckalpha bound: holder norm of g bar j i bounded}) holds if $D_{k-1}:= C_0(C_k+A_{k-1})$. 
	Then the interior Schauder estimates \cite{gilbarg2015elliptic}[Theorem 6.17] provide a constant $E_{k-1}>0$,  depending only on $n$, $k$, $\alpha$ and $D_{k-1}$, such that 
	\begin{align*}
		&||\partial _j \varphi_{ s} ||_{C^{k+1,\alpha}(B_x)} \\ \leq & E_{k-1} \left(||\Delta_{g_{\varphi_{ s}}}  (\partial_j \varphi_{ s})  ||_{C^{k-1,\alpha}(B_x)} + ||\partial _j \varphi_{ s}||_{C^0(B_x)}  \right)\\
		\leq& E_{k-1} \left( ||\varphi_{ s}||_{C^{k+1,\alpha}(B_x)}  
		+ ||F||_{C^{k,\alpha}(B_x)}  + ||(\operatorname{tr}_\omega -\operatorname{tr}_{\omega_{\varphi_{ s} }}) \mathcal{L}_{\partial_j } (\omega) ||_{C^{k-1,\alpha}(B_x)}  \right)\\
		\leq& E_{k-1} \left( ||\varphi_{ s}||_{C^{k+1,\alpha}(B_x)}  
		+ ||F||_{C^{k,\alpha}(B_x)} + A_{k-1} ||g^{-1}-g^{-1}_{\varphi_{ s}}||_{C^{k-1,\alpha}(B_x)} \right) 
	\end{align*}
	where  (\ref{in proposition uniform local Ckalpha bound: equation to whcih we apply schauder}) implies the second inequality, and for the third one, we used the bounds in $(iii)$ of Theorem \ref{Theorem: uniform coordinates for (M,g)}. Hence, we conclude from this, together with the induction hypothesis (\ref{in proposition uniform local Ckalpha bound: the induction hypothesis}) and (\ref{in proposition uniform Ckalpha bound: holder norm of g bar j i bounded}), that
	\begin{align*}
		||\partial _j \varphi_{ s} ||_{C^{k+1,\alpha}(B_x)} \leq C_k
	\end{align*}
	for some $C_k>0$ only depending on $E_{k-1}$, $C_{k-1}$, $F$ and the constants in Theorem \ref{Theorem: uniform coordinates for (M,g)}. As in the previous proof,  we finally arrive at 
	\begin{align*}
		||\varphi_{ s}||_{C^{k+2,\alpha} (B_x)} 
	&	\leq  \sum_{j=1}^n ||\partial _j \varphi_{ s} ||_{C^{k+1,\alpha}(B_x)} +||\partial _{\bar j} \varphi_{ s} ||_{C^{k+1,\alpha}(B_x)} + ||\varphi_{ s}||_{C^{0}(B_x)} \\
	&\leq 3n C_k,
	\end{align*}
 as required. 
\end{proof}

\subsubsection{Weighted $C^{k,\alpha}$-estimates} 

Recall from Propositions \ref{upper-C^0 estimate for varphi_s from above} and \ref{proposition lower weighted bound for varepsilon0} that $|\varphi_{ s}|$ is uniformly bounded from above by $e^{-\varepsilon_0t}$ for some $0<\varepsilon_0\ll1$. 
First, we will see that the $C^{k,\alpha}_{\varepsilon_0}$-norms of $\varphi_{ s}$ are also uniformly bounded and, in a second step, we explain how to improve the decay from $\varepsilon_0$ to $\varepsilon$.

\begin{prop}[Weighted $C^{k,\alpha}$-bounds on $\varphi_{ s}$] \label{proposition weighted Ckalpha estimate for epsilon 0}
	Let $1<\varepsilon<2$, $\alpha\in (0,1)$, $k\in \mathbb{N}_{0}$ and suppose  $(\varphi_{ s})_{0\leq s\leq 1}$ is a family in $C^{\infty}_{\varepsilon, \, JX}(M)$ solving (\ref{one-parameter family of MA}). For the constant $0<\varepsilon_0<1$ given by Proposition \ref{proposition lower weighted bound for varepsilon0}, exists a  $C>0$ such that
	\begin{align*}
		\sup_{s\in [0,1]} ||e^{\varepsilon_0 t} \varphi_{ s}||_{C^{k,\alpha}} \leq C,
	\end{align*}
	where $C$ only depends on $\varepsilon_0$, $k$, $\alpha$, $F\in C^{\infty}_{\varepsilon, \, JX}(M)$ and the geometry of $(M,g)$.
\end{prop}

\begin{proof}
	We follow the argument given in \cite{conlon2020steady}[Proposition 7.22]. For $\tau\in [0,1]$, consider the function 
	\begin{align}\label{in proposition weighted Ckalpha estimate for epsilon 0: definition of H}
		H(\tau):= \log \frac{\left(  \omega+ i \partial \partialb (\tau\cdot \varphi_{ s}) \right)^n}{\omega^n},
	\end{align}
	so that
	\begin{align*}
		H'(\tau) = \frac{1}{2} \Delta_{g_{\tau \varphi_{ s}}} (\varphi_{ s}),
	\end{align*} 
	where $g_{\tau \varphi_{ s}}$ denotes the metric with K\"ahler form $\omega + i \partial \partialb (\tau\varphi_{ s})$. By using (\ref{one-parameter family of MA}) and $H(0)=0$, we can write
	\begin{align} \label{in proposition weighted Ckalpha estimate for epsilon 0: global version of linearize MA equation}
		sF-\frac{X}{2}(\varphi_{ s}) = H(1) = \int_0 ^1 H'(\tau)d\tau = \frac{1}{2} \int_0 ^1 \Delta_{g_{\tau \varphi_{ s}}}(\varphi_{ s}) d\tau.
	\end{align}
	The goal is to apply local Schauder estimates to this differential equation. For any  $x\in M$, let $\phi:U \to M$ be the holomorphic chart  with $\phi(0)=x$ given by Theorem \ref{Theorem: uniform coordinates for (M,g)}. Then (\ref{in proposition weighted Ckalpha estimate for epsilon 0: global version of linearize MA equation}) becomes
	\begin{align*}
		sF = \left( \int_0 ^1  g^{\bar j  i} _{\tau \varphi_{ s}}  d\tau \right) \partial_i \partial_{\bar j} \varphi_{ s} + \frac{X}{2} (\varphi_{ s}) =: a^{\bar j i} \partial_i \partial_{\bar j} \varphi_{ s} + b_j \partial_j \varphi_{ s},
	\end{align*}
	where we use Einstein's sum convention, and the fact that $X$ is real-holomorphic as well as $JX(\varphi_{ s})=0$. 
	
	Let $k\geq 0$ be an integer and $\alpha \in (0,1)$.  Recall that by conditions (ii), (iii) of Theorem \ref{Theorem: uniform coordinates for (M,g)} and Proposition \ref{proposition uniform local Ckalpha bound}, there exists a constant $C_1>0$ such that 
	\begin{align*}
		a^{\bar j i}\geq C_1^{-1} \delta_{i\bar j}, \;\;\;  \text{and } \;\; ||a^{\bar j i }||_{C^{k,\alpha}(B_x,g)} \leq C_1,
	\end{align*}
	where $B_x$ is the holomorphic ball of radius $r_1$ around $x$ and $||\cdot||_{C^{k,\alpha}(B_x,g)}$  the H\"older norm on $B_x$ induced by the restriction of $g$.
	Moreover, we can arrange that $||b_j||_{C^{k,\alpha} (B_x,g)}\leq C_1$ since $X$ and all its covariant derivatives (w.r.t. $g$) are uniformly bounded. Recall from (\ref{in proposition uniform C3alpha bound: equivalence of holder norms}) that the norms on $B_x$ induced by $g$ are uniformly equivalent to the Euclidean H\"older norms, so that  interior Schauder estimates (\cite{gilbarg2015elliptic}[Theorem 6.17]) can be applied. Hence, there exists  a constant $C_2>0$, depending only on $n$, $k$, $\alpha$ and $C_1$, such that 
	\begin{equation} \label{in proposition weighted Ckalpha varepsilon_0: local Ckalpha schauder estimates for varphi}
	\begin{aligned}
		||\varphi_s||_{C^{k+2,\alpha}(B_x,g) }&\leq C_2 \left( ||\varphi_s||_{C^0 (B_x)} + ||F||_{C^{k,\alpha}(B_x,g)}   \right) \\
		&\leq C_2\left( C_3 + C_3 ||F||_{C^{k,\alpha}_{\varepsilon_0}(M,g)  } \right)e^{-\varepsilon_0 t (x)},
	\end{aligned} 
		\end{equation} 
	for some $C_3>0$ only depending on the radius $r_1$ of the ball $B_x$ and the bounds from Propositions \ref{upper-C^0 estimate for varphi_s from above} and \ref{proposition lower weighted bound for varepsilon0}.  Note that in the last inequality, we also used that the function $t$ is uniformly equivalent to the distance function of $(M,g)$ to some fixed point. 
	
	As the constants in (\ref{in proposition weighted Ckalpha varepsilon_0: local Ckalpha schauder estimates for varphi}) are independent of the considered point $x\in M$, we  conclude the desired estimate for $||e^{\varepsilon_0 t}\varphi_{ s}||_{C^{k,\alpha}}$ as follows. Let $0\leq l \leq k+1$ and notice that  (\ref{in proposition weighted Ckalpha varepsilon_0: local Ckalpha schauder estimates for varphi}) implies 
	\begin{align*}
		 |(\nabla^g )^l \varphi_{ s}|_g (x) \leq ||\varphi_s|| _{C^{k,\alpha}(B_x,g)} &\leq C_2 C_3 \left(  1+ ||F||_{C^{k,\alpha}_{\varepsilon_0}(M,g)}\right) e^{-\varepsilon_0 t(x)} \\
		 &=: Ce^{-\varepsilon_0 t(x)}
		\end{align*}
		holds for all $x\in M$, or equivalently,
		\begin{align*}
			||e^{\varepsilon_0 t} \varphi_{ s}||_{C^{k+1}(M,g)} \leq C.
		\end{align*}
This finishes the proof because the inclusion $C^{k+1}_{\varepsilon_0}(M) \subset C^{k,\alpha}_{\varepsilon_0}(M)$ is continuous. 

\end{proof}

It remains to improve the uniform decay rate of $\varphi_s$ from $e^{-\varepsilon_0t}$ to $e^{-\varepsilon t}$, which is achieved in the next

\begin{prop}[Improved weighted $C^{k,\alpha}$-bounds on $\varphi_{ s}$] \label{prop: final uniform bound on varepsilon weighted Ckalpha norm}
Let $1<\varepsilon<2$, $\alpha\in (0,1)$, $k \in \mathbb{N}_0$ and suppose $(\varphi_{ s})_{0\leq s\leq1}$ is a family in $C^{\infty}_{\varepsilon, \, JX}(M)$ solving (\ref{one-parameter family of MA}). Then there exists a constant $C>0$ such that 
	\begin{align*}
		\sup_{s\in [0,1]} ||e^{\varepsilon t}\varphi_{ s}||_{C^{k,\alpha}} \leq C,
	\end{align*}
	where $C$ only depends on  $k$, $\alpha$, $F\in C^{\infty}_{\varepsilon, \, JX}(M)$ and the geometry of $(M,g)$.
\end{prop}

\begin{proof}
	This improvement of the rate based on  \cite{conlon2020steady}[p. 63]. We begin by noting that $H(\tau)$ define by (\ref{in proposition weighted Ckalpha estimate for epsilon 0: definition of H}) satisfies
	\begin{align*}
		H''(\tau)= - |\partial \partialb \varphi_{ s}|^2 _{g_{\tau \varphi_{ s}}},
	\end{align*}
		so that we can write
	\begin{equation}\label{in proposition improvement of weighted bound: linearized MA equation}
	\begin{aligned}
		sF + \int _0 ^1 \int _0 ^\tau |\partial \partialb \varphi_{ s}|^2 _{g_{\tau \varphi_{ s}}}d\sigma \,d\tau &= sF + H'(0)-H(1) + H(0)  \\
		&=\frac{1}{2} \left( \Delta_g + X \right)(\varphi_{ s}),
	\end{aligned}
\end{equation}
where we used (\ref{one-parameter family of MA}) for the second inequality. From (\ref{in proposition improvement of weighted bound: linearized MA equation}) and Proposition \ref{proposition weighted Ckalpha estimate for epsilon 0}, we conclude that there exists a uniform constant $C>0$ such that
\begin{align}\label{in proposition improvement of weighted decay: iteration start}
	\left(\Delta_g + X  \right)(\varphi_{ s}) \leq C e^{-\varepsilon_1 t} \;\; \text{ with } \;\; \varepsilon_1:= \min \{2\varepsilon_0, \varepsilon   \}>\varepsilon_0.
\end{align}
Starting from this equation, we can obtain a uniform lower bound on $e^{\varepsilon_1 t} \varphi_{ s}$ by using the maximum principle and arguing as in Proposition \ref{upper-C^0 estimate for varphi_s from above}. Let $v\in C^{\infty}_{\varepsilon_1}(M)$ be the unique solution to 
\begin{align*}
	(\Delta_g + X)(v) = C e^{-\varepsilon_1 t},
\end{align*}
so that we have
\begin{align*}
		(\Delta_g + X)(\varphi_{ s} - v) \leq 0 \;\; \text{ on } \;\; M.
\end{align*}
Thus, the maximum principle implies
\begin{align} \label{in proposition final varepsilon boun: after applying max principle}
	\varphi_{ s} -v \geq  \lim_{t\to \infty} (\varphi_{ s} -v )=0 \;\; \text{ on } \;\; M
\end{align}
which is  a uniform weighted lower bound on $\varphi_{ s}$ since $v\in C^{\infty}_{\varepsilon_1}(M)$ only depends on $\varepsilon_1$, $C$ and $(M,g)$. Combining (\ref{in proposition final varepsilon boun: after applying max principle}) with the upper bound in Proposition \ref{upper-C^0 estimate for varphi_s from above},  the term $||e^{\varepsilon_1 t}\varphi_s||_{C^0}$ is uniformly bound from above.

%Indeed, similar to the proof of Lemma \ref{lemma: construction of compactum}), we can choose a  compact set $K\subset M$, which only depends on $\varepsilon_0$ and $(M,g)$, such that
%\begin{align*}
%	(\Delta_g+ X)\left(e^{- \varepsilon_0 f}\right) \leq \frac{\varepsilon_0}{2}  e^{-\varepsilon_0 f}  (\varepsilon_0-1) <0 \;\; \text{ on  } \;\; M\setminus K.
%\end{align*}
%This is possible because
%\begin{align*}
%	e^{\varepsilon_0 f} \left( \Delta_g + X  \right) \left( e^{-\varepsilon_0 f}  \right) = \varepsilon_0\left( (\varepsilon_0-1) |\nabla^g f|_g^2  - \Delta_g f  \right) \to \varepsilon_0 (\varepsilon_0-1)
%\end{align*}
%as $t\to \infty$. Since $f-2t$ tends to zero at infinity, we can choose $A>0$ sufficiently large such that
%\begin{align*}
%	\left(\Delta_g+X  \right)\left(\varphi_{ s}  + A e^{-\varepsilon_0 f} \right ) \leq 0 \;\; \text{ on } \;\; M\setminus K.
%\end{align*}
%Consequently, the maximum principle implies that 
%\begin{align}
%	\begin{split} \label{in proposition final varepsilon boun: after applying max principle}
%	\varphi_{ s}  &\geq \min \left\{ 0, \min_{\partial K} \left( \varphi_{ s} + A e^{-\varepsilon_0 f} \right) \right \} -  Ae^{-\varepsilon_0 f}\\
%	&\geq - Ae^{-\varepsilon_0 f}
%	\end{split}
 %\end{align}
%on $M\setminus K$, where we used the lower bound on $\varphi_{ s}$ (Proposition \ref{proposition: lower bound on inf varphi}) to choose $A>0$ even larger so that
%\begin{align*}
%	\varphi_{ s} + Ae^{-\varepsilon_0 f} \geq 0 \;\; \text{ on } \;\; K.
%\end{align*}

The next step is to prove that for each $k\in \mathbb{N}_0$, $\alpha \in (0,1)$, there exists a uniform constant $C>0$ such that
\begin{align}\label{in proposition improvement of weighted bound: iteration end}
	||e^{\varepsilon_1 t } \varphi_{ s}||_{C^{k,\alpha}} \leq C.
\end{align}
Indeed, the same argument as in Proposition \ref{proposition weighted Ckalpha estimate for epsilon 0} goes through verbatim, starting this time from the uniform bound on $||e^{\varepsilon_1 t}\varphi_s||_{C^0}$ instead of merely $||e^{\varepsilon_0 t}\varphi_s||_{C^0}$. 
Hence, we improved the uniform decay from $\varepsilon_0 $ to $\varepsilon_1$. 

If $\varepsilon_1= \varepsilon$, we are done, so we  assume  $\varepsilon_1= 2\varepsilon_0<\varepsilon$.
Notice that (\ref{in proposition improvement of weighted bound: iteration end}) and (\ref{in proposition improvement of weighted bound: linearized MA equation}) can then be used to further improve the uniform decay of $(\Delta_g+X)\varphi_{ s}$ in (\ref{in proposition improvement of weighted decay: iteration start}) to 
\begin{align*}
	\varepsilon_2:= \min\{2\varepsilon_1, \varepsilon\} >\varepsilon_1=2\varepsilon_0
\end{align*}
so that repeating the entire argument then gives a uniform bound on $||e^{\varepsilon_2 t } \varphi_{ s}||_{C^{k,\alpha}}$.

After  iterating this process a bounded number of times, we finally conclude the required uniform estimate on $||e^{\varepsilon t} \varphi_{ s}||_{C^{k,\alpha}}$. 
\end{proof}

Since the previous Proposition is precisely the content of Theorem \ref{theorem APRIORI estimates and regularity}, the only statement left to show is the regularity result in Proposition \ref{proposition regularity}. 

\begin{proof}[Proof of Proposition \ref{proposition regularity}]
	Let $F\in C^{\infty}_{\varepsilon, \, JX}(M)$ for some $1<\varepsilon<2$ and suppose that $\varphi \in C^{3,\alpha}_{\varepsilon', \, JX}(M)$ solves 
	\begin{align}\label{in prop regularity: the MA equation}
		(\omega+ i \partial \partialb \varphi) ^n = e^{F-\frac{X}{2}(\varphi)} \omega ^n 
	\end{align}
with $0<\varepsilon'\leq \varepsilon$. We have essentially seen all required arguments in Propositions \ref{proposition uniform local Ckalpha bound}, \ref{proposition weighted Ckalpha estimate for epsilon 0} and \ref{prop: final uniform bound on varepsilon weighted Ckalpha norm}, but the difference is that we now only require \textit{qualitative} information on the solution $\varphi$, i.e. all the constants below a priori \textit{do} depend on $\varphi$. 

First, we improve the regularity and claim that $\varphi \in C^{k,\alpha}_{\operatorname{loc}}(M)$ for each integer $k\geq 3$ and $\alpha \in (0,1)$. As in Proposition \ref{proposition uniform local Ckalpha bound}, we work around some $x\in M$ in the holomorphic chart $\phi: B_x= B(0,r_1) \to M$ given by Theorem \ref{Theorem: uniform coordinates for (M,g)}. Differentiating (\ref{in prop regularity: the MA equation}) in direction of $\partial_j = \partial / \partial z_j$, we obtain
\begin{align} \label{in prop regularity: local bootstrapping}
	\frac{1}{2} \Delta_{g_\varphi} (\partial_j \varphi) = \partial_j \left( F-\frac{X}{2}(\varphi)  \right) + \left( \operatorname{tr}_\omega-\operatorname{tr}_{\omega_\varphi} \right) \mathcal{L}_{\partial_j} (\omega)
\end{align}
Then we notice that the coefficients $g_\varphi^{\bar j i }$ of $\Delta_{g_\varphi}$, as well as the right-hand side of (\ref{in prop regularity: local bootstrapping}), are in $C^{1,\alpha}(B_x)$, so that the local regularity for elliptic equations (\cite{gilbarg2015elliptic}[Theorem 6.17])  implies $\partial_j \varphi  \in C^{3,\alpha}(B_x)$ for all $j=1,\dots,n$. Similarly, one can show that each $\partial_{\bar j } \varphi $ is also in  $C^{3,\alpha}(B_x)$, implying $\varphi \in C^{4,\alpha}(B_x)$. Hence, the standard bootstrapping gives $\varphi \in C^{k,\alpha}(B_x)$ for any given $k\in \mathbb{N}$ and $\alpha\in (0,1)$. Indeed, using $\varphi \in C^{4,\alpha}(B_x)$, we observe that the coefficients $g_{\varphi}^{\bar j i }$ and the right-hand side of (\ref{in prop regularity: local bootstrapping}) are in $C^{2,\alpha}(B_x)$, so that $\partial_j\varphi, \partial_{\bar j} \varphi \in C^{4,\alpha}(B_x)$. This implies $\varphi$ is $C^{5,\alpha}(B_x)$, and so forth, until we finally arrive that $\varphi \in C^{k,\alpha}(B_x)$, as claimed.

In the second step, we show that $\varphi \in C^{k,\alpha}_{\varepsilon'}(M)$, i.e. that higher order derivatives of $\varphi$ decay as $e^{-\varepsilon't}$. In the same notation as in Proposition \ref{proposition weighted Ckalpha estimate for epsilon 0}, consider the following equation on $B_x$:
\begin{align*}
	F = \left( \int_0 ^1  g^{\bar j  i} _{\tau \varphi}  d\tau \right) \partial_i \partial_{\bar j} \varphi + \frac{X}{2} (\varphi) =: a^{\bar j i} \partial_i \partial_{\bar j} \varphi + b_j \partial_j \varphi.
\end{align*}
Then, by local Schauder estimates, there exists a constant $C>0$, depending on $k$, $\alpha$, $||g_{\varphi}||_{C^{k,\alpha}(M)}$ and $||X||_{C^{k,\alpha}(M)}$, such that
\begin{align*}
	||\varphi||_{C^{k+2,\alpha}(B_x)} \leq C\left( ||\varphi||_{C^{0}(B_x)} + ||F||_{C^{k,\alpha}(B_x)}  \right).
\end{align*}
Since $\varphi= O(e^{-\varepsilon ' t})$, $F\in C^{\infty}_{\varepsilon}(M)$ and because the Euclidean H\"older norms on $B_x$ are uniformly equivalent to the ones induced by the restriction of $g$, we conclude from this equation that $\varphi \in C^{k,\alpha}_{\varepsilon'}(M)$ by the same argument used in Proposition \ref{proposition weighted Ckalpha estimate for epsilon 0}. 

Finally, it remains to show $\varphi \in C^{k,\alpha}_{\varepsilon}(M)$, i.e. to improve the decay rate from $\varepsilon'$ to $\varepsilon$. Similarly to Proposition \ref{prop: final uniform bound on varepsilon weighted Ckalpha norm}, consider the equation 
\begin{align}\label{in prop: regularity last equation for improving rate}
	F + \int _0 ^1 \int _0 ^\tau |\partial \partialb \varphi|^2 _{g_{\tau \varphi}}d\sigma \,d\tau =\frac{1}{2} \left( \Delta_g + X \right)(\varphi),
\end{align}
and deduce that 
\begin{align*}
	\left( \Delta_g + X \right)(\varphi) \in C^{\infty}_{\varepsilon_1}(M) \;\; \text{ for } \;\; \varepsilon_1= \min \{ 2\varepsilon', \varepsilon \}. 
\end{align*}
Applying Theorem \ref{theorem ACyl drift laplace is iso}, we find a unique $v\in C^{\infty}_{\varepsilon_1}(M)$ such that
\begin{align*}
	(\Delta_g + X) (v) = (\Delta_g + X)(\varphi),
\end{align*}
but then the maximum principle implies $\varphi=v \in C^{\infty}_{\varepsilon_1}(M)$. If $\varepsilon_1<\varepsilon$, iterate this process starting from (\ref{in prop: regularity last equation for improving rate}) a bounded number of times, and conclude that $\varphi \in C^{\infty}_\varepsilon (M)$, finishing the proof.

\end{proof}

		\bibliography{ms02}
		\bibliographystyle{amsalpha}
\end{document}